\documentclass[11pt,reqno,a4paper]{amsart}
\usepackage{amssymb,amscd,amsmath}
\usepackage{pspicture}
\usepackage{graphicx}

\usepackage{mathptmx}

\usepackage[matrix,arrow,curve,frame]{xy}    

\xymatrixcolsep{1.9pc}                          
\xymatrixrowsep{1.9pc}
\newdir{ >}{{}*!/-5pt/\dir{>}}                  

 \calclayout

\makeatletter

\renewcommand{\subsection}{\@startsection{subsection}{1}{0pt}{-3.25ex plus -1ex minus-.2ex}{1.5ex plus.2ex}{\normalfont\it}}
\renewcommand{\section}{\@startsection{section}{1}{\parindent}{3.5ex plus 1ex minus .2ex}{2.3ex plus.2ex}{\sc}}

\renewcommand{\phi}{\varphi}
\renewcommand{\leq}{\leqslant}
\renewcommand{\geq}{\geqslant}
\renewcommand{\epsilon}{\varepsilon}

\renewcommand{\kappa}{\varkappa}

\DeclareMathOperator{\Ext}{Ext}

\DeclareMathOperator{\Hom}{Hom} 
 
 \DeclareMathOperator{\id}{id}

\DeclareMathOperator{\chr}{char} 
 
\DeclareMathOperator{\coker}{Coker} \DeclareMathOperator{\nis}{nis}

\newcommand{\cc}{\mathcal}
\newcommand{\bb}{\mathbb}

\newcommand{\op}{{\textrm{\rm op}}}

\newtheorem{thm}{Theorem}[section]
\newtheorem{prop}[thm]{Proposition}

\newtheorem{cor}[thm]{Corollary}
\newtheorem{lem}[thm]{Lemma}

\newtheorem{rem}[thm]{Remark}

\newtheorem{defs}[thm]{Definition}
\newtheorem{notn}[thm]{Notation}
\newtheorem{constr}[thm]{Construction}
\newtheorem{claim}[thm]{Claim}

\makeatother

\begin{document}

\footskip30pt


\title{ Homotopy invariant presheaves with framed transfers}
\author{Grigory Garkusha}
\address{Department of Mathematics, Swansea University, Singleton Park, Swansea SA2 8PP, United Kingdom}
\email{g.garkusha@swansea.ac.uk}

\author{Ivan Panin}
\address{St. Petersburg Branch of V. A. Steklov Mathematical Institute,
Fontanka 27, 191023 St. Petersburg, Russia}

\email{paniniv@gmail.com}

\address{Institute for Advanced Study, Einstein Drive, Princeton, NJ, 08540, USA}

\email{panin@math.ias.edu}

\thanks{Theorem \ref{thm1} is proved thanks to the support of the
Russian Science Foundation (grant no. 14-21-00035).}

\thanks{The second author thanks the
Institute for Advanced Study for the support and for the kind
hospitality during his visit in the Second Term of 2014--2015.}

\begin{abstract}
The category of framed correspondences $Fr_*(k)$, framed presheaves
and framed shea\-ves were invented by Voevodsky in his unpublished
notes~\cite{Voe2}. Based on the theory, framed motives are
introduced and studied in~\cite{GP3}. The main aim of this paper is
to prove that for any $\bb A^1$-invariant quasi-stable radditive
framed presheaf of Abelian groups $\cc F$, the associated Nisnevich
sheaf $\cc F_{\nis}$ is $\bb A^1$-invariant whenever the base field
$k$ is infinite of characteristic different from 2. Moreover, if the
base field $k$ is infinite perfect of characteristic different from
2, then every $\bb A^1$-invariant quasi-stable Nisnevich framed
sheaf of Abelian groups is strictly $\bb A^1$-invariant and
quasi-stable. Furthermore, the same statements are true in
characteristic 2 if we also assume that the $\bb A^1$-invariant
quasi-stable radditive framed presheaf of Abelian groups $\cc F$ is
a presheaf of $\bb Z[1/2]$-modules.

This result and the paper are inspired by Voevodsky's
paper~\cite{Voe1}.
\end{abstract}

\keywords{Motivic homotopy theory, framed presheaves}

\subjclass[2010]{14F42, 14F05}

\maketitle

\thispagestyle{empty} \pagestyle{plain}

\newdir{ >}{{}*!/-6pt/@{>}} 

\tableofcontents

\section{Introduction}
The main goal of the Voevodsky theory on framed correspondences (see~\cite[Introduction]{Voe2})
is to suggest a new approach to the stable motivic homotopy theory $SH(k)$ more
amenable to explicit calculations. Inspired by the Voevodsky theory~\cite{Voe2},
the authors introduce and develop the theory of framed motives of algebraic varieties in~\cite{GP3}.
One of the main purposes of the theory~\cite{GP3} is to give {\it explicit\/} computations
of the suspension spectra/bispectra of smooth algebraic varieties in terms of {\it explicit\/}
motivic spaces with framed correspondences. One of the {\it key steps\/} in the computations of~\cite{GP3} is
Theorem~\ref{thm1} proved in this paper. Theorem~\ref{thm1} is also inspired by Voevodsky's
paper~\cite{Voe1}.

The main result, Theorem~\ref{thm1}, can be reformulated in terms of $\bb ZF_*$-presheaves of
abelian groups on smooth algebraic varieties $Sm/k$. Recall that $\bb ZF_*(k)$ is
defined in~\cite{GP3} as an additive category whose objects are
those of $Sm/k$ and Hom-groups are defined as follows. We set for
every $n\geq 0$ and $X,Y\in Sm/k$,
   $$\bb ZF_n(X,Y)=\bb ZFr_n(X,Y)/\langle Z_1\sqcup Z_2-Z_1-Z_2\rangle,$$
where $Z_1,Z_2$ are supports of correspondences. In other words,
$\bb ZF_n(X,Y)$ is a free abelian group generated by the framed
correspondences of level $n$ with connected supports. We then set
   $$\Hom_{\bb ZF_*(k)}(X,Y)=\bigoplus_{n\geq 0}\bb ZF_n(X,Y).$$
The canonical morphisms $Fr_*(X,Y) \to \Hom_{\bb ZF_*(k)}(X,Y)$
define a functor $R: Fr_*(k) \to \bb ZF_*(k)$, which is the identity
on objects. For any $\bb A^1$-invariant quasi-stable $\bb
ZF_*$-presheaf of abelian groups $\cc F$ the functor $\cc F \circ R:
Fr_*(k)^{\op} \to Ab$ is $\bb A^1$-invariant quasi-stable radditive
framed presheaf of abelian groups.

By definition, a $Fr_*$-presheaf $\cc F$ of Abelian groups is
stable if for any $k$-smooth variety the pull-back map
$\sigma^*_X: \cc F(X) \to \cc F(X)$
equals the identity map, where
$\sigma_X=(X\times 0, X\times \bb A^1, t; pr_X) \in Fr_1(X,X)$.
In turn, $\cc F$ is quasi-stable if for any $k$-smooth variety the pull-back map
$\sigma^*_X: \cc F(X) \to \cc F(X)$
is an isomorphism. Also, recall that $\cc F$ is  radditive if
$\cc F(\emptyset)=\{0\}$ and $\cc F(X_1\sqcup X_2)=\cc F(X_1)\times \cc F(X_2)$.

For any $\bb A^1$-invariant stable (respectively quasi-stable) radditive $Fr_*$-presheaf of
Abelian groups $G$ there is a unique $\bb A^1$-invariant stable (respectively quasi-stable) $\bb
ZF_*$-presheaf of Abelian groups $\cc F$ such that $G=\cc F \circ
R$. This follows easily from the Additivity Theorem of~\cite{GP3}.

Therefore {\it the category of $\bb A^1$-invariant stable (respectively quasi-stable) radditive
framed pre\-sheaves of Abelian groups is equivalent to the category of
$\bb A^1$-invariant stable (respectively quasi-stable) $\bb ZF_*$-pre\-sheaves of Abelian groups}.

The latter means that the main result formulated in the abstract is equivalent to the following

\begin{thm}[Main]\label{thm1}
For any $\bb A^1$-invariant quasi-stable $\bb
ZF_*$-presheaf of abelian groups $\cc F$, the associated Nisnevich
sheaf $\cc F_{\nis}$ is $\bb A^1$-invariant whenever the base field
$k$ is infinite of characteristic different from 2. Moreover, if the
base field $k$ is infinite perfect of characteristic different from
2, then every $\bb A^1$-invariant quasi-stable Nisnevich framed
sheaf of Abelian groups is strictly $\bb A^1$-invariant and
quasi-stable. Furthermore, the same statements are true in
characteristic 2 if we also assume that the $\bb ZF_*$-presheaf of
abelian groups $\cc F$ is a presheaf of $\bb Z[1/2]$-modules.
\end{thm}

In the rest of the paper we suppose that {\it the base field $k$ is infinite}.

\subsubsection*{Acknowledgements} The authors would like to thank
Alexey Ananyevskiy,
Andrey Druzhinin and Alexander Neshitov for many
helpful discussions.

\section{A few theorems}
The main aim of this section is to state a few major theorems on
preshaves with framed transfers. As an application, we deduce the
following result (which is the first assertion of
Theorem~\ref{thm1}).

\begin{thm}\label{Nisnevich_sheaf}
For any $\bb A^1$-invariant quasi-stable $\bb ZF_*$-presheaf of
abelian groups $\cc F$, the associated Nisnevich sheaf $\cc
F_{\nis}$ is $\bb A^1$-invariant and quasi-stable if the
characteristic of the base field $k$ is different from 2. If the
characteristic of $k$ equals 2 and $\cc F$ is an $\bb A^1$-invariant
quasi-stable $\bb ZF_*$-presheaf of $\bb Z[1/2]$-modules, then the
associated Nisnevich sheaf $\cc F_{\nis}$ is $\bb A^1$-invariant and
quasi-stable presheaf of $\bb Z[1/2]$-modules.
\end{thm}

We need some definitions. We will write $(V,\varphi;g)$ for an
element $a$ in $Fr_n(X,Y)$. We also write $Z_a$ to denote the
support of $(V,\varphi;g)$. It is a closed subset in $X\times \bb
A^n$ which is finite over $X$ and which coincides with the common
vanishing locus of the functions $\varphi_1, ... , \varphi_n$ in
$V$. Next, by $\langle V,\varphi;g \rangle$ we denote the image of
the element $1\cdot(V,\varphi;g)$ in $\bb ZF_n(X,Y)$.

\begin{defs}
\label{sigma} Given any $k$-smooth variety $X$, there is a
distinguished morphism $\sigma_X=(X\times \bb A^1,t,pr_X)\in
Fr_1(X,X)$. Each morphism $f: Y\to X$ in $Sm/k$ can be regarded
tautologically as a morphism in $Fr_0(Y,X)$.
\end{defs}

In what follows by $SmOp/k$ we mean a category whose objects
are pairs $(X,V)$, where $X\in Sm/k$ and $V$ is an open subset of
$X$, with obvious morphisms of pairs.

\begin{defs}
\label{bar_ZF} Define $\bb ZF^{pr}_*(k)$ as an additive category
whose objects are those of $SmOp/k$
and Hom-groups are defined as follows. We set for every
$n\geq 0$ and $(X,V),(Y,W)\in SmOp/k$:
$$\bb ZF^{pr}_*((Y,W),(X,V))=\ker[\bb ZF_n(Y,X)\oplus \bb ZF_n(W,V)\xrightarrow{i^*_Y-i_{X,*}} \bb ZF_n(W,X)],$$
where $i_Y: W \to Y$ is the embedding and $i_X: V \to X$ is the
embedding. In other words, the group $\bb ZF^{pr}_*((Y,W),(X,V))$
consists of pairs $(a,b) \in \bb ZF_n(Y,X)\oplus \bb ZF_n(Y^0,X^0)$
such that $i_Y\circ b= a\circ i_X$. By definition, the composite
$(a,b)\circ (a',b')$ is the pair $((a\circ b), (a'\circ b'))$.

We define $\overline {\bb ZF}_*(k)$ as an additive category whose
objects are those of $Sm/k$ and Hom-groups are defined as follows.
We set for every $n\geq 0$ and $X,Y\in Sm/k$:
$$\overline {\bb ZF}_*(Y,X)= \coker[\bb ZF_*(\bb A^1\times Y,X) \xrightarrow{i^*_0-i^*_1} \bb ZF_*(Y,X)].$$

Next, one defines $\overline {\bb ZF}^{pr}_*(k)$ as an additive
category whose objects are those of $SmOp/k$
and Hom-groups are defined as follows. We set for every
$n\geq 0$ and $(X,V),(Y,W)\in SmOp/k$:
$$\overline {\bb ZF}^{pr}_*((Y,W),(X,V))= \coker[\bb ZF^{pr}_*(\bb A^1\times (Y,W),(X,V)) \xrightarrow{i^*_0-i^*_1} \bb ZF^{pr}_*((Y,W),(X,V)].$$
\end{defs}

\begin{notn}
Given $a \in \bb ZF_*(Y,X)$, denote by $[a]$ its class in $\overline
{\bb ZF}_*(Y,X)$. Similarly, if $r=(a,b) \in \bb
ZF^{pr}_*((Y,W),(X,V))$, then we will write $[[r]]$ to denote its
class in $\overline {\bb ZF}^{pr}_*((Y,W),(X,V))$.

If $X^0$ is open in $X$ and $Y^0$ is open in $Y$ and $g(Z^0) \subset
Y^0$ with $Z^0$ the support of $(V,\varphi;g)$, then $\langle\langle
V,\varphi;g \rangle\rangle$ will stand for the element $(\langle
V,\varphi;g \rangle,\langle V^0,\varphi^0;g^0 \rangle)$ in $\bb
ZF_n((X,X^0),(Y,Y^0))$.

We will as well write $[V,\varphi;g]$ to denote the class of
$\langle V,\varphi;g \rangle$ in $\overline {\bb ZF}_n(X,Y)$. In
turn, $[[V,\varphi;g]]$ will stand for the class of $\langle\langle
V,\varphi;g \rangle\rangle$ of $\overline {\bb
ZF}_n((X,X^0),(Y,Y^0))$.
\end{notn}

\begin{rem} Clearly, the category
$\bb ZF_*(k)$ is a full subcategory of $\bb ZF^{pr}_*(k)$ via the
assignment $X \mapsto (X,\emptyset)$. Similarly, the category
$\overline {\bb ZF}_*(k)$ is a full subcategory of $\overline {\bb
ZF}^{pr}_*(k)$ via the assignment $X \mapsto (X,\emptyset)$.
\end{rem}

In what follows we will also use the following category.

\begin{defs}
Let $\overline{\overline {\bb ZF}}_*(k)$ be a category whose objects
are those of $SmOp/k$ and whose $Hom$-groups are obtained from the
$Hom$-groups of the category $\overline {\bb ZF}^{pr}_*(k)$ by
annihilating the identity morphisms $id_{(X,X)}$ of objects of the
form $(X,X)$ for all $X \in Sm/k$.
\end{defs}

\begin{notn}
\label{sigma_and_pairs} If $r=(a,b) \in \bb ZF^{pr}_*((Y,W),(X,V))$,
then we will write $\overline {[[r]]}$ for its class in
$\overline{\overline {\bb ZF}}_*((Y,W),(X,V))$. For $(X,V)$ in
$SmOp/k$ we write $\langle\langle \sigma_X \rangle\rangle$ for the
morphism $(1\cdot \sigma_X, 1\cdot \sigma_V)$ in $\bb
ZF_1((X,V),(X,V))$.

We will denote by $\overline {[[V,\varphi;g]]}$ the class of the
element $[[V,\varphi;g]]$ in $\overline{\overline {\bb
ZF}}_n((X,X^0),(Y,Y^0))$.
\end{notn}

\begin{constr}
\label{F_on_pairs}
Let $\cc F$ be an $\bb A^1$-invariant $\bb ZF_*$-presheaf of abelian
groups. Then the assignments
$(X,V) \mapsto \cc F(X,V):= \cc F(V)/Im(\cc F(X))$
and
$$(a,b)\mapsto [(a,b)^*=b^*: \cc F(V)/Im(\cc F(X)) \to \cc F(W)/Im(\cc F(Y))],$$
for any $(a,b)\in \bb ZF_*((Y,W),(X,V))$ define a presheaf $\cc
F^{pairs}$ on the category $\overline{\overline {\bb ZF}}_*(k)$.
\end{constr}

The nearest aim is to formulate a series of theorems (each of which
is of independent interest), which are crucial for the proof of
Theorem~\ref{thm1}.

\begin{thm}[Injectivity on affine line]
\label{InjOnAffLine}
Let $U \subset \bb A^1_k$ be an open subset and let $i: V\hookrightarrow U$ be a non-empty open subset.
Then there is a morphism $r \in \bb ZF_1(U,V)$ such that
$[i]\circ [r]= [\sigma_U]$ in $\overline {\bb ZF_1}(U,U)$.
\end{thm}

\begin{thm}[Excision on affine line]
\label{Ex_on_line} Let $U \subset \bb A^1_k$ be an embedding. Let
$i: V \hookrightarrow U$ be an open inclusion with $V$ non-empty.
Let $S \subset V$ be a closed subset. Then there are morphisms $r
\in\bb ZF_1((U,U-S),(V,V-S))$ and $l\in\bb ZF_1((U,U-S),(V,V-S))$
such that
   $$\overline {[[i]]}\circ \overline {[[r]]}=\overline{[[\sigma_U]]} \ \text{and} \ \overline {[[i]]}\circ \overline {[[I]]}=\overline{[[\sigma_V]]}$$
in $\overline {\overline {\bb ZF_1}}((U,U-S),(U,U-S))$
and $\overline {\overline {\bb ZF_1}}((V,V-S),(V,V-S))$
respectively.
\end{thm}

\begin{thm}[Injectivity for local schemes]
\label{Local_Inj}
Let $X \in Sm/k$, $x\in X$ be a point, $U=Spec(\cc O_{X,x})$,
$i: D \hookrightarrow X$ be a closed subset. Then there exists
an integer $N$ and a morphism $r\in \bb ZF_N(U,X-D)$ such that
$$[r]\circ [j]= [can]\circ [\sigma^N_U]$$
in $\overline {\bb ZF}_N(U,X)$ with $j: X-D \hookrightarrow X$ the
open inclusion and $can: U \to X$ the canonical morphism.
\end{thm}

\begin{thm}[Excision on relative affine line]
\label{Exc_On_Rel_Aff_Line} Let $X \in Sm/k$, $x\in X$ be a point,
$W=Spec(\cc O_{X,x})$. Let $i:V=(\bb A^1_W)_f \subset \bb A^1_W$ be
an affine open subset, where $f\in \cc O_{X,x}[t]$ is monic such
that $f(0)\in \cc O^{\times}_{X,x}$.
Then there are morphisms
$$r \in\bb ZF_1((\bb A^1_W,\bb A^1_W-0\times W),(V,V-0\times W)) \ \ \text{and} \ \ l\in\bb ZF_1((\bb A^1_W,\bb A^1_W-0\times W),(V,V-0\times W))$$
such that
$$\overline {[[i]]}\circ \overline {[[r]]}=\overline{[[\sigma_{\bb A^1_W}]]} \ \text{and} \ \overline {[[l]]}\circ \overline {[[i]]}=\overline{[[\sigma_V]]}$$
in $\overline {\overline {\bb ZF_1}}((\bb A^1_W,\bb A^1_W-0\times W),(\bb A^1_W,\bb A^1_W-0\times W))$
and $\overline {\overline {\bb ZF_1}}((V,V-0\times W),(V,V-0\times W))$ respectively.
\end{thm}

To formulate further two theorems relating \'{e}tale excision
property, we need some preparations. Let
$S\subset X$ and $S'\subset X'$ be closed subsets. Let
$$\xymatrix{V'\ar[r]\ar[d]&X'\ar^{\Pi}[d]\\
               V\ar[r]&X}$$
be an elementary distinguished square with $X$ and $X'$ affine
$k$-smooth. Let $S=X-V$ and $S'=X'-V'$ be closed subschemes equipped
with reduced structures. Let $x\in S$ and $x' \in S'$ be two points
such that $\Pi(x')=x$. Let $U=Spec(\cc O_{X,x})$ and $U'=Spec(\cc
O_{X',x'})$. Let $\pi: U' \to U$ be the morphism induced by $\Pi$.

\begin{thm}[Injective \'{e}tale excision]
\label{Inj_Etale_exc} Under the notation above there is an integer
$N$ and a morphism $r \in\bb ZF_N((U,U-S),(X',X'-S'))$ such that
$$\overline {[[\Pi]]}\circ \overline {[[r]]}=\overline{[[can]]}\circ \overline{[[\sigma^N_U]]} $$
in $\overline {\overline {\bb ZF_N}}((U,U-S),(X,X-S))$,
where $can: U \to X$ is the canonical morphism.
\end{thm}

The statements of the next theorem depend on the characteristic of the base field $k$.

\begin{thm}[Surjective \'{e}tale excision]
\label{Surj_Etale_exc} Under the above notations suppose in addition that $S$ is $k$-smooth
and $k$ is of characteristic different from 2. Then there are an integer
$N$ and a morphism $l \in \bb ZF_N((U,U-S),(X',X'-S'))$ such that
$$\overline {[[l]]}\circ \overline {[[\pi]]}=\overline{[[can']]}\circ \overline{[[\sigma^N_{U'}]]} $$
in $\overline {\overline {\bb ZF_N}}((U',U'-S'),(X',X'-S'))$ with
$can': U' \to X'$ the canonical morphism.

If the charcteristic of $k$ is $2$, then
there are an integer
$N$ and a morphism $l \in \bb ZF_N((U,U-S),(X',X'-S'))$ such that
$$2\cdot \overline {[[l]]}\circ \overline {[[\pi]]}=2\cdot \overline{[[can']]}\circ \overline{[[\sigma^N_{U'}]]} $$
in $\overline {\overline {\bb ZF_N}}((U',U'-S'),(X',X'-S'))$.
\end{thm}

We are now in a position to prove the following

\begin{thm}
\label{Few_On_P_w_Tr} For any $\bb A^1$-invariant quasi-stable $\bb
ZF_*$-presheaf of abelian groups $\cc F$ the following statements
are true:
\begin{itemize}
\item[(1)] under the assumptions of Theorem \ref{InjOnAffLine} the map
$i^*: \cc F(U)\to  \cc F(V)$ is injective;
\item[(2)] under the assumptions of Theorem \ref{Ex_on_line} the map
$$[[i]]^*: \cc F(U-S)/\cc F(U) \to  \cc F(V-S)/\cc F(V)$$ is an isomorphism;
\item[(3)] under the assumptions of Theorem \ref{Local_Inj} the map
$$\eta^*: \cc F(U)\to  \cc F(Spec(k(X))$$
is injective, where $\eta: Spec(k(X)) \to U$
is the canonical morphism;
\item[(3')] under the assumptions of Theorem \ref{Local_Inj} let $U^h_x$ be the henselization of $U$ at $x$ and let
$k(U^h_x)$ be the function field on $U^h_x$. Then the map
$$\eta^*_h: \cc F(U^h_x)\to  \cc F(Spec(k(U^h_x))$$
is injective, where $\eta_h: Spec(k(U^h_x)) \to U^h_x$
is the canonical morphism;
\item[(4)] under the assumptions of Theorem \ref{Exc_On_Rel_Aff_Line} the map
$$[[i]]^*: \cc F(\bb A^1_W-0\times W)/\cc F(\bb A^1_W) \to  \cc F(V-0\times W)/\cc F(V)$$ is an isomorphism;
\item[(5)] under the assumptions of Theorems \ref{Inj_Etale_exc} and \ref{Surj_Etale_exc}
the map
$$[[\Pi]]^*: \cc F(U-S)/\cc F(U) \to  \cc F(U'-S')/\cc F(U')$$ is an
isomorphism whenever the characteristic of $k$ is different from 2.

If the characteristic of $k$ is 2 and the presheaf $\cc F$ is a presheaf of
$\bb Z[1/2]$-modules, then the map
$$[[\Pi]]^*: \cc F(U-S)/\cc F(U) \to  \cc F(U'-S')/\cc F(U')$$ is an isomorphism.
\end{itemize}
\end{thm}

\begin{proof}
Without loss of generality we may assume that $\cc F$ is stable and prove the theorem for this case
(by slight modifications, which are left to the reader, the theorem is similarly proved if $\cc F$ is quasi-stable).
Assertions (1), (3) and (3') follow from Theorems \ref{InjOnAffLine} and
\ref{Local_Inj}. To prove assertions (2), (4) and (5), use
Construction \ref{F_on_pairs} and apply Theorems \ref{Ex_on_line},
\ref{Exc_On_Rel_Aff_Line}, \ref{Inj_Etale_exc}, and
\ref{Surj_Etale_exc} respectively (recall that $\cc F$ is stable).
\end{proof}

\begin{proof}[Proof of Theorem \ref{Nisnevich_sheaf}]
We prove the theorem for fields of characteristic not 2 and leave
the reader to prove it for fields of characteristic 2.
Firstly, $(1)$ and $(2)$ imply $\cc F|_{\bb A^1}$ is a Zariski
sheaf. Using $(5)$ applied to $X=\bb A^1$, one shows that for any
open $V$ in $\bb A^1$ one has $\cc F_{Nis}(V)=\cc F(V)$.

Now consider the following Cartesian square of schemes
$$\xymatrix{Spec(k(X)) \ar[r]^{\eta}\ar[d]_{i_{0,k(X)}}&X\ar^{i_{0,X}}[d]\\
               \bb A^1_{k(X)} \ar[r]^{\eta \times id}&X \times \bb A^1}$$
Evaluating the Nisnevich sheaf $\cc F_{Nis}$ on this square, we get
a square of abelian groups
$$\xymatrix{\cc F_{Nis}(Spec(k(X))) & \cc F_{Nis}(X)\ar[l]_{\eta^*}\\
               \cc F_{Nis}(\bb A^1_{k(X)})\ar[u]^{i^*_{0,k(X)}} & \cc F_{Nis}(X \times \bb A^1) \ar[l]_{(\eta \times id)^*} \ar[u]_{i^*_{0,X}}}$$
The map $i^*_{0,X}$ is plainly surjective. It remains to check its
injectivity. The map $(\eta \times id)^*$ is injective (apply (3')).
As already mentioned in this proof, $\cc F_{Nis}(\bb A^1_{k(X)})=\cc
F(\bb A^1_{k(X)}))$. Since $\cc F_{Nis}(Spec(k(X))=\cc
F(Spec(k(X))$, we see that the map $i^*_{0,k(X)}$ is an isomorphism.
Thus the map $i^*_{0,X}$ is injective.
\end{proof}

We finish the section by proving the following useful statement,
which is a consequence of Theorem~\ref{Few_On_P_w_Tr}(4):

\begin{cor}
\label{Exc_On_Rel_Aff_Line_3}
Let $X \in Sm/k$, $x\in X$ be a point,
$W=Spec(\cc O_{X,x})$. Let
$\cc V:=Spec(\cc O_{W\times \bb A^1,(x,0)})$
and $can: \cc V \hookrightarrow W\times \bb A^1$
be the canonical embedding.
Let $\cc F$ be an $\bb A^1$-invariant quasi-stable $\bb ZF_*$-presheaf of abelian groups.
Then the pull-back map
   $$[[can]]^*: \cc F(W\times(\bb A^1-\{0\}))/\cc F(W\times \bb A^1) \to \cc F(\cc V - W\times \{0\})/\cc F(\cc V)$$
is an isomorphism (both quotients make sense: the second quotient makes sense due to
Theorem~\ref{Few_On_P_w_Tr}(3), the first one makes sense due to homotopy invariance of $\cc F$).
\end{cor}

\begin{proof}
Consider the $W$-scheme $W\times {\bb P^1}$ and effective divisors of the form
$H \sqcup D$
on $W\times {\bb P^1}$ such that
$H$ is a section of the projection $W\times {\bb P^1}\to W$,
$D$ is a reduced divisor,
$H\cap (W\times 0)=\emptyset$
and
$D\cap (W\times 0)=\emptyset$. For such a divisor
set $V_{H,D}:=W\times \bb P^1-(H \sqcup D)$. Note that
$(W\times 0)\subset V_{H,D}$.

Consider the category, $\cc C$, of
Zarisky neighborhoods of $(W\times 0)$ in $W\times {\bb P^1}$ as well as
the presheaf $V\mapsto \cc F(V-W\times 0)/Im(\cc F(V))$ on $\cc C$.
Clearly, the category $\cc C$ is co-filtered.
By definition, one has
$$\cc F(\cc V)=\lim_{\rightarrow} \cc F(V) \ \
\text{and} \ \
\cc F(\cc V-W\times 0)=\lim_{\rightarrow} \cc F(V-W\times 0),$$
where $V$ runs over all
Zarisky neighborhoods of $(W\times 0)$ in $W\times {\bb P^1}$. Thus
$$\cc F(\cc V - W\times \{0\})/\cc F(\cc V)=\lim_{\rightarrow} \cc F(V-W\times 0)/Im(\cc F(V)).$$
Let $\cc C^{\prime}$ be the full subcategory of $\cc C$ consisting of objects of the form $V_{H,D}$.
Since the base field $k$ is infinite and $W$ is regular local, then the subcategory
$\cc C^{\prime}$ is cofinal in $\cc C$.
Thus,
$$\cc F(\cc V - W\times \{0\})/\cc F(\cc V)=\lim_{\rightarrow} \cc F(V_{H,D}-W\times 0)/Im(\cc F(V_{H,D})).$$
{\it We claim that for any inclusion}
$\epsilon: V_{H_2,D_2} \hookrightarrow V_{H_1,D_1}$
the pull-back map
$$[[\epsilon]]^*: \cc F(V_{H_1,D_1}-W\times 0)/Im(\cc F(V_{H_1,D_1})) \to \cc F(V_{H_2,D_2}-W\times 0)/Im(\cc F(V_{H_2,D_2}))$$
is an isomorphism. To prove this claim, note that the inclusion above yields an inclusion
$H_1 \sqcup D_1 \subset H_2 \sqcup D_2$. We have that either $H_1=H_2$ or $H_1\subset D_2$.
In the second case one has
$H_1\cap H_2=\emptyset$,
$H_1 \subset D_2$
and
$H_2\sqcup H_1 \subset H_2\sqcup D_2$.
If $H_1=H_2$, then one has inclusions
$V_{H_2,D_2} \xrightarrow{\alpha} V_{H_1,D_1} \xrightarrow{\beta} V_{H_1,\emptyset}$.
Set $\gamma=\beta \circ \alpha$.
By Theorem \ref{Few_On_P_w_Tr}(item (4)) the maps
$[[\beta]]^*$ and $[[\gamma]]^*$ are isomorphisms.
Thus the map $[[\alpha]]^*$ is an isomorphism in this case, too.
In the second case consider inclusions
$V_{H_2,D_2} \xrightarrow{\gamma} V_{H_2,H_1} \xrightarrow{\beta} V_{H_2,\emptyset}$
and set $\alpha=\beta \circ \gamma$.
By Theorem \ref{Few_On_P_w_Tr}(item (4)) the maps
$[[\beta]]^*$ and $[[\alpha]]^*$ are isomorphisms.
Thus the map $[[\gamma]]^*$ is an isomorphism.
Now consider inclusions
$V_{H_2,H_1} \xrightarrow{\delta} V_{H_1,\emptyset}$
and
$V_{H_1,D_1} \xrightarrow{\rho} V_{H_1,\emptyset}$.
One has $\delta \circ \gamma=\rho\circ \epsilon$.
We already know that $[[\gamma]]^*$ is an isomorphism.
By Theorem \ref{Few_On_P_w_Tr}(item (4)) the maps
$[[\rho]]^*$ and $[[\delta]]^*$ are isomorphisms.
Thus $[[\epsilon]]^*$ is an isomorphism in the second case, too.
{\it The claim is proved}. Thus for any
$V_{H,D} \in \cc C^{\prime}$ the map
$$\cc F(V_{H,D}-W\times 0)/Im(\cc F(V_{H,D})) \to \cc F(\cc V - W\times \{0\})/\cc F(\cc V)$$
is an isomorphism.
In particular, the map
$$\cc F(W\times {\bb A^1}-W\times 0)/\cc F(W\times {\bb A^1}) =
\cc F(V_{W\times {\infty}\ ,\emptyset}-W\times 0)/\cc F(V_{W\times {\infty}\ , \emptyset}) \to \cc F(\cc V - W\times \{0\})/\cc F(\cc V)$$
is an isomorphism, whence the corollary.
\end{proof}

\section{Notation and agreements}

\begin{notn}
\label{Main} Given a morphism $a \in Fr_n(Y,X)$, we will write
$\langle a\rangle$ for the image of $1\cdot a$ in $\bb ZF_n(Y,X)$ and write
$[a]$ for the class of $\langle a\rangle$ in $\overline {\bb
ZF}_n(Y,X)$.

Given a morphism $a \in Fr_n(Y,X)$, we will write $Z_a$ for the
support of $a$ (it is a closed subset in $Y\times \bb A^n$ which
finite over $Y$ and determined by $a$ uniquely). Also, we will often
write
$$(\cc V_a,\varphi_a: \cc V_a \to \bb A^n; g_a: \cc V_a\to X) \ \text{or shorter} \ (\cc V_a,\varphi_a; g_a) $$
for a representative of the morphism $a$ (here $(\cc V_a, \rho: \cc
V_a \to Y\times \bb A^n, s: Z_a \hookrightarrow \cc V_a)$ is an
\'{e}tale neighborhood of $Z_a$ in $Y\times \bb A^n$).
\end{notn}

\begin{lem}
\label{Disjoint_Support} If the support $Z_a$ of an element $a=(\cc
V,\varphi; g) \in Fr_n(X,Y)$ is a disjoint union of $Z_1$ and
$Z_2$, then the element $a$ determines two elements $a_1$ and $a_2$
in $Fr_n(X,Y)$. Namely, $a_1=(\cc V-Z_2,\varphi|_{\cc V-Z_2};
g|_{\cc V-Z_2})$ and $a_2=(\cc V-Z_1,\varphi|_{\cc V-Z_1}; g|_{\cc
V-Z_1})$. Moreover, by the definition of $\bb ZF_n(X,Y)$ one has the
equality
$$\langle a \rangle= \langle a_1 \rangle + \langle a_2 \rangle$$
in $\bb ZF_n(X,Y)$.
\end{lem}

\begin{defs}
\label{Runs_Inside} Let $i_Y: Y'\hookrightarrow Y$ and $i_X:
X'\hookrightarrow X$ be open embeddings. Let $a\in Fr_n(Y,X)$. We
say that the restriction $a|_{Y'}$ of $a$ to $Y'$ runs inside $X'$,
if there is $a' \in Fr_n(Y',X')$ such that
\begin{equation}
\label{Eq_Runs_Inside}
i_X\circ a'=a\circ i_Y
\end{equation}
in $Fr_n(Y',X)$.

It is easy to see that if there is a morphism $a'$ satisfying
condition (\ref{Eq_Runs_Inside}), then it is unique. In this case
the pair $(a,a')$ is an element of $\bb ZF_n((Y,Y'),(X,X'))$. For
brevity we will write $\langle\langle a\rangle\rangle$
for $(a,a') \in \bb ZF_n((Y,Y'),(X,X'))$
and write $[[a]]$ to denote the class of $\langle\langle a\rangle\rangle$
in $\overline {\bb ZF}_n((Y,Y'),(X,X'))$.
\end{defs}

\begin{lem}
\label{Criteria} Let $i_Y: Y'\hookrightarrow Y$ and $i_X:
X'\hookrightarrow X$ be open embeddings. Let $a\in Fr_n(Y,X)$. Let
$Z_a \subset Y\times \bb A^n$ be the support of $a$. Set $Z'_a=Z_a
\cap Y' \times \bb A^n$. Then the following are equivalent:
\begin{itemize}
\item[(1)]
$g_a(Z'_a)\subset X'$;
\item[(2)]
the morphism $a|_{Y'}$ runs inside $X'$.
\end{itemize}
\end{lem}

\begin{proof}
$(1) \Rightarrow (2)$. Set $\cc V'=p^{-1}_Y \cap g^{-1}(X')$, where
$p_Y=pr_Y\circ \rho_a: \cc V \to Y\times \bb A^n$. Then $a':=(\cc
V', \varphi|_{\cc V'}); g|_{\cc V'})\in Fr_n(Y',X')$ satisfies
condition (\ref{Eq_Runs_Inside}).

$(2) \Rightarrow (1)$. If $a|_{Y'}$ runs inside $X'$, then for some
$a'=(\cc V', \varphi'; g') \in Fr_n(Y',X'))$ equality
(\ref{Eq_Runs_Inside}) holds. In this case the support $Z'$ of $a'$
must coincide with $Z'_a=Z_a \cap Y' \times \bb A^n$ and
$g_a|_{Z'}=g'|_{Z'}$. Since $g'(Z')$ is a subset of $X'$, then
$g_a(Z'_a)=g_a(Z')\subset X'$.
\end{proof}

\begin{cor}
\label{Critaria_for_H} Let $i_Y: Y'\hookrightarrow Y$ and $i_X:
X'\hookrightarrow X$ be open embeddings. Let $h_{\theta}=(\cc
V_{\theta},\varphi_{\theta}; g_{\theta})\in Fr_n(\bb A^1\times
Y,X)$. Suppose $Z_{\theta}$, the support of $h_{\theta}$, is such
that for $Z'_{\theta}:=Z_{\theta}\cap \bb A^1\times Y' \times \bb
A^n$ one has $g_{\theta}(Z'_{\theta})\subset X'$. Then there are
morphisms $\langle\langle h_{\theta} \rangle\rangle \in\bb ZF_n(\bb
A^1\times (Y,Y'),(X,X'))$, $\langle\langle h_0 \rangle\rangle \in\bb
ZF_n((Y,Y'),(X,X'))$, $\langle\langle h_1 \rangle\rangle \in\bb
ZF_n((Y,Y'),(X,X'))$ and one has an obvious equality
$$[[h_0]=[[h_1]]$$
in $\overline {\bb ZF}_n((Y,Y'),(X,X'))$.
\end{cor}

\begin{lem}[A disconnected support case]
\label{Criteria_for_pairs} Let $i_Y: Y'\hookrightarrow Y$ and $i_X:
X'\hookrightarrow X$ be open embeddings. Let $a\in Fr_n(Y,X)$ and
let $Z_a \subset Y\times \bb A^n$ be the support of $a$. Set
$Z'_a=Z_a \cap Y' \times \bb A^n$. Suppose that $Z_a=Z_{a,1} \sqcup
Z_{a,2}$. For $i=1,2$ set $\cc V_i=\cc V_a-Z_{a,j}$ with $j\in
\{1,2\}$ and $j\neq i$. Also set $\varphi_i=\varphi_a|_{\cc V_i}$
and $g_i=g_a|_{\cc V_i}$. Suppose $a|_{Y'}$ runs inside $X'$, then
\begin{itemize}
\item[(1)]
for each $i=1,2$ the morphism $a_i:=(\cc V_i, \varphi_i; g_i)$ is such that $a_i|_{Y'}$ runs inside $X'$;
\item[(2)]
$\langle\langle a \rangle\rangle = \langle\langle a_1 \rangle\rangle + \langle\langle a_2 \rangle\rangle$
in $\bb ZF_n((Y,Y'),(X,X'))$.
\end{itemize}

\end{lem}

\section{Some homotopies}

Suppose $U,W \subset \bb A^1_k$ are open and non-empty.

\begin{lem}
\label{g_and_g'} Let $a_0=(\mathcal V,\varphi;g_0)\in Fr_1(U,W)$,
$a_1=(\mathcal V,\varphi;g_1) \in Fr_1(U,W)$. Suppose that the
supports of $a_0$ and $a_1$ coincide. Denote their common support by
$Z$. If $g_0|_Z=g_1|_Z$, then $[a_0]=[a_1]$ in $\bb ZF_1(U,W)$.
\end{lem}

\begin{proof}
Consider a function $g_{\theta}=(1-\theta)g_0+ \theta g_1: \bb
A^1\times \cc V \to \bb A^1$ and set $\cc
V_{\theta}=g^{-1}_{\theta}(W)$, $\varphi_{\theta}= \varphi \circ
pr_{\cc V}: \cc V_{\theta} \to \bb A^1_k$. Next, consider a homotopy
\begin{equation}
\label{H_g_and_g'}
h_{\theta}=(\cc V_{\theta}, \varphi_{\theta};g_{\theta}) \in Fr_1(\bb A^1 \times U, W).
\end{equation}
The support of $h_{\theta}$ equals $\bb A^1\times Z \subset \bb
A^1\times U\times \bb A^1$. Clearly, $h_0=a_0$ and $h_1=a_1$. Whence
the lemma.
\end{proof}

\begin{cor}
\label{g_and_g'_for_pairs} Under the assumptions of Lemma
\ref{g_and_g'} let $U'\subset U$ and $W'\subset W$ be open subsets.
Suppose that $a_0|_{U'}$ runs inside $W'$. Then $a_1|_{U'}$ runs
inside $W'$, the restriction $h_{\theta}|_{\bb A^1\times U'}$ of the
homotopy $h_{\theta}$ runs inside $W'$ and
   $$[[a_0]]=[[a_1]]$$
in $\overline {\bb ZF}_1((U,U'),(W,W'))$.
\end{cor}

\begin{lem}
\label{u_and_u'} Let $a_0=(\mathcal V,\varphi u_0 ;g)\in Fr_1(U,W)$,
$a_1=(\mathcal V,\varphi u_1;g) \in Fr_1(U,W)$, where $u_0,u_1 \in
k[\cc V]$ are units. In this case the supports of $a_0$ and $a_1$
coincide. Denote their common support by $Z$. Suppose
$u_0|_Z=u_1|_Z$, then $[a_0]=[a_1]$ in $\bb ZF_1(U,W)$.
\end{lem}

\begin{proof}
Set $u_{\theta}=(1-\theta) u_0+ \theta u_1 \in k[\bb A^1\times \cc
V]$. Clearly, $u_{\theta}|_{\bb A^1\times Z}=pr^*_Z(u_0)=pr^*_Z(u_1)
\in k[\bb A^1\times Z]$. Let $\cc V_{\theta}=\{u_{\theta}\neq 0\}
\subset \bb A^1\times \cc V$. Set,
\begin{equation}
\label{H_g_and_g'}
h_{\theta}=(\cc V_{\theta}, u_{\theta} \varphi ;g\circ pr_{\cc V}) \in Fr_1(\bb A^1 \times U, W).
\end{equation}
The support of $h_{\theta}$ equals $\bb A^1\times Z \subset \bb
A^1\times U\times \bb A^1$. Clearly, $h_0=a_0$ and $h_1=a_1$. Whence
the lemma.
\end{proof}

\begin{cor}
\label{u_and_u'_for_pairs} Under the assumptions of Lemma
\ref{u_and_u'}, let $U'\subset U$ and $W'\subset W$ be open subsets.
Suppose $a_0|_{U'}$ runs inside $W'$. Then $a_1|_{U'}$ runs inside
$W'$, the restriction $h_{\theta}|_{\bb A^1\times U'}$ of the
homotopy $h_{\theta}$ from the proof of Lemma \ref{u_and_u'} runs
inside $W'$ and
$$[[a_0]]=[[a_1]]$$
in $\overline {\bb ZF}_1((U,U'),(W,W'))$.
\end{cor}

\begin{lem}
\label{Higher_terms} Let $U \subset \bb A^1_k$ be non-empty open as
above. Suppose $F_0(Y),F_1(Y) \in k[U][Y]$. Let
$deg_Y(F_0)=deg_Y(F_1)=d>0$ and let their leading coefficients
coincide and are units in $k[U]$. Then,
$$[U\times \bb A^1, F_0(Y), pr_U]= [U\times \bb A^1, F_1(Y), pr_U] \in \bb ZF_1[U,U].$$
\end{lem}

\begin{proof}
Set $F_{\theta}(Y)=(1-\theta)F_0(Y)+ \theta F_1(Y) \in
k[U][\theta,Y]$. Consider a morphism
\begin{equation}
\label{H_for_higher_terms}
h_{\theta}=(\bb A^1\times U\times \bb A^1, F_{\theta} ; pr_{U}) \in Fr_1(\bb A^1 \times U, U).
\end{equation}
Clearly, $h_0=(U\times \bb A^1, F_0(Y), pr_U)$ and $h_1=(U\times \bb A^1, F_1(Y), pr_U)$. Whence the lemma.
\end{proof}

\begin{cor}
\label{Higher_terms_for_pairs}
Under the assumptions of Lemma \ref{Higher_terms} let $U'\subset U$ be an open subset. Then
$(U\times \bb A^1, F_0(Y), pr_U)|_{U'}$, $(U\times \bb A^1, F_1(Y), pr_U)|_{U'}$ runs inside $U'$
and the restriction
$h_{\theta}|_{\bb A^1\times U'}$ of the homotopy
$h_{\theta}$ from the proof of Lemma \ref{Higher_terms} runs inside $W'=\bb A^1\times U'$ and
$$[[U\times \bb A^1, F_0(Y), pr_U]]= [[U\times \bb A^1, F_0(Y), pr_U] \in \bb ZF_1[(U,U'),(U,U')]]$$
in $\overline {\bb ZF}_1((U,U'),(U,U'))$.
\end{cor}

\begin{prop}
\label{Diagonal} Let $U\subset \bb A^1_k$ and $U' \subset U$ be open
subsets. Let $t\in k[\bb A^1]$ be the standard parameter on $\bb
A^1_k$. Set $X:=(t\otimes 1)|_{U\times U}\in k[U\times U]$ and
$Y:=(1\otimes t)|_{U\times U}\in k[U\times U]$. Then for any integer
$n\geq 1$, one has an equality
$$[[U\times U, (Y-X)^{2n+1}, p_2]]=[[U\times U, (Y-X)^{2n}, p_2]]+[[\sigma_U]]$$
in $\overline {\bb ZF}_1((U,U'),(U,U'))$.
\end{prop}

\begin{proof}
Let $m\geq 1$ be an integer. Then
\begin{equation}
\label{star} [[U\times U, (Y-X)^{m}, p_2]]= [[U\times U, (Y-X)^{m},
p_1]]= [[U\times \bb A^1, (Y-X)^{m}, p_1]]=[[U\times \bb A^1, Y^{m},
p_1]]
\end{equation}
in $\overline {\bb ZF}_1((U,U'),(U,U'))$.
The first equality follows from Corollary
\ref{g_and_g'_for_pairs}, the third one follows from Corollary
\ref{Higher_terms_for_pairs}, the middle one is obvious.

There is a chain of equalities in
$\overline {\bb ZF}_1((U,U'),(U,U'))$:
$$
[[U\times \bb A^1, Y^{2n+1}; p_1]]=[[U\times \bb A^1, Y^{2n}(Y+1); p_1]]=
$$
$$
=[[U\times (\bb A^1-\{-1\}),Y^{2n}(Y+1); p_1]]+ [[U\times (\bb A^1-\{0\}),Y^{2n}(Y+1); p_1]]=
$$
$$
=[[\cc V_0,Y^{2n}; p_1]]+ [[\cc V_1,(Y+1); p_1]]=
$$
$$
=[[U\times \bb A^1,Y^{2n}; p_1]]+ [[U\times \bb A^1,(Y+1); p_1]].
$$
Here the first equality holds by Corollary
\ref{Higher_terms_for_pairs}, the second one holds by Lemma
\ref{Criteria_for_pairs}, the third one holds by Corollary
\ref{u_and_u'_for_pairs}, the forth one is obvious (replacement of
neighborhoods).

Continue the chain of equalities in $\overline {\bb ZF}_1((U,U'),(U,U'))$ as follows:
$$
[[U\times \bb A^1,Y^{2n}; p_1]]+ [[U\times \bb A^1,(Y+1); p_1]]=[[U\times \bb A^1,(Y-X)^{2n}; p_1]]+ [[U\times \bb A^1,Y; p_1]]=
$$
$$
=[[U\times \bb A^1,(Y-X)^{2n}; p_1]]+[[\sigma_U]]=[[U\times U,(Y-X)^{2n}; p_1]]+[[\sigma_U]]=[[U\times U,(Y-X)^{2n}; p_2]]+[[\sigma_U]].
$$
Here the first equality holds by
Corollary \ref{Higher_terms_for_pairs}, the second one holds by the definition of $\sigma_U$
(see Notation \ref{sigma_and_pairs}), the third one is obvious,
the fouth one holds by Corollary~\ref{g_and_g'_for_pairs}. We proved the equality
\begin{equation}
\label{star_star} [[U\times \bb A^1, Y^{2n+1}; p_1]]=[[U\times U,(Y-X)^{2n}; p_2]]+[[\sigma_U]].
\end{equation}
Combining that with the equality
(\ref{star}) for $m=2n+1$ we get the desired equality
$$[[U\times U, (Y-X)^{2n+1}; p_2]]=[[U\times U,(Y-X)^{2n}; p_2]]+[[\sigma_U]]$$
in $\overline {\bb ZF}_1((U,U'),(U,U'))$. Whence the proposition.
\end{proof}

\section{Injectivity and excision on affine line}\label{Inj_and_Exc_On_Aff_Line}

The aim of this section is to prove Theorems \ref{InjOnAffLine} and \ref{Ex_on_line}.

\begin{lem}
\label{G0_and_G1}
Let $U\subset \bb A^1$ be open and non-empty. Let $A=\bb A^1_k - U$.
Let $G_0(Y), G_1(Y)\in k[U][Y]$ be such that
\begin{itemize}
\item[(1)]
$deg_Y(G_0)=deg_Y(G_1)$;
\item[(2)]
both are unitary in $Y$ and the leading coefficients equal one;
\item[(3)]
$G_0|_{U\times A}= G_1|_{U\times A} \in k[U\times A]^{\times}$.
\end{itemize}
Then
$$[U\times U, G_0; p_2]=[U\times U, G_1; p_2]$$
in $\overline {\bb ZF}_1(U,U)$.
\end{lem}

\begin{proof}
One has a homotopy $h_{\theta}=(\bb A^1\times U\times U,
G_{\theta},p_{2,U}) \in Fr_1(\bb A^1\times U,U)$,
where
$G_{\theta}=(1-\theta)G_0+\theta G_1$
and $pr_{2,U}: \bb A^1\times U\times U \to U$
is the projection onto the second copy of $U$.
Its restriction to $0\times U$ and to $1\times U$ coincides with
morphisms $(U\times U, G_0; p_2)$ and $(U\times U, G_1; p_2)$
respectively. Whence the lemma.
\end{proof}

\begin{proof}[Proof of Theorem \ref{InjOnAffLine}]
Under the assumptions of this theorem set $A=\bb A^1_k - U$ and
$B=U-V$. For each big enough integer $m\geq 0$ find a polinomial
$F_m(Y)\in k[U][Y]$ such that $F_m(Y)$ is of degree $m$ with the
leading coefficient equal 1 and such that
\begin{itemize}
\item[(i)]
$F_m(Y)|_{U\times A}=(Y-X)^m|_{U\times A}\in k[U\times A]^{\times}$;
\item[(ii)]
$F_m(Y)|_{U\times B}=1\in k[U\times B]^{\times}$.
\end{itemize}
Take $n\gg 0$ and set $r=\langle U\times V, F_{2n+1}; pr_V\rangle -
\langle U\times V, F_{2n}; pr_V\rangle \in \bb ZF_1(U,V)$. Then one
has a chain equalities in $\overline {\bb ZF}_1(U,U)$:
$$
[i]\circ [r]=[U\times U, F_{2n+1}; p_2]- [U\times U, F_{2n}; p_2]
=[U\times U, (Y-X)^{2n+1}; p_2]-[U\times U, (Y-X)^{2n}; p_2]=
$$
$$
=[\sigma_U].
$$
Here the first equality is obvious, the second one holds by Lemma
\ref{G0_and_G1}, the third one holds by Proposition \ref{Diagonal}.
Whence the theorem.
\end{proof}

\begin{cor}[of Lemma \ref{G0_and_G1}]
\label{G0_and_G1_for_pairs} Under the conditions and notation of
Lemma \ref{G0_and_G1} let $S\subset U$ be a closed subset.
Additionally to the conditions $(1)-(3)$ suppose that the following
to conditions hold:
\begin{itemize}
\item[(4)]
$G_0(Y)|_{U\times S}=G_1(Y)|_{U\times S}$,
\item[(5)] $G_0(Y)|_{(U-S)\times S}$ is invertible.
\end{itemize}
Then one has an equality
$$[[U\times U, G_0; p_2]]= [[U\times U, G_1; p_2]]$$
in $\overline {\bb ZF}_1((U,U-S),(U,U-S))$.
\end{cor}

\begin{proof}[Proof of the corollary]
The support $Z_{\theta}$ of the homotopy $h_{\theta}$ from the proof
of Lemma \ref{G0_and_G1} coincides with the vanishing locus of the
polinomial $G_{\theta}$. Since $G_{\theta}|_{\bb A^1\times
(U-S)\times S}$ is invertible, then $Z_{\theta}\cap \bb A^1\times
(U-S)\times S=\emptyset$. By Lemma \ref{Criteria} the homotopy $h_{\theta}|_{\bb
A^1\times (U-S)}$ runs inside $U-S$. Hence
$$[[U\times U, G_0; p_2]]=[[h_0]]=[[h_1]]=[[U\times U, G_1; p_2]]$$
in $\overline {\bb ZF}_1((U,U-S),(U,U-S))$.
In fact, the second equality here holds by Corollary
\ref{Critaria_for_H}.
The first and the third equalities hold since
for $i=0,1$ one has $h_i=(U\times U, G_i; p_2)$ in $Fr_1(U,U)$.
\end{proof}

\begin{proof}[Proof of Theorem \ref{Ex_on_line}]
Firstly construct a morphism
$r\in \bb ZF_1((U,U-S)),(V,V-S))$
such that for its class $[[r]]$ in
$\overline {\bb ZF}_1((U,U-S)),(V,V-S))$
one has
\begin{equation}
\label{Injectivity_Part}
[[i]]\circ [[r]]=[[\sigma_U]]
\end{equation}
in
$\overline {\bb ZF}_1((U,U-S)),(U,U-S))$.

To this end set $A=\bb A^1_k - U$, $B=U-V$. Recall that $S\subset V$
is a closed subset. Take any big enough integer $m\geq 1$ and find a
unitary polinomial $F_m(Y)$ of degree $m$ satisfying the following
properties:
\begin{itemize}
\item[(i)]
$F_m(Y)|_{U\times A}=(Y-X)^m|_{U\times A}\in k[U\times A]^{\times}$;
\item[(ii)]
$F_m(Y)|_{U\times B}=1\in k[U\times B]^{\times}$;
\item[(iii)]
$F_m(Y)|_{U\times S}=(Y-X)^m|_{U\times S}\in k[U\times S]$.
\end{itemize}
Note that $F_m(Y)|_{(U-S)\times S}\in k[(U-S)\times S]^{\times}$.
Hence by Lemma \ref{Criteria} the morphism
$(U\times V, F_{m}; pr_V) \in Fr_1(U,V)$
being restricted to $U-S$ runs inside $V-S$.
Thus using Definition \ref{Runs_Inside} we get a morphism
$$\langle\langle U\times V, F_{m}; pr_V\rangle\rangle \in \bb ZF_1((U,U-S)),(V,V-S)).$$
For that morphism one has equalities
$$
[[i]]\circ [[U\times V, F_{m}; pr_V]]= [[U\times U, F_{m}; p_2]]=[[U\times U, (Y-X)^{m}; p_2]]
$$
in
$\bb ZF_1((U,U-S)),(U,U-S))$.
Here the first equality is obvious, the second one follows from Corollary
\ref{G0_and_G1_for_pairs}.
Take a big enough integer $n$.
Set
$$r=\langle\langle U\times V, F_{2n+1}; pr_V \rangle\rangle-\langle\langle U\times V, F_{2n}; pr_V\rangle\rangle \in \bb ZF_1((U,U-S)),(V,V-S)).$$
{\it We claim that}
$[[i]]\circ [[r]]=[[\sigma_U]]$
in
$\overline {\bb ZF}_1((U,U-S)),(U,U-S))$. In fact,
$$
[[i]]\circ [[r]]=[[U\times U, (Y-X)^{2n+1}; p_2]]-[[U\times U, (Y-X)^{2n}; p_2]]=[[\sigma_U]].
$$
The first equality proven a few lines above and the second one
follow from Proposition \ref{Diagonal}. Whence equality
(\ref{Injectivity_Part}) holds.

Now find morphisms $l\in \bb ZF_1((U,U-S)),(V,V-S))$ and $g\in \bb ZF_1((V,V-S)),(V-S,V-S))$ such that
\begin{equation}
\label{Surjectivity_Part}
[[l]]\circ [[i]]-[[j]]\circ [[g]] =[[\sigma_V]]
\end{equation}
in $\overline {\bb ZF}_1(V,V-S)),(V,V-S))$. Here $j: (V-S,V-S) \to
(V,V-S)$ is the inclusion. Clearly, equality
(\ref{Surjectivity_Part}) yields $[[l]]\circ [[i]]=[[\sigma_V]] \in
\overline{\overline {\bb ZF}}_1(V,V-S)),(V,V-S))$.

Set
$A'=\bb A^1_k - U$, $B=U-V$ and recall that $S\subset V$ is a closed subset.
Take an integer $m$ big enough and find a monic in $Y$ polinomial $F_m(Y)\in k[U][Y]$
such that
\begin{itemize}
\item[(i)]
$F_m(Y)|_{U\times A'}=(Y-X)|_{U\times A'}\in k[U\times A']^{\times}$;
\item[(ii)]
$F_m(Y)|_{U\times B}=1\in k[U\times B]^{\times}$;
\item[(iii)]
$F_m(Y)|_{U\times S}=(Y-X)|_{U\times S}\in k[U\times S]$.
\end{itemize}
Note that $F_m(Y)|_{(U-S)\times S}\in k[(U-S)\times S]^{\times}$.
Hence by Lemma \ref{Criteria} the morphism $(U\times V, F_{m}; pr_V)
\in Fr_1(U,V)$ being restricted to $U-S$ runs inside $V-S$. Thus,
using Definition \ref{Runs_Inside}, we get a morphism
$$l=\langle\langle U\times V, F_{m}; pr_V\rangle\rangle \in \bb ZF_1((U,U-S)),(V,V-S)).$$
To construct the desired morphism $g$, find a monic in $Y$
polynomial $E_{m-1}\in k[V][Y]$ of degree $m-1$ such that
\begin{itemize}
\item[(i')]
$E_{m-1}(Y)|_{V\times A'}=1|_{U\times A'}\in k[V\times A']^{\times}$;
\item[(ii')]
$E_{m-1}(Y)|_{V\times B}=(Y-X)^{-1}\in k[V\times B]^{\times}$;
\item[(iii')]
$E_{m-1}(Y)|_{V\times S}=1|_{V\times S}\in k[U\times S]$;
\item[(iv')]
$E_{m-1}(Y)|_{\Delta(V)}=1|_{\Delta(V)}\in k[\Delta(V)]$.
\end{itemize}
Let $G\subset V\times \bb A^1_k$ be a closed subset defined by
$E_{m-1}(Y)=0$. By conditions $(i')-(iv')$ one has $G\subset V\times
(V-S)$ and $G\cap \Delta(V)=\emptyset$. Set $ g'=\langle V\times
(V-S) - \Delta(V),(Y-X)E_{m-1}(Y)\rangle; pr_{V-S})\in \bb
ZF_1(V,V-S). $ Since $g'|_{V-S}\in \bb ZF_1(V-S,V-S)$, we get a
morphism
\begin{equation}
\label{g} g=(g',g'|_{V-S}) \in \bb ZF_1((V,V-S)),(V-S,V-S)).
\end{equation}

\begin{claim}
Equality (\ref{Surjectivity_Part}) holds for the morphisms $l$ and
$g$ defined above.
\end{claim}
Note firstly that $l \circ \langle\langle i\rangle\rangle=
\langle\langle V\times V, F_{m}(Y)|_{V\times V};pr_2 \rangle\rangle
\in \bb ZF_1((V,V-S),(V,V-S))$. Applying Corollary
\ref{G0_and_G1_for_pairs} to the case $V\subset \bb A^1$, $S\subset
V$ and $A:=A'\cup B$, we get an equality
$$[[V\times V, F_{m}(Y)|_{V\times V};pr_2]]=[[V\times V, (Y-X)E_{m-1}(Y);pr_2]]$$
in $\overline {\bb ZF}_1((V,V-S),(V,V-S))$. By Lemma
\ref{Criteria_for_pairs} and the fact that $G\cap
\Delta(V)=\emptyset$, one has
$$
[[V\times V, (Y-X)E_{m-1}(Y);pr_2]]=
$$
$$
[[ V\times V-G, E_{m-1}(Y-X);pr_2]] + [[ V\times V-\Delta(V),(Y-X)E_{m-1}; pr_2]]=
$$
$$
=[[ V\times V-G, E_{m-1}(Y-X);pr_2]] + [[j]]\circ [[g]]
$$
in $\overline {\bb ZF}_1((V,V-S)),(V,V-S))$.

One has a chain of equalities
$$[[ V\times V-G, E_{m-1}(Y-X);pr_2]]=[[ V\times V-G,(Y-X);pr_2]]= [[ V\times V,(Y-X);pr_2]]=$$
$$=[[ V\times \bb A^1,Y;pr_1]]=[[\sigma_V]].$$
The first equality holds by condition $(iv')$ and Corollary
\ref{u_and_u'_for_pairs}. The second one is obvious. The third one
is equality (\ref{star}) for $m=1$ from the proof of Proposition
\ref{Diagonal}. The forth one is the definition of $\langle\langle
\sigma_V \rangle\rangle$ (see Definition \ref{sigma} and Notation
\ref{sigma_and_pairs}). Combining altogether, we get a chain of
equalities
$$[[l]] \circ [[i]]=[[V\times V, F_{m}(Y)|_{V\times V};pr_2]]=[[V\times V, (Y-X)E_{m-1}(Y);pr_2]]=[[\sigma_V]]+ [[j]]\circ [[g]],$$
which proves the claim. Whence the theorem.
\end{proof}

\section{Excision on relative affine line}

\begin{proof}[Proof of Theorem \ref{Exc_On_Rel_Aff_Line}]
Let $U=\bb A^1_W$, let $V \subset U$ be the open $V$ from
Theorem \ref{Exc_On_Rel_Aff_Line}.
Let $S=0\times W$. Note that $S\subset V$.
Set $A=\bb A^1_W-U=\emptyset$, $B=U-V=\{f=0\}$.
Then $B$ is finite over $U$, since $f$ is monic.
Note that $B\cap (0\times W)=\emptyset$.

Repeat literally the proof of Theorem \ref{Ex_on_line}
(see Section \ref{Inj_and_Exc_On_Aff_Line}).
\end{proof}

\section{Almost elementary fibrations}
In this section we recall a modification of a result of M. Artin
from~\cite{LNM305} concerning existence of nice neighborhoods. The
following notion
(see \cite[Defn.2.1]{PSV})
is a modification of that introduced by Artin
in~\cite[Exp. XI, D\'ef. 3.1]{LNM305}.
\begin{defs}(\cite{PSV})
\label{DefnElemFib} An \emph{almost elementary fibration\/} over a
scheme $B$ is a morphism of schemes $p:X\to B$ which can be included
in a commutative diagram
\begin{equation}
\label{SquareDiagram_2}
    \xymatrix{
     X\ar[drr]_{q}\ar[rr]^{j}&&
\overline X\ar[d]_{\overline q}&&X_{\infty}\ar[ll]_{i}\ar[lld]^{q_{\infty}} &\\
     && B  &\\    }
\end{equation}
of morphisms satisfying the following conditions:
\begin{itemize}
\item[{\rm(i)}] $j$ is an open immersion dense at each fibre of
$\overline q$, and $X=\overline X-X_{\infty}$;
\item[{\rm(ii)}]
$\overline q$ is smooth projective all of whose fibres are geometrically
irreducible of dimension one;
\item[{\rm(iii)}] $q_{\infty}$ is a finite flat morphism all of whose fibres are non-empty;
\item[\rm(iv)] the morphism $i$ is a closed embedding and
the ideal sheaf
$I_{X_{\infty}} \subset \mathcal O_{\overline X}$ defining the closed subscheme $X_{\infty}$ in $\overline X$
is locally principal.
\end{itemize}
\end{defs}

\begin{prop}[\cite{PSV}]
\label{ArtinsNeighbor} Let $k$ be an infinite field, $X$ be a
smooth geometrically irreducible affine variety over $k$, $x_1,x_2,\dots,x_n\in X$
be closed points. Then there exists a Zariski open neighborhood
$X^0$ of the family $\{x_1,x_2,\dots,x_n\}$ and an almost elementary
fibration $p:X^0\to S$, where $S$ is an open sub-scheme of the
projective space $\mathbb P^{\dim X-1}_k$.
\par
If, moreover, $Z$ is a closed co-dimension one subvariety in $X$,
then one can choose $X^0$ and $p$ in such a way that $p|_{Z\bigcap
X^0}:Z\bigcap X^0\to S$ is finite surjective.
\end{prop}

\begin{prop}[\cite{PSV}]
\label{CartesianDiagramArtin} Let $p: X \to S$ be an almost elementary
fibration. If $S$ is a regular semi-local irreducible scheme, then there
exists a commutative diagram of $S$-schemes
\begin{equation}
\label{RactangelDiagram}
    \xymatrix{
X\ar[rr]^{j}\ar[d]_{\pi}&&\overline X\ar[d]^{\overline \pi}&&
Y\ar[ll]_{i}\ar[d]_{}&\\
\mathbb A^1\times S\ar[rr]^{\text{\rm in}}&&\mathbb P^1\times S&&
\ar[ll]_{i}\{\infty\}\times S &\\  }
\end{equation}
\noindent
such that the left hand side square is Cartesian. Here $j$ and
$i$ are the same as in Definition \ref{DefnElemFib}, while $pr_S
\circ\pi=p$, where $pr_S$ is the projection $\mathbb A^1\times S\to
S$.

In particular, $\pi:X\to\mathbb A^1\times S$ is a finite surjective
morphism of $S$-schemes, where $X$ and $\mathbb A^1\times S$ are
regarded as $S$-schemes via the morphism $p$ and the projection
$pr_S$, respectively.
\end{prop}

\section{Injectivity for local schemes}
The main aim of this section is to prove Theorem~\ref{Local_Inj}.
Let $X \in Sm/k$, $x\in X$ be a point, $U=Spec(\cc O_{X,x})$,
$i: D \hookrightarrow X$ be a closed subset.
Let $j: X-D \hookrightarrow X$ be the open inclusion.
Under the notation of Theorem~\ref{Local_Inj}
we will construct an integer $N$ and a morphism $r\in \bb ZF_N(U,X-D)$ such that
$$[r]\circ [j]= [can]\circ [\sigma^N_U]$$
in $\overline {\bb ZF_N}(U,X)$ (see Definition~\ref{r_for_local_inj}).
For this we need some preparations.

Let $X'\subset X$ be an open subset containing the point $x$ and let
$D'=X'\cap D$. Clearly, if we solve a similar problem for the triple
$U$, $X'$ and $X'-D'$, then we solve the problem for the given
triple $U$, $X$ and $X-D$. So, we may shrink $X$ appropriately. In
particular, we may assume that $X$ is irreducible and
the canonical sheaf $\omega_{X/k}$ is trivial,
i.e. is isomorphic to the sheaf $\cc O_X$. Let $d=dim X$.

Shrinking $X$ further (and replacing $D$ with its trace) and using Propositions
\ref{ArtinsNeighbor} and \ref{CartesianDiagramArtin}, we can find a
commutative diagram of the form
\begin{equation}\label{SquareDiagram}
    \xymatrix{
     \bb A^1\times B \ar[drr]_{pr_B} &&
X\ar[d]_{p}\ar[ll]_{\pi} &&D\ar[dll]^{p|_D}  \ar[ll]_{i}&&\\
     && B  &\\    }
\end{equation}
where $p: X\to B$ is an almost elementary fibration in the sense
of~\cite{PSV}, $B$ is an affine open subset of the projective space
$\bb P^{d-1}_k$, $\pi$ is a finite surjective morphism, $p|_D$ is a
finite morphism.

The canonical sheaf $\omega_{X/k}$ remains trivial. Since $p$
is an almost elementary fibration, then it is a smooth morphism such
that for each point $b\in B$ the fibre $p^{-1}(b)$ is a
$k(b)$-smooth absolutely irreducible affine curve. Since $\pi$ is finite, then the
$B$-scheme $X$ is affine.

Set $U=Spec(\cc O_{X,x})$, $\cc X=U\times_B X$, $\cc D=U\times_B D$.
There is an obvious morphism $\Delta=(id,can): U\to \cc X$. It is a
section of the projection $p_U: \cc X \to U$. Let $p_X: \cc X \to X$
be the projection to $X$.
The base change of diagram (\ref{SquareDiagram}) gives a commutative
diagram of the form
\begin{equation}
\label{SquareDiagram_over_U}
    \xymatrix{
     \bb A^1\times U \ar[drr]_{pr_U} &&
\cc X\ar[d]_{p_U}\ar[ll]_{\Pi} &&\cc D\ar[dll]^{{p_U}|_{\cc D}}  \ar[ll]_{i}&&\\
     && U  &\\    }
\end{equation}
where $p_U: \mathcal X\to U$ is an almost elementary fibration over $U$ in the sense
of Definition \ref{DefnElemFib}, $i$ is a closed embedding, $\Pi$ is a finite surjective morphism, $p_U|_{\mathcal D}$ is a
finite morphism. Since
$\omega_{X/k}$ is trivial and $U$ is local and essentially $k$-smooth,
hence the relative canonical sheaf $\omega_{\cc X/U}$ is trivial,
t.e. isomorphic to the structure sheaf $\mathcal O_{\cc X}$.

\begin{lem}[\cite{OjP}, Lemma 10.1]\label{OjPanin}
Given the commutative diagram (\ref{SquareDiagram_over_U}),
there is a finite surjective morphism
$H_{\theta}=(p_U, h_{\theta}): \cc X \to \bb A^1\times U$ of
$U$-schemes such that for the closed subschemes
$\mathcal D_1:=H^{-1}_{\theta}(1\times U)$ and
$\mathcal D_0:=H^{-1}_{\theta}(0\times U)$ of $\cc X$ one has
\begin{itemize}
\item[(i)]
$\mathcal D_1 \subset \cc X- \cc D$;
\item[(ii)]
$\mathcal D_0=\Delta(U)\sqcup \mathcal D'_0$ (equality of schemes) and $\mathcal D'_0 \subset
\cc X- \cc D$.
\end{itemize}
\end{lem}

Now regard $\cc X$ as an affine $\bb A^1\times U$-scheme via the
morphism $\Pi$. And also regard $\cc X$ as an $X$-scheme via $p_X$.

\begin{lem}\label{Framing_of_X_script}
There are an integer $N\geq 0$, a closed
embedding $\cc X \hookrightarrow \bb A^1\times U\times \bb A^N$ of
$\bb A^1\times U$-schemes, an \'{e}tale affine neighborhood $(\cc V,
\rho: \cc V \to \bb A^1\times U\times \bb A^N, s: \cc X
\hookrightarrow \cc V)$ of $\cc X$ in $\bb A^1\times U\times \bb
A^N$, functions $\varphi_1,...,\varphi_{N} \in k[\cc V]$ and a
morphism $r: \cc V \to \cc X$ such that:
\begin{itemize}
\item[(i)]
the functions $\varphi_1,...,\varphi_{N}$ generate the ideal
$I_{s(\cc X)}$ in $k[\cc V]$ defining the closed subscheme $s(\cc
X)$ of $\cc V$;
\item[(ii)]
$r\circ s = id_{\cc X}$;
\item[(iii)]
the morphism $r$ is a $U$-scheme morphism if $\cc V$ is regarded as
a $U$-scheme via the morphism $pr_U\circ \rho$ and $\cc X$ is
regarded as a $U$-scheme via the morphism $p_U$.
\end{itemize}
\end{lem}

\begin{proof}
Since $\Pi$ is a finite morphism, then for some integer $N\geq 1$
there is a closed embedding of $U$-schemes
$in: \cc X \hookrightarrow \bb A^1\times U\times \bb A^N$.
Consider the short exact sequence of vector bundles on
$\cc X$ defining the normal bundle $\cc N=N_{\bb A^1\times U\times \bb A^N/\cc X}$:
\begin{equation}\label{NormalBundle}
\{0\} \to T_{\cc X/U} \to \bb A^1\times \cc X \times \bb A^{N}=T_{(\bb A^1\times U\times \bb A^N)/U}|_{\cc X}  \xrightarrow{q} \cc N \to \{0\}
\end{equation}
Since $\Pi$ is finite, the scheme $\cc X$ is affine.
As mentioned above the bundle $T_{\cc X/U}$ is trivial. Thus the bundle $\cc N$
is stably trivial. Increasing the integer $N$ we may assume that
the bundle $\cc N$ is trivial.
Hence there is a linear section
$t: \cc N\to \bb A^{1} \times \cc X \times \bb A^N$ of the morphism $q$.
Let $q_{\cc X}: \cc N \to \cc X$ be the projection on $\cc X$.
There are two morphisms of $U$-schemes:
$$in\circ q_{\cc X}: \cc N \to \bb A^1\times U\times \bb A^N \
\text{and} \
(id\times p_U \times id)\circ t: \cc N \to \bb A^1\times U\times \bb A^N.$$
Regarding $\bb A^1\times U\times \bb A^N$ as a vector bundle over $U$
we have a morphism
   $$+: (\bb A^1\times U\times \bb A^N)\times_U (\bb A^1\times U\times \bb A^N) \to \bb A^1\times U\times \bb A^N.$$
Set $\rho'=in\circ q_{\cc X}+(id\times p_U \times id)\circ t: \cc N \to \bb A^1\times U\times \bb A^N$. It is easy to check that
$\rho'$ is \'{e}tale along $s_0(\cc X)$, where $s_0: \cc X \to \cc N$ is the zero section of $\cc N$.
Hence $\rho'$ is \'{e}tale in an affine neighborhood $\cc V'$ of $s_0(\cc X)$.
Since $\rho' \circ s_0=in: \cc X \hookrightarrow \bb A^1\times U\times \bb A^N$,
hence $(\rho')^{-1}(in(\cc X))=s_0(\cc X)\sqcup Y$. Hence there is an open affine subscheme
$\cc V$ in $\cc V'$ containing $s_0(\cc X)$ such that
$(\rho'|_{\cc V})^{-1}(in(\cc X))=s_0(\cc X)$.
Set $\rho=\rho'|_{\cc V}: \cc V \to \bb A^1\times U\times \bb A^N$.
Set $s=s_0: \cc X \to \cc V$.

Clearly, $(\cc V,\rho: \cc V \to \bb A^1\times U\times \bb A^N, s: \cc X \hookrightarrow \cc V)$
is an \'{e}tale neighborhood of $in(\cc X)$ in $\bb A^1\times U\times \bb A^N$.
We will write in this proof $\cc X$ for $in(\cc X)$.

Set $r=(q_{\cc X})|_{\cc V}: \cc V \to \cc X$.
Since the bundle $\cc N$ is trivial we can choose its
trivialization $\cc N\cong \cc X \times \bb A^{N}$.
The trivialization gives functions $\varphi_1,...,\varphi_{N}$ which generate the ideal
$I_{s(\cc X)}$ in $k[\cc V]$ defining the closed subscheme $s_0(\cc X)$ of $\cc V$.

Clearly, $r\circ s=id_{\cc X}$.
Also, the morphism $r$ is a $U$-scheme morphism if $\cc V$ is regarded as
a $U$-scheme via the morphism $pr_U\circ \rho$ and $\cc X$ is
regarded as a $U$-scheme via the morphism $p_U$.
Whence the lemma.
\end{proof}

By Lemma \ref{OjPanin}, $\cc D_0= \Delta(U)\sqcup \cc D'_0$. Set
$\cc V_0=\rho^{-1}(0\times U\times \bb A^N)$ and let $\cc W$ be the
henselization of $\cc V_0$ in $s(\Delta(U))$ (which is the same as
the henselization of $0\times U\times \bb A^N$ in $\Delta(U)$).

\begin{rem}
\label{Deteminant_one} By Lemma \ref{Framing_of_X_script} the
functions ${\varphi_1}|_{\cc W},...,{\varphi_N}|_{\cc W}$ generate
the ideal $I$ defining the closed subscheme $s(\Delta(U))$ of the
scheme $\cc W$. In particular, the family
$$\overline {({\varphi_1}|_{\cc W})},...,\overline {({\varphi_N}|_{\cc W})} \in I/I^2$$
is a free basis of the $k[U]$-module $I/I^2$.
Another free basis of the
$k[U]$-module $I/I^2$
is the family
$$\overline {({t_1-\Delta^*(t_1)})|_{\cc W}},...,\overline {({t_1-\Delta^*(t_N)})|_{\cc W}} \in I/I^2.$$

Let $A\in GL_N(k[\cc W])$ be a unique matrix which converts the
second free basis to the first one and let $J:=det(A)$ be its
determinant. Replacing ${\varphi_1}$ by $J^{-1}\varphi_1$, we may
and will assume below in this section that $J=1 \in k[\cc W]$. This
is useful to apply Theorem \ref{Spriamlenie_1} below.
\end{rem}

Set $\cc V_1=\rho^{-1}(1\times U\times \bb A^N)\cap r^{-1}(\cc X-\cc
D)$. Then $s(\cc D_1)\subset \cc V_1$. In fact, $(r\circ s)(\cc
D_1)=\cc D_1\subset \cc X-\cc D$ and $\rho(\cc D_1)\subset 1\times
U\times \bb A^N$. Thus $\cc V_1\neq \emptyset$.

\begin{constr}[\'{E}tale neighborhood of $\cc D_1$]
\label{Et_neighb_of_D_1}
The morphism
$\rho|_{1\times U\times \bb A^N}: \rho^{-1}(1\times U\times \bb A^N)\to 1\times U\times \bb A^N$
is \'{e}tale
and the inclusion
$i_1: \cc V_1\hookrightarrow \rho^{-1}(1\times U\times \bb A^N)$
is open.
Set $\rho_1=(\rho|_{1\times U\times \bb A^N}) \circ i_1$. Then the triple
$$(\cc V_1, \rho_1: \cc V_1 \to 1\times U\times \bb A^N,s_1=s|_{\cc D_1}: \cc D_1 \to \cc V_1)$$
is an \'{e}tale neighborhood of $\cc D_1$ in $1\times U\times \bb A^N$.
Let $r_1=r|_{\cc V_1}: \cc V_1 \to \cc X-\cc D$.
\end{constr}

\begin{defs}
\label{a_1_for_local_inj}
Set
$a_1=(\cc D_1,\cc V_1,{\varphi_1}|_{\cc V_1},...,{\varphi_N}|_{\cc V_1}; (p_X)|_{\cc X-\cc D}\circ r_1) \in Fr_N(U,X-D)$.
\end{defs}


Set $\cc V'_0=\rho^{-1}(0\times U\times \bb A^N)\cap r^{-1}(\cc X-\cc D)$. Then
$s(\cc D'_0)\subset \cc V'_0$. In fact,
$(r\circ s)(\cc D'_0))=\cc D'_0\subset \cc X-\cc D$
and
$\rho(\cc D'_0)\subset 0\times U\times \bb A^N$.
Thus $\cc V'_0\neq \emptyset$.

\begin{constr}
\label{Et_neighb_of_D'_0}
The morphism
$\rho|_{0\times U\times \bb A^N}: \rho^{-1}(0\times U\times \bb A^N)\to 0\times U\times \bb A^N$
is \'{e}tale
and the inclusion
$i_0: \cc V'_0\hookrightarrow \rho^{-1}(0\times U\times \bb A^N)$
is open.
Set $\rho_0=(\rho|_{1\times U\times \bb A^N}) \circ i_0$. Then the triple
$$(\cc V'_0, \rho_0: \cc V'_0 \to 0\times U\times \bb A^N,s_0=s|_{\cc D'_0}: \cc D'_0 \to \cc V_1)$$
is an \'{e}tale neighborhood of $\cc D'_0$ in $0\times U\times \bb A^N$.
Let $r_0=r|_{\cc V'_0}: \cc V'_0 \to \cc X-\cc D$.
\end{constr}

\begin{defs}
\label{a_0_for_local_inj}
Set
$a_0=(\cc D'_0,\cc V'_0,{\varphi_1}|_{\cc V_0},...,{\varphi_N}|_{\cc V_0}; (p_X)|_{\cc X-\cc D}\circ r_0) \in Fr_N(U,X-D)$.
\end{defs}

\begin{defs}
\label{r_for_local_inj}
Set $r = \langle a_1 \rangle - \langle a_0 \rangle \in \bb ZF_N(U,X-D)$.
\end{defs}

\begin{claim}
\label{} One has an equality $[j]\circ [r]=[can]\circ [\sigma^N_U]
\in \overline {\bb ZF}_N(U,X)$.
\end{claim}

In fact, take an element $h_{\theta}=(\cc X, \cc V,
\varphi_1,...,\varphi_N; p_X\circ r) \in Fr_N(\bb A^1\times U,X)$.
By Lemma \ref{OjPanin} the support of $h_0$ is the closed subset
$\Delta\sqcup \cc D'_0$. Thus by Lemma \ref{Disjoint_Support}
$\langle h_0 \rangle$ is the sum of two summands. Namely,
$$\langle h_0 \rangle= j\circ \langle a_0 \rangle +
\langle \Delta(U),\cc W, {\varphi_1}|_{\cc W},...,{\varphi_N}|_{\cc W}; p_X \circ (r|_{s(\Delta(U)}) \rangle$$
in $\bb ZF_N(U,X)$.
By Remark \ref{Deteminant_one} and Theorem \ref{Spriamlenie_1} for the second summand one has
$$[\Delta(U),\cc W, {\varphi_1}|_{\cc W},...,{\varphi_N}|_{\cc W}; p_X \circ (r|_{s(\Delta(U)})]=
[p_X\circ r|_{s(\Delta(U)} \circ (s\circ \Delta)]\circ [\sigma^N_U]=[can]\circ [\sigma^N_U]$$
in
$\overline {\bb ZF}_N(U,X)$.
Clearly, $h_1= j\circ a_1$
in $Fr_N(U,X)$.
Thus one has a chain of equalities
$$[j]\circ [a_1]=[h_1]=[h_0]=[j]\circ [a_0] + [can]\circ [\sigma^N_U]$$
in $\overline {\bb ZF}_N(U,X)$. Whence the Claim. Whence
Theorem~\ref{Local_Inj}.

\section{Preliminaries for the injective part of the \'{e}tale excision}
\label{Pre_for_Inj_Etale_Exc}
Let
$$\xymatrix{V'\ar[r]\ar[d]&X'\ar^{\Pi}[d]\\
               V\ar[r]&X}$$
be an elementary distinguished square with affine $k$-smooth $X$ and $X'$.
Let $S=X-V$ and $S'=X'-V'$ be closed subschemes equipped with reduced structures.
Let $x\in S$ and $x' \in S'$ be two points such that $\Pi(x')=x$.
Let
$U=Spec(\cc O_{X,x})$
and
$U'=Spec(\cc O_{X',x'})$.
Let $\pi: U' \to U$ be the morphism induced by $\Pi$.

To prove Theorem \ref{Inj_Etale_exc}, it suffices to find morphisms
$a\in \bb ZF_N((U,U-S)),(X',X'-S'))$ and $b_G\in \bb
ZF_N((U,U-S)),(X-S,X-S))$ such that
\begin{equation}
\label{Injectivity_Part_Details}
[[\Pi]]\circ [[a]]-[[j]]\circ [[b_G]] =[[can]]\circ [[\sigma^N_U]]
\end{equation}
in $\overline {\bb ZF}_N(U,U-S)),(X,X-S))$. Here $j: (X-S,X-S) \to
(X,X-S)$ and $can: (U,U-S) \to (X,X-S)$ are inclusions.
In this section we do some preparations to
construct the desired morphisms
$a\in \bb ZF_N((U,U-S)),(X',X'-S'))$ and $b_G\in \bb
ZF_N((U,U-S)),(X-S,X-S))$ in Section~\ref{Proof_of_Etale_Injectivity}
satisfying~\eqref{Injectivity_Part_Details}.

Let $in:X^{\circ}\hookrightarrow X$ and $in':
(X')^{\circ}\hookrightarrow X'$ be open such that
\begin{itemize}
\item[(1)]
$x\in X^{\circ}$,
\item[(2)]
$x'\in (X')^{\circ}$,
\item[(3)]
$\Pi((X')^{\circ})\subset X^{\circ}$,
\item[(4)]
the square
$$\xymatrix{V'\cap (X')^{\circ} \ar[r]\ar[d]& (X')^{\circ}\ar^{\Pi|_{(X')^{\circ}}}[d]\\
               V\cap X^{\circ}\ar[r]&X^{\circ}}$$
is an elementary distinguished square.
\end{itemize}

Suppose morphisms $a^{\circ}\in \bb
ZF_N((U,U-S)),((X')^{\circ},(X')^{\circ}-S'))$, $b^{\circ}_G\in \bb
ZF_N((U,U-S)),(X^{\circ}-S,X^{\circ}-S))$ are such that for the
inclusions $j^{\circ}: (X^{\circ}-S,X^{\circ}-S) \to
(X^{\circ},X^{\circ}-S)$ and $can_{X^{\circ}}: (U,U-S) \to
(X^{\circ},X^{\circ}-S)$ one has
\begin{equation}\label{Injectivity_Part_Details_new}
[[\Pi|_{(X')^{\circ}}]]\circ [[a^{\circ}]]-[[j^{\circ}]]\circ [[b^{\circ}_G]] =[[can_{X^{\circ}}]]\circ [[\sigma^N_{U}]].
\end{equation}
Then the morphisms $a=in' \circ a^{\circ}$ and $b_G=in \circ
b^{\circ}_G$ satisfy~\eqref{Injectivity_Part_Details}. {\it
Thus if we shrink $X$ and $X'$ in such a way that properties
$(1)-(4)$ are fulfilled and find appropriate morphisms $a^Y$ and
$b^Y_G$, then we find $a$ and $b_G$ satisfying~\eqref{Injectivity_Part_Details}}.

\begin{rem}
\label{Srinking_X_and_X'} One way of shrinking $X$ and $X'$ such
that properties $(1)-(4)$ are fulfilled is as follows. Replace $X$
by an affine open $X^{\circ}$ containing $x$ and then replace $X'$
by $(X')^{\circ}=\Pi^{-1}(X^{\circ})$.
\end{rem}

Let $X'_n$ be the normalization of $X$ in $Spec(k(X'))$. Let $\Pi_n:
X'_n \to X$ be the corresponding finite morphism. Since $X'$ is
$k$-smooth it is an open subscheme of $X'_n$. Let $Y''=X'_n - X'$.
It is a closed subset in $X'_n$. Since $\Pi|_{S'}: S' \to S$ is an
isomorphism of schemes, then $S'$ is closed in $X'_n$. Thus $S' \cap Y''
= \emptyset$. Hence there is a function $f\in k[X'_n]$ such that
$f|_{Y''}=0$ and $f|_{S'}=1$.

\begin{defs}
\label{X'_new} Set $X'_{new}=(X'_n)_f$, $Y'=\{f=0\}$,
$Y=\Pi_n(Y'_{red})\subset X$. Note that $X'_{new}$ is an affine
$k$-variety as a principal open subset of the affine $k$-variety
$X'_n$. We regard $Y'$ as an effective Cartier divisor of $X'_n$.
The subset $Y$ is closed in $X$, because $\Pi_n$ is finite. Set
$\Pi_{new}=\Pi|_{X'_{new}}$.
\end{defs}

\begin{rem}
\label{X'andX'_new} We have that $\Pi^{-1}_{new}(S)=S'$. Therefore
the varieties $X$ and $X'_{new}$ are subject to properties $(1)-(4)$
of the present section. Below we will work with this $X'_{new}$.
However, we will write $X'$ for $X'_{new}$.
\end{rem}

\begin{rem}
\label{Elementary_fibr} Take $X$ and $X'$ as in Remark
\ref{X'andX'_new}. Shrinking $X$ and $X'$ as described in Remark
\ref{Srinking_X_and_X'},
and using Proposition \ref{CartesianDiagramArtin}
one can find an almost elementary fibration $q: X\to B$
in the sense of
Definition \ref{DefnElemFib}
(here $B$ is
affine open in $\bb P^{n-1}$) such that $q|_{Y\cup S}: Y\cup S \to B$ is finite,
$\omega_{B/k}\cong \cc O_B$, $\omega_{X/k}\cong \cc O_X$.

The shrunk scheme $X'$ will be regarded below as a $B$-scheme via
the morphism $q\circ \Pi$.
\end{rem}

\begin{rem}
\label{Stable_Normal_Bundle}
If $q: X\to B$ is the almost elementary fibration from
Remark \ref{Elementary_fibr}, then
$\Omega^1_{X/B}\cong \cc O_X$.
In fact, $\omega_{X/k}\cong q^*(\omega_{B/k})\otimes \omega_{X/B}$.
Thus $\omega_{X/B}\cong \cc O_X$. Since $X/B$ is a smooth relative curve,
then $\Omega^1_{X/B}=\omega_{X/B}\cong \cc O_X$.

If, furthermore, $j: X\hookrightarrow B\times \bb A^N$ is a closed
embedding of $B$-schemes, then one has
$[\cc N(j)]=(N-1)[\cc O_X]$ in $K_0(X)$, where $\cc N(j)$ is the normal bundle to $X$
for the imbedding $j$.

Thus by increasing the integer $N$, we may assume that the normal
bundle $\cc N(j)$ is isomorphic to the trivial bundle $\cc
O^{N-1}_X$.
\end{rem}

Repeating arguments from the proof of Lemma~\ref{Framing_of_X_script}
we get the following
\begin{prop}
\label{Framing_Of_X} Let $q: X\to B$ be the almost elementary
fibration from Remark \ref{Elementary_fibr}. Then there are an
integer $N\geq 0$, a closed embedding $X \hookrightarrow B\times \bb
A^N$ of $B$-schemes, an \'{e}tale affine neighborhood $(\cc V, \rho:
\cc V \to B\times \bb A^N, s: X \hookrightarrow \cc V)$ of $X$ in
$B\times \bb A^N$, functions $\varphi_1,...,\varphi_{N-1} \in k[\cc
V]$ and a morphism $r: \cc V \to X$ such that:
\begin{itemize}
\item[(i)]
the functions $\varphi_1,...,\varphi_{N-1}$ generate the ideal
$I_{s(X)}$ in $k[\cc V]$ defining the closed subscheme $s(X)$ of
$\cc V$;
\item[(ii)]
$r\circ s = id_{X}$;
\item[(iii)]
the morphism $r$ is a $B$-scheme morphism if $\cc V$ is regarded as
a $B$-scheme via the morphism $pr_U\circ \rho$, and $X$ is regarded
as a $B$-scheme via the morphism $q$.
\end{itemize}
\end{prop}

\begin{defs}
Let $x\in S$, $x'\in S'$ be such that $\Pi(x')=x$.
Set $U=Spec(\cc O_{X,x})$,
There is an obvious morphism $\Delta=(id,can): U\to U\times_B X$. It
is a section of the projection $p_U: U\times_B X \to U$. Let $p_X:
U\times_B X \to X$ be the projection onto $X$. Let $\pi: U'\to U$ be
the restriction of $\Pi$ to $U'$.
\end{defs}

\begin{notn}
\label{X_script} In what follows we will write $U\times X$ to denote
$U\times_B X$,
$U\times X'$  to denote $U\times_B X'$, $U'\times X'$  to denote
$U'\times_B X'$, etc. Here $X'$ is regarded as a $B$-scheme via the
morphism $q\circ \Pi$.
\end{notn}

\begin{prop}
\label{h_theta}
Under the conditions of Remark \ref{Elementary_fibr}
and Notation \ref{X_script}
there is a function
$h_{\theta}\in k[\bb A^1\times U\times X]$
($\theta$ is the parameter on the left factor $\bb A^1$)
such that the following properties hold for the functions
$h_{\theta}$, $h_1:=h_{\theta}|_{1\times U\times X}$ and $h_0:=h_{\theta}|_{0\times U\times X}$:
\begin{itemize}
\item[(a)]
the morphism $(pr,h_{\theta}): \bb A^1\times U\times X \to \bb
A^1\times U\times \bb A^1$ is finite surjective, and hence the
closed subscheme $Z_{\theta}:=h^{-1}_{\theta}(0)\subset \bb
A^1\times U\times X$ is finite flat and surjective over $\bb A^1
\times U$;
\item[(b)]
for the closed subscheme
$Z_{0}:=h^{-1}_{0}(0)$ one has
$Z_{0}=\Delta(U) \sqcup G$ (an equality of closed subschemes) and $G\subset U\times (X-S)$;
\item[(c)]
the closed subscheme $(id_U \times \Pi)^{*}(h_1)=0$ is a disjoint union of the form
$Z'_1\sqcup Z'_2$ and
$m:=(id_U \times \Pi)|_{Z'_1}$ identifies $Z'_1$ with the closed subscheme $Z_1:=\{h_1=0\}$;
\item[(d)]
$Z_{\theta}\cap \bb A^1\times (U-S)\times S=\emptyset$ or,
equivalently, $Z_{\theta}\cap \bb A^1\times (U-S)\times X \subset
\bb A^1 \times (U-S) \times (X-S)$.
\end{itemize}
\end{prop}

\begin{rem}
\label{Running_in} Item $(d)$ yields the following inclusions:
$Z_{\theta}\cap \bb A^1\times (U-S)\times X \subset \bb A^1 \times
(U-S) \times (X-S)$, $Z_{0}\cap (U-S)\times X \subset (U-S) \times
(X-S)$, and $Z_{1}\cap (U-S)\times X \subset (U-S) \times (X-S)$.
Applying item $(c)$, we get another inclusion: $Z'_{1}\cap
(U-S)\times X' \subset (U-S) \times (X'-S')$.
\end{rem}

\section{Reducing Theorem \ref{Inj_Etale_exc} to Propositions \ref{Framing_Of_X} and \ref{h_theta}}\label{Proof_of_Etale_Injectivity}

In this section we construct the desired morphisms
$a\in \bb ZF_N((U,U-S)),(X',X'-S'))$ and $b_G\in \bb
ZF_N((U,U-S)),(X-S,X-S))$ satisfying the relation~\eqref{Injectivity_Part_Details}.
To construct a morphism $b \in Fr_N(U,X)$, we first construct its
support in $U\times \bb A^N$ for some integer $N$, then we construct
an \'{e}tale neighborhood of the support in $U \times \bb A^N$, then
one constructs a framing of the support in the neighborhood, and
finally one constructs $b$ itself. In the same manner we construct a
morphism $a\in Fr_N(U,X')$ and a homotopy $H\in Fr_N(\bb A^1\times
U,X)$ between $\Pi \circ a$ and $b$. Using the fact that the support
$Z_0$ of $b$ is of the form $\Delta(U)\sqcup G$ with $G\subset
U\times (X-S)$, we get an equality
$$\langle b \rangle=\langle b_1 \rangle + \langle b_2 \rangle$$
in $\bb ZF_N(U,X)$. Then we prove that $[b_1]=[can]\circ
[\sigma^N_U]$ and $[b_2]$ factors through $X-S$. Moreover, we are
able to work with morphisms of pairs.
These will end up with the equality~\eqref{Injectivity_Part_Details}
and will complete the proof of Theorem~\ref{Inj_Etale_exc}
at the very end of the section. We will use systematically the data from
Proposition~\ref{Framing_Of_X} in
this section (the details are given below).

Under the assumptions and notation of
Proposition~\ref{Framing_Of_X}, Lemma \ref{Framing_Of_X} and Remark
\ref{X'andX'_new}, set $\cc V'=X'\times_B \cc V$. So we have a Cartesian
square
$$\xymatrix{\cc V'\ar^{\Pi'}[r]\ar^{r'}[d]&\cc V \ar^{r}[d]\\
               X' \ar^{\Pi}[r]& X,}$$
where $r'$ and $\Pi'$ are the projections to the first and second
factors respectively. The section $s: X \to \cc V$ defines a section
$s'=(id,s): X' \to \cc V'$ of $r'$. {\it For brevity, we will write
below} $U\times \cc V$ to denote $U\times_{B} \cc V$, $U\times \cc
V'$ for $U\times_B \cc V'$, and $id\times \rho$ for $id\times_B
\rho: U\times_B \cc V \to U\times_B \cc (B\times \bb A^n)=U\times
\bb A^N$. Let $p_{\cc V}: U\times \cc V \to \cc V$ be the
projection.

Let $X\subset B\times \bb A^N$ be the closed inclusion from
Proposition \ref{Framing_Of_X}. Taking the base change of the latter
inclusion by means of the morphism $U\to B$, we get a closed
inclusion $U\times X \subset U\times \bb A^N$.

Under the notation from Proposition \ref{Framing_Of_X} and
Proposition \ref{h_theta}, construct now a morphism $b\in
Fr_N(U,X)$. Let $Z_0\subset U\times X$ be the closed subset from
Proposition \ref{h_theta}. Then one has the closed inclusions
$$\Delta(U) \sqcup G=Z_0\subset U\times X \subset U\times \bb A^N.$$
Let $in_0: Z_0\subset U\times X$
be the closed inclusion.
Define an \'{e}tale neighborhood of
$Z_0$ in $U\times \bb A^N$ as follows:
\begin{equation}\label{Ordinary_Etale_for_Z_0}
(U\times \cc V, id\times \rho:
U\times \cc V\to U\times \bb A^N, (id\times s)\circ in_0: Z_0 \to
U\times \cc V).
\end{equation}
We will write
$\Delta(U) \sqcup G=Z_0\subset U\times \cc V$
for $((id\times s)\circ in_0)(Z_0) \subset U\times \cc V$.
Let $f\in k[U\times \cc V]$ be a function such that
$f|_{G}=1$ and $f|_{\Delta(U)}=0$.
Then $\Delta(U)$ is a closed subset of the affine scheme
$(U\times \cc V)_f$.

\begin{defs}
\label{b'_for_etale_exc} Under the notation from Proposition
\ref{Framing_Of_X} and Proposition \ref{h_theta}, set
$$
b^{\prime}=(Z_0,U\times \cc V,p^*_{\cc V}(\varphi_1),...,p^*_{\cc
V}(\varphi_{N-1}), (id\times r)^*(h_0); pr_{X}\circ (id\times r)
)\in Fr_N(U,X).
$$
We will sometimes write below $(Z_0,U\times \cc V,p^*_{\cc
V}(\varphi), (id\times r)^*(h_0); pr_{X}\circ (id\times r) )$ to
denote the morphism $b^{\prime}$.
\end{defs}
To construct the desired morphism $b \in Fr_N(U,X)$, we need to
modify slightly the function $p^*_{\cc V}(\varphi_1)$ in the framing
of $Z_0$. By Proposition \ref{Framing_Of_X} and item $(b)$ of
Proposition \ref{h_theta}, the functions
$$p^*_{\cc V}(\varphi_1),...,p^*_{\cc V}(\varphi_{N-1}), (id\times r)^*(h_0)$$
generate an ideal $I_{(id\times s)(\Delta(U))}$ in $k[(U\times \cc
V)_f]$ defining the closed subscheme $\Delta(U)$ of the scheme
$(U\times \cc V)_f$. Let $t_1,t_2,\dots,t_N \in k[U\times \bb A^N]$
be the coordinate functions. For any $i=1,2,\dots,N$, set
$t'_i=t_i-(t_i|_{\Delta(U)}) \in k[U\times \bb A^N]$. Then the
family
$$(t''_1,t''_2, \dots,t''_N)=(id\times \rho)^*(t'_1),(id\times \rho)^*(t'_2), \dots,(id\times \rho)^*(t'_N)$$
also generates the ideal $I=I_{(id\times s)(\Delta(U))}$ in
$k[(U\times \cc V)_f]$. This holds, because
\eqref{Ordinary_Etale_for_Z_0} is an \'{e}tale neighborhood of $Z_0$
in $U\times \bb A^N$. By Remark~\ref{Stable_Normal_Bundle} the
$k[U]=k[(id\times s)(\Delta(U))]$-module $I/I^2$ is free of rank
$N$. Thus the families $(\bar t''_1,\bar t''_2, \dots,\bar t''_N)$
and $(\overline {p^*_{\cc V}(\varphi_1)},...,\overline {p^*_{\cc
V}(\varphi_{N-1})}, \overline {(id\times r)^*(h_0)})$ are two free
bases of the $k[((id\times s)\circ \Delta)(U))]$-module $I/I^2$. Let
$J \in k[U]^{\times}$ be the Jacobian of a unique matrix $A\in
M_N(k[U])$ which transforms the first free basis to the second one. Set,
   $$\varphi^{new}_1=q^*_U(J^{-1})\varphi_1 \in k[\cc V],$$
where $q_U=pr_U \circ (id\times \rho): \cc V \to U$. Let $A^{new}\in
M_N(k[U])$ be a unique matrix changing the first free basis to the basis
 $$(\overline {p^*_{\cc V}(\varphi^{new}_1)},\overline {p^*_{\cc V}(\varphi_{2})},...,\overline {p^*_{\cc V}(\varphi_{N-1})}, \overline {(id\times r)^*(h_0)}).$$
Then the Jacobian $J^{new}$ of $A^{new}$ is equal to $1$:
\begin{equation}
\label{J_new_equal_one}
J^{new}=1\in k[U]^{\times}.
\end{equation}
We will write
$$(\psi_1,\psi_2,\dots,\psi_{N-1}) \ \ \text{for} \ \ (p^*_{\cc V}(\varphi_1^{new}),p^*_{\cc V}(\varphi_{2}),...,p^*_{\cc V}(\varphi_{N-1})).$$

\begin{defs}
\label{b'_for_etale_exc}
Under the notation from Proposition \ref{Framing_Of_X} and Proposition \ref{h_theta} set
$$
b=(Z_0,U\times \cc V,\psi_1,...,\psi_{N-1}, (id\times r)^*(h_0); pr_{X}\circ (id\times r) )\in
Fr_N(U,X).
$$
For brevity, we will sometimes write
$$b=(Z_0,U\times \cc V,p^*_{\cc V}(\psi), (id\times r)^*(h_0); pr_{X}\circ (id\times r) ).$$
\end{defs}

Under the notation from Proposition \ref{Framing_Of_X} and
Proposition \ref{h_theta} construct now a morphism $a\in Fr_N(U,X)$.
Let $Z_1\subset U\times X$ be the closed subset from Proposition
\ref{h_theta}. Then one has closed inclusions
$$Z_1\subset U\times X \subset U\times \bb A^N.$$
Set $(U\times X')_{\circ}=(U\times X')-Z''_2$ and $(U\times \cc
V')_{\circ}=(id\times r')^{-1}((U\times X')_{\circ})$. Let $in_1:
Z_1\subset U\times X$ and $in'_1: Z'_1\subset (U\times X')_{\circ}$
be closed inclusions. Set,
$$r_{\circ}=(id\times r')|_{(U\times \cc V')_{\circ}}: (U\times \cc V')_{\circ} \to (U\times X')_{\circ}.$$
Using the notation of Proposition \ref{Framing_Of_X}
and Proposition \ref{h_theta} (item (c)), define an
\'{e}tale neighborhood of $Z_1$ in $U\times \bb A^N$ as follows:
\begin{equation}
\label{Refined_Etale_for_Z_1} ((U\times \cc V')_{\circ}, (id\times
\rho)\circ (id\times \Pi'): (U\times \cc V')_{\circ}\to U\times \bb
A^N, (id\times s')\circ in'_1 \circ m^{-1}: Z_1 \to (U\times \cc
V')_{\circ}).
\end{equation}

\begin{defs}
\label{a_for_etale_exc} Under the notation of Proposition
\ref{Framing_Of_X} and Proposition \ref{h_theta} set,
$$
a=(Z_1,(U\times \cc
V')_{\circ},(id\times\Pi')^*(\psi_1),...,(id\times\Pi')^*(\psi_{N-1}),
r_{\circ}^*(id\times \Pi)^*(h_1); pr_{X'}\circ r_{\circ})\in
Fr_N(U,X').
$$
For brevity, we will sometimes write
$$a=(Z_1,(U\times \cc V')_{\circ},(id\times\Pi')^*(\psi), r_{\circ}^*(id\times \Pi)^*(h_1); pr_{X'}\circ r_{\circ}).$$
\end{defs}

Under the notation of Proposition \ref{Framing_Of_X} and Proposition
\ref{h_theta}, let us construct now a morphism $H_{\theta}\in
Fr_N(\bb A^1\times U,X)$. Let $Z_{\theta}\subset \bb A^1\times
U\times X$ be the closed subset from Proposition \ref{h_theta}. Then
one has closed inclusions
$$Z_{\theta}\subset \bb A^1\times U\times X \subset \bb A^1\times U\times \bb A^N.$$
Let
$in_{\theta}: Z_{\theta}\subset \bb A^1\times U\times X$
be the closed inclusion.
Define an \'{e}tale neighborhood of
$Z_{\theta}$ in $\bb A^1\times U\times \bb A^N$ as follows:
\begin{equation}
\label{Ordinary_Etale_for_Z_theta} (\bb A^1\times U\times \cc V,
id\times id\times \rho: \bb A^1\times U\times \cc V\to \bb A^1\times
U\times \bb A^N, (id\times id\times s)\circ in_{\theta}: Z_{\theta}
\to \bb A^1\times U\times \cc V).
\end{equation}

\begin{defs}
\label{H_for_etale_exc} Under the notation of Propositions
\ref{Framing_Of_X} and~\ref{h_theta} we set
$$
H_{\theta}=(Z_{\theta},\bb A^1\times U\times \cc
V,\psi_1,...,\psi_{N-1}, (id\times id\times r)^*(h_{\theta});
pr_{X}\circ (id\times id\times r) )\in Fr_N(\bb A^1\times U,X).
$$
We will sometimes write below $(Z_{\theta},\bb A^1\times U\times \cc
V,\psi, (id\times id\times r)^*(h_{\theta}); pr_{X}\circ (id\times
id\times r) )$ to denote the morphism $H_{\theta}$.
\end{defs}

\begin{lem}
\label{H_a_and_b}
One has equalities $H_0=b$, $H_1=\Pi \circ a$ in $Fr_N(U,X)$.
\end{lem}

\begin{proof}
The first equality is obvious. To check the second one, consider
$$H_1=(Z_1, U\times \cc V, \psi, (id\times r)^*(h_1); pr_{X}\circ (id\times r)) \in Fr_N(U,X).$$
Here we use
$
(U\times \cc V, id\times \rho: U\times \cc V\to U\times \bb A^N,
(id\times s)\circ in_{1}: Z_{1} \to U\times \cc V)
$
as an \'{e}tale neighborhood of $Z_1$ in $U\times \bb A^N$.
Take another \'{e}tale neighborhood of $Z_1$ in $U\times \bb A^N$
$$
((U\times \cc V')_{\circ}, (id\times \rho)\circ (id\times \Pi'): (U\times \cc V')_{\circ}\to U\times \bb A^N,
(id\times s')\circ in'_1 \circ m^{-1}: Z_1 \to (U\times \cc V')_{\circ})
$$
and the morphism $id\times \Pi': (U\times \cc V')_{\circ} \to
U\times \cc V$ regarded as a morphism of \'{e}tale neighborhoods.
Refining the \'{e}tale neighborhood of $Z_1$ in the definition of
$H_1$ by means of that morphism, we get a framed correspondence
$H'_1=H_1$ of level $N$, which has the form
$$(Z_1,(U\times \cc V')_{\circ},(id\times \Pi')^*(\psi),(id\times \Pi')^*(id\times r)^*(h_1); pr_X \circ (id \times r)\circ (id\times \Pi')).$$
Note that
$$(id\times \Pi')^*(id\times r)^*(h_1)=r_{\circ}^*(id\times \Pi)^*(h_1) \ \ \text{and} \ \ pr_X \circ (id \times r)\circ (id\times \Pi')=
\Pi\circ pr_{X'}\circ r_{\circ}.$$ Thus, $H_1=H'_1=\Pi \circ a$ in
$Fr_N(U,X)$.
\end{proof}

The following lemma follows from Lemma \ref{Criteria} and Remark
\ref{Running_in}.
\begin{lem}
\label{a_b_H_and_Pi_runs inside} The morphisms $a|_{U-S}$,
$b|_{U-S}$, $H_{\theta}|_{\bb A^1\times(U-S)}$ and $\Pi|_{X'-S'}$
run inside $X'-S'$, $X-S$, $X-S$ and $X-S$ respectively.
\end{lem}

By the preceding lemma and Definition
\ref{Runs_Inside}
the morphisms $a$, $b$, $H_{\theta}$ and
$\Pi$ define morphisms
$$\langle\langle a \rangle\rangle \in\bb ZF_N((U,U-S),(X',X'-S')),\langle\langle b\rangle\rangle \in\bb ZF_N((U,U-S),(X,X-S)),$$
$$\langle\langle H_{\theta} \rangle\rangle \in\bb ZF_N(\bb A^1\times (U,U-S),(X,X-S)),\langle\langle\Pi\rangle\rangle\in\bb ZF_N((X',X'-S'),(X,X-S)).$$
Lemma \ref{H_a_and_b} and Definition \ref{Runs_Inside} yield equalities
$$\langle\langle \Pi \rangle\rangle \circ \langle\langle a \rangle\rangle = \langle\langle H_1 \rangle\rangle \ \
\text{and} \ \
\langle\langle H_0 \rangle\rangle = \langle\langle b \rangle\rangle$$
in $\bb ZF_N((U,U-S),(X,X-S))$.

\begin{cor}
\label{Pi_a_homotopic_b} One has an equality $[[\Pi]]\circ
[[a]]=[[b]]$ in $\overline {\bb ZF}_N((U,U-S),(X,X-S))$.
\end{cor}

\begin{proof}[Proof of Corollary \ref{Pi_a_homotopic_b}]
In fact, by Corollary
\ref{Critaria_for_H}
one has a chain of equalities
$$[[\Pi]]\circ [[a]]= [[H_1]]=[[H_0]]=[[b]]$$
in $\overline {\bb ZF}_N((U,U-S),(X,X-S))$.
\end{proof}

\begin{proof}[Reducing Theorem \ref{Inj_Etale_exc} to Propositions \ref{Framing_Of_X} and \ref{h_theta}]
The support $Z_0$ of $b$ is the disjoint union $\Delta(U) \sqcup G$.
Thus, by Lemma \ref{Criteria_for_pairs} one has an equality
$$\langle\langle b \rangle\rangle=\langle\langle b_1 \rangle\rangle + \langle\langle b_2 \rangle\rangle$$
in $\bb ZF_N((U,U-S),(X,X-S))$, where
$$b_1=(\Delta(U),(U\times \cc V)_f,\psi_1,...,\psi_{N-1}, (id\times r)^*(h_0); pr_{X}\circ (id\times r) ),$$
$$b_2= (G,(U\times \cc V - \Delta(U),\psi_1,...,\psi_{N-1}, (id\times r)^*(h_0); pr_{X}\circ (id\times r) ).$$
By Proposition \ref{h_theta} one has $G\subset U\times (X-S)$. Thus
$b_2=j\circ b_G$ for an obvious morphism $b_G \in Fr_N(U,X-S)$.
Also,
$$\langle\langle b_2 \rangle\rangle=\langle\langle j \rangle\rangle \circ \langle\langle b_G \rangle\rangle \in ZF_N((U,U-S),(X,X-S)),$$
where $j: (X-S,X-S) \hookrightarrow (X,X-S)$ is a natural inclusion.
By the latter comments and Corollary \ref{Pi_a_homotopic_b} one gets
an equality
$$[[\Pi]]\circ [[a]]- [[j]]\circ [[b_G]]= [[b_1]]$$
in $\overline {\bb ZF}_N((U,U-S),(X,X-S))$. To prove equality
(\ref{Injectivity_Part_Details}), and hence to prove Theorem
\ref{Inj_Etale_exc}, it remains to check that $[[b_1]]=[[can]]\circ
[[\sigma^N_U]]$. Recall that one has equality
(\ref{J_new_equal_one}). Thus the equality $[[b_1]]= [[can]]\circ
[[\sigma^N_U]]$ holds by Theorem \ref{Spriamlenie_1}. This finishes
the proof of Theorem \ref{Inj_Etale_exc}.
\end{proof}

\section{Preliminaries for the surjective part of the \'{e}tale excision}

Let $S\subset X$ and $S'\subset X'$ be closed subsets. Let
$$\xymatrix{V'\ar[r]\ar[d]&X'\ar^{\Pi}[d]\\
               V\ar[r]&X}$$
be an elementary distinguished square with affine $k$-smooth $X$ and $X'$.
Let $S=X-V$ and $S'=X'-V'$ be closed subschemes equipped with reduced structures.
Let $x\in S$ and $x' \in S'$ be two points such that $\Pi(x')=x$.
Let $U=Spec(\cc O_{X,x})$ and $U'=Spec(\cc O_{X',x'})$.
Let $\pi: U' \to U$ be the morphism induced by $\Pi$.

To prove Theorem \ref{Surj_Etale_exc} it suffices to find morphisms
$a\in \bb ZF_N((U,U-S)),(X',X'-S'))$ and $b_G\in \bb ZF_N((U',U'-S')),(X'-S',X'-S'))$
such that in the characteristic different from 2 the following equality holds:
\begin{equation}
\label{Surjectivity_Part_Details}
[[a]]\circ [[\pi]]-[[j]]\circ [[b_G]] =[[can']]\circ [[\sigma^N_{U'}]].
\end{equation}
If the characteristic of $k$ is 2 then the following equality holds in $\overline {\bb ZF}_N(U',U'-S')),(X',X'-S'))$:
\begin{equation}
\label{Surjectivity_Part_Details_char_2}
2\cdot[[a]]\circ [[\pi]]-2\cdot[[j]]\circ [[b_G]] =2\cdot[[can']]\circ [[\sigma^N_{U'}]].
\end{equation}
Here $j:(X'-S',X'-S') \to (X',X'-S')$ and $can': (U',U'-S') \to (X',X'-S')$
are inclusions. In this section we do some preparations to
construct the desired morphisms
$a\in \bb ZF_N((U,U-S)),(X',X'-S'))$ and $b_G\in \bb
ZF_N((U',U'-S')),(X'-S',X'-S'))$ in Section \ref{Proof_of_Etale_Surjectivity}
satisfying~\eqref{Surjectivity_Part_Details} in the characteristic different from 2,
and satisfying~\eqref{Surjectivity_Part_Details_char_2} if the characteristic of $k$ is 2.
The preparations are done independently of the characteristic of
the base field $k$.

Replace $X$ by an affine open neighborhood $in:X^{\circ}
\hookrightarrow X$ of the point $x$. Replace $X'$ by $(X')^{\circ}:=
\Pi^{-1}(X^{\circ})$ and write $in': (X')^{\circ} \hookrightarrow
X'$ for the inclusion. Replace $V$ by $V\cap X^{\circ}$ and $V'$
with $V'\cap (X')^{\circ}$. Let $can'_{\circ}: U'\to (X')^{\circ}$
be the canonical inclusion. Let $j^{\circ}:
((X'^{\circ})-S',(X')^{\circ}-S') \to
((X')^{\circ},(X')^{\circ}-S')$ be an inclusion of pairs. If we find
$$a^{\circ}\in \bb ZF_N((U,U-S)),((X')^{\circ},(X')^{\circ}-S')) \ \ \text{and} \ \ b^{\circ}_G\in \bb ZF_N((U',U'-S')),((X')^{\circ}-S',(X')^{\circ}-S'))$$
such that
\begin{equation}
\label{Surjectivity_Part_Details_new}
[[a^{\circ}]]\circ [[\pi]]-[[j^{\circ}]]\circ [[b^{\circ}_G]] =[[can'_{\circ}]]\circ [[\sigma^N_{U'}]],
\end{equation}
then the morphisms $a=in' \circ a^{\circ}$ and $b_G=in' \circ
b^{\circ}_G$ satisfy condition (\ref{Surjectivity_Part_Details}).

Let $X'_n$ be the normalization of $X$ in $Spec(k(X'))$. Let $\Pi_n:
X'_n \to X$ be the corresponding finite morphism. Since $X'$ is
$k$-smooth it is an open subscheme of $X'_n$. Let $Y''=X'_n - X'$.
It is a closed subset in $X'_n$. Since $\Pi|_{S'}: S' \to S$ is a
scheme isomorphism, then $S'$ is closed in $X'_n$. Thus $S' \cap Y''
= \emptyset$. Hence there is a function $f\in k[X'_n]$ such that
$f|_{Y''}=0$ and $f|_{S'}=1$.

\begin{rem}\label{X'andX'_new_2}
In this section we use agreements and notation from Definition~\ref{X'_new} and Remark~\ref{X'andX'_new}.
In particular, $X'_{new}:=(X'_n)_f$ is an affine $k$-smooth variety and $X'_{new}\subset X'$,
and $S'\subset X'_{new}$. Below we work with $X'_{new}$ and write $X'$ for $X'_{new}$.
\end{rem}

\begin{rem}
\label{Stable_Normal_Bundle_2}
If $q: X\to B$ is the almost elementary fibration from
Remark \ref{Elementary_fibr}, then
$\Omega^1_{X'/B}\cong \cc O_X$.
In fact, by Remark \ref{Stable_Normal_Bundle}
$\Omega^1_{X/B}=\omega_{X/B}\cong \cc O_X$. The morphism $\Pi: X'
\to X$ is \'{e}tale. Thus $\Omega^1_{X'/B}\cong \cc O_{X'}$.

Since $X'$ is an affine $k$-scheme, there is a closed embedding
$j: X'\hookrightarrow B\times \bb A^N$ of $B$-schemes.
Choose and fix such an embedding $j$. Since $X'$ is affine $k$-smooth, hence
$[\cc N(j)]=(N-1)[\cc O_X]$ in $K_0(X')$, where $\cc N(j)$ is the normal bundle to $X'$
associated with the imbedding $j$.

Thus by increasing the integer $N$, we may assume that the normal
bundle $\cc N(j)$ is isomorphic to the trivial bundle $\cc O^{N-1}_{X'}$.
\end{rem}

Repeating arguments from the proof of Lemma~\ref{Framing_of_X_script}
we get the following

\begin{prop}\label{Framing_Of_X'}
Let $q: X\to B$ is the almost elementary
fibration from Remark~\ref{Elementary_fibr} and let $X'=X'_{new}$ be
as in the Remark~\ref{X'andX'_new_2}. Then there are an integer $N\geq
0$, a closed embedding $j:X' \hookrightarrow B\times \bb A^N$ of
$B$-schemes, an \'{e}tale affine neighborhood $(\cc V'', \rho'': \cc
V'' \to B\times \bb A^N, s'': X' \hookrightarrow \cc V'')$ of $X'$
in $B\times \bb A^N$, functions $\varphi'_1,...,\varphi'_{N-1} \in
k[\cc V'']$ and a morphism $r'': \cc V'' \to X'$ such that
\begin{itemize}
\item[(i)]
the functions $\varphi'_1,...,\varphi'_{N-1}$ generate the ideal
$I_{s''(X')}$  in $k[\cc V'']$ defining the closed subscheme
$s''(X')$ of $\cc V''$;
\item[(ii)]
$r''\circ s'' = id_{X'}$;
\item[(iii)]
the morphism $r''$ is a $B$-scheme morphism if $\cc V''$ is regarded
as a $B$-scheme via the morphism $pr_U\circ \rho''$ and $X'$ is
regarded as a $B$-scheme via the morphism $q\circ \Pi$ from Lemma
\ref{Elementary_fibr}.
\end{itemize}
\end{prop}

\begin{defs}
Let $x\in S$, $x'\in S'$ be such that $\Pi(x')=x$. We put
$U=Spec(\cc O_{X,x})$.
There is an obvious morphism $\Delta'=(id,can): U'\to U'\times_B
X'$. It is a section of the projection $p_{U'}: U'\times_B X' \to
U'$. Let $p_{X'}: U'\times_B X' \to X'$ be the projection onto $X'$.
Let $\pi: U'\to U$ be the restriction of $\Pi$ to $U'$.
\end{defs}

\begin{notn}
\label{X_script_2} We regard $X$ as a $B$-scheme via the morphism
$q$ and regard $X'$ as a $B$-scheme via the morphism $q\circ \Pi$.
In what follows we write
$U\times X'$ for $U\times_B X'$, $U'\times X'$ for $U'\times_B X'$,
$\bb A^1\times U'\times X'$ for $\bb A^1\times U'\times_B X'$ etc.
\end{notn}

\begin{prop}
\label{h'_theta}
Under the conditions of Remark \ref{Elementary_fibr}
and Notation \ref{X_script_2}
there are functions
$F\in k[U\times X']$
and
$h'_{\theta}\in k[\bb A^1\times U'\times X']$
($\theta$ is the parameter on the left factor $\bb A^1$)
such that the following properties hold for the functions
$h'_{\theta}$, $h'_1:=h'_{\theta}|_{1\times U'\times X'}$ and $h'_0:=h'_{\theta}|_{0\times U'\times X'}$:
\begin{itemize}
\item[(a)]
the morphism $(pr,h'_{\theta}): \bb A^1\times U'\times X' \to \bb
A^1\times U'\times \bb A^1$ is finite and surjective, hence the
closed subscheme $Z'_{\theta}:=(h'_{\theta})^{-1}(0)\subset \bb
A^1\times U'\times X'$ is finite flat and surjective over $\bb A^1
\times U'$;
\item[(b)]
for the closed subscheme
$Z'_{0}:=(h'_{0})^{-1}(0)$ one has
$Z'_{0}=\Delta'(U') \sqcup G'$ (an equality of closed subschemes) and $G'\subset U'\times (X'-S')$;
\item[(c)]
$h'_1=(\pi \times id_{X'})^{*}(F)$ (we write $Z'_1$ to denote the closed
subscheme $\{h'_1=0\}$);
\item[(d)]
$Z'_{\theta}\cap \bb A^1\times (U'-S')\times S'=\emptyset$ or,
equivalently, $Z'_{\theta}\cap \bb A^1\times (U'-S')\times X'
\subset \bb A^1 \times (U'-S') \times (X'-S')$;
\item[(e)]
the morphism $(pr_U,F): U\times X' \to U\times \bb A^1$ is finite
surjective, and hence the closed subscheme $Z_1:=F^{-1}(0)\subset
U\times X'$ is finite flat and surjective over $U$;
\item[(f)]
$Z_1\cap (U-S)\times S'=\emptyset$ or, equivalently, $Z_1\cap
(U-S)\times X' \subset (U-S)\times (X'-S')$.
\end{itemize}
\end{prop}

\begin{rem}
\label{Running_in_2}
Item $(d)$ yields the following inclusions:
\begin{itemize}
\item[$\diamond$] $Z'_{\theta}\cap \bb A^1\times (U'-S')\times X' \subset \bb A^1 \times (U'-S') \times (X'-S')$;
\item[$\diamond$] $Z'_{0}\cap (U'-S')\times X' \subset (U'-S') \times (X'-S')$;
\item[$\diamond$] $Z'_{1}\cap (U'-S')\times X' \subset (U'-S') \times (X'-S')$.
\end{itemize}
Applying $(f)$, we get another inclusion: $Z_{1}\cap (U-S)\times X'\subset (U-S) \times (X'-S')$.
\end{rem}

\section{Reducing Theorem \ref{Surj_Etale_exc} to Propositions \ref{Framing_Of_X'} and \ref{h'_theta}}\label{Proof_of_Etale_Surjectivity}

{\it We suppose in this section that $S\subset X$ is $k$-smooth}.
In the present section we construct the desired morphisms
$a\in \bb ZF_N((U,U-S)),(X',X'-S'))$ and $b_G\in \bb
ZF_N((U',U'-S')),(X'-S',X'-S'))$ satisfying~\eqref{Surjectivity_Part_Details}
in the characteristic different from 2,
and satisfying~\eqref{Surjectivity_Part_Details_char_2}
if the characteristic equals 2. This construction does not depend on
the characteristic of the base field $k$.

To construct a morphism $a \in Fr_N(U,X')$, we first construct its
support in $U\times \bb A^N$ for some integer $N$, then we construct
an \'{e}tale neighborhood of the support in $U \times \bb A^N$, then
one constructs a framing of the support in the neighborhood and
finally one constructs $a$ itself. In the same fashion we construct
a morphism $b\in Fr_N(U',X')$ and a homotopy $H\in Fr_N(\bb
A^1\times U',X')$ between $a \circ \pi$ and $b$. Using the fact that
the support $Z'_0$ of $b$ is of the form $\Delta'(U')\sqcup G'$ with
$G'\subset U'\times (X'-S')$, we get a relation
$$\langle b \rangle=\langle b_1 \rangle + \langle b_2 \rangle$$
in $\bb ZF_N(U',X')$. Then we prove that $[b_1]=[can']\circ
[\sigma^N_{U'}]$ if $\chr k\not=2$
and $[b_2]$ factors through $X'-S'$.
If $\chr k=2$ we prove that $2\cdot[b_1]=2\cdot ([can']\circ[\sigma^N_{U'}])$
and $[b_2]$ factors through $X'-S'$.
Moreover, we are able to work with morphisms of pairs. In this section we will
use systematically Propositions \ref{Framing_Of_X'}
and \ref{h'_theta} and Notation \ref{X_script_2}.
These will end up with the equalities
(\ref{Surjectivity_Part_Details}),
(\ref{Surjectivity_Part_Details_char_2})
and will complete the proof of Theorem \ref{Surj_Etale_exc}
at the very end of the section (details are given below).

Let $X'\subset B\times \bb A^N$ be the closed inclusion from
Proposition \ref{Framing_Of_X'}. Taking the base change of the
latter inclusion by means of the morphism $U\to B$, we get a closed
inclusion $U\times_B X' \subset U\times_B (B\times \bb A^N)=U\times
\bb A^N$.
Recall (see Notation \ref{X_script_2}) that we regard $X$ as a $B$-scheme via the morphism
$q$ and regard $X'$ as a $B$-scheme via the morphism $q\circ \Pi$.
In what follows we write
$U\times X'$ for $U\times_B X'$, $U'\times X'$ for $U'\times_B X'$,
$U\times \cc V''$ for $U\times_{B} \cc V''$, $U'\times \cc
V''$ for $U'\times_B \cc V''$, and $id\times \rho$ for $id\times_B
\rho: U\times_B \cc V'' \to U\times_B \cc (B\times \bb A^n)=U\times
\bb A^N$. Let $p_{\cc V}: U\times \cc V \to \cc V$ be the
projection.

Under the notation from Proposition \ref{Framing_Of_X'} and
Proposition \ref{h'_theta}, construct now a morphism $b\in
Fr_N(U',X')$. Let $Z'_0\subset U'\times X'$ be the closed subset
from Proposition \ref{h'_theta}. Then one has closed inclusions
$$\Delta'(U') \sqcup G'=Z'_0\subset U'\times X' \subset U'\times \bb A^N.$$
Let $in_0: Z'_0\subset U'\times X'$ be a closed inclusion. Define an
\'{e}tale neighborhood of $Z'_0$ in $U'\times \bb A^N$ as follows:
\begin{equation}
\label{Ordinary_Etale_for_Z'_0} (U'\times \cc V'', id\times \rho'':
U'\times \cc V''\to U'\times \bb A^N, (id\times s'')\circ in_0: Z'_0
\to U'\times \cc V'').
\end{equation}
We will write
$\Delta'(U') \sqcup G'=Z'_0\subset U'\times \cc V''$
for
$((id\times s'')\circ in_0)(Z'_0) \subset U'\times \cc V''$.
Let $f\in k[U'\times \cc V'']$ be a function such that
$f|_{G'}=0$ and $f|_{\Delta'(U')}=1$.
Then $\Delta'(U')$ is a closed subset of the affine scheme $(U'\times \cc V'')_f$.

\begin{defs}
\label{b'_for_etale_exc}
Under the notation from Proposition \ref{Framing_Of_X} and Proposition \ref{h_theta}, set
$$b^{\prime}=(Z'_0,U'\times \cc V'',(\pi\times id)^*(p^*_{\cc V''}(\varphi'_1),...,p^*_{\cc V''}(\varphi'_{N-1})),
 (id\times r'')^*(h'_0); pr_{X'}\circ (id\times r'') )\in Fr_N(U',X').$$
Here $p_{\cc V''}: U\times \cc V'' \to \cc V''$ is the projection.
Below we will sometimes write
$(Z'_0,U'\times \cc V'',(\pi\times
id)^*(p^*_{\cc V''}(\varphi')), (id\times r'')^*(h'_0); pr_{X'}\circ
(id\times r'') )$ to denote the morphism $b^{\prime}$.
\end{defs}

To construct the desired morphism $b \in Fr_N(U',X')$, we slightly
modify the function $p^*_{\cc V''}(\varphi'_1)$ in the framing
of $Z'_0$. By Proposition \ref{Framing_Of_X'} and item $(b)$ of
Proposition \ref{h'_theta}, the functions
$$(\pi\times id)^*(p^*_{\cc V''}(\varphi'_1)),...,(\pi\times id)^*(p^*_{\cc V''}(\varphi'_{N-1})), (id\times r'')^*(h'_0)$$
generate an ideal $I_{(id\times s'')(\Delta'(U'))}$ in $k[(U'\times
\cc V'')_f]$ defining the closed subscheme $\Delta'(U')$ of the
scheme $(U'\times \cc V'')_f$. Let $t_1,t_2,\dots,t_N \in k[U'\times
\bb A^N]$ be the coordinate functions. For any $i=1,2,\dots,N$, set
$t'_i=t_i-(t_i|_{\Delta'(U')}) \in k[U'\times \bb A^N]$. Then the
family
$$(t''_1,t''_2, \dots,t''_N)=(id\times \rho'')^*(t'_1),(id\times \rho'')^*(t'_2), \dots,(id\times \rho'')^*(t'_N)$$
also generates the ideal $I=I_{(id\times s'')(\Delta'(U'))}$  in
$k[(U'\times \cc V'')_f]$. This holds, because
(\ref{Ordinary_Etale_for_Z'_0}) is an \'{e}tale neighborhood of
$Z'_0$ in $U\times \bb A^N$. By Remark \ref{Stable_Normal_Bundle_2}
the $k[U']=k[(id\times s'')(\Delta'(U'))]$-module $I/I^2$ is free of
rank $N$. Thus the families
$$(\bar t''_1,\bar t''_2, \dots,\bar t''_N)
\ \ \text{and} \ \
(\overline {p^*_{\cc V''}(\varphi'_1)},...,\overline {p^*_{\cc V''}(\varphi'_{N-1})}, \overline {(id\times r'')^*(h'_0)})
$$
are two free bases of the $k[((id\times s'')\circ
\Delta')(U'))]$-module $I/I^2$. Let $J \in k[U']^{\times}$ be the
Jacobian of a unique matrix $A\in M_N(k[U'])$ converting the first
free basis to the second one. There is an element $\lambda \in k[U]^{\times}$
such that $\lambda|_{S\cap U}=J|_{S'\cap U'}$ (we identify here
$S'\cap U'$ with $S\cap U$ via the morphism $\pi|_{S'\cap U'}$).
Clearly, $\lambda \in k[U]^{\times}$.

Set,
   $$(\varphi'_1)^{new}=q^*_U(J^{-1})(\varphi'_1) \in k[\cc V''],$$
where $q_U=pr_U \circ (id\times \rho''): \cc V'' \to U$. Let
$A^{new}\in M_N(k[U])$ be a unique matrix which converts the first
free basis to the basis
$$(\overline {p^*_{\cc V''}((\varphi'_1)^{new}}),...,\overline {p^*_{\cc V''}(\varphi'_{N-1})}, \overline {(id\times r'')^*(h'_0)}).$$
Then the Jacobian $J^{new}$ of $A^{new}$ has the property:
\begin{equation}
\label{J_new_henselian}
J^{new}|_{S'\cap U'} =1\in k[S'\cap U'].
\end{equation}
We will write
$(\psi_1,\psi_2,\dots,\psi_{N-1})$ for $(p^*_{\cc V''}((\varphi'_1)^{new}),...,p^*_{\cc V''}(\varphi'_{N-1})).$
\begin{defs}
\label{b_for_etale_exc_surj}
Under the notation from Proposition \ref{Framing_Of_X'} and Proposition \ref{h'_theta} set
$$b=(Z'_0,U'\times \cc V'',(\pi\times id)^*(\psi_1),...,(\pi\times id)^*(\psi_{N-1}),
(id\times r'')^*(h'_0); pr_{X'}\circ (id\times r'') )\in Fr_N(U',X').$$
We will often write for brevity
$b=(Z'_0,U'\times \cc V'',(\pi\times id)^*(\psi), (id\times r'')^*(h'_0); pr_{X'}\circ (id\times r'') ).$
\end{defs}

Under the notation from Proposition \ref{Framing_Of_X'} and
Proposition \ref{h'_theta} construct now a morphism $a\in
Fr_N(U,X')$. Let $Z_1\subset U\times X'$ be the closed subset from
Proposition \ref{h'_theta}. Then one has closed inclusions
$$Z_1\subset U\times X' \subset U\times \bb A^N.$$
Let $in_1: Z_1\subset U\times X$
be the closed inclusion.
Define an \'{e}tale neighborhood of
$Z_1$ in $U\times \bb A^N$ as follows:
\begin{equation}
\label{Ordinary_Etale_for_Z_1_surj}
(U\times \cc V'', id\times \rho'': U\times \cc V''\to U\times \bb A^N,
(id\times s'')\circ in_1: Z_1 \hookrightarrow U\times \cc V'').
\end{equation}

\begin{defs}
\label{a'_for_etale_exc_surj}
Under the notation from Proposition \ref{Framing_Of_X'} and Proposition \ref{h'_theta} set
$$a=(Z_1,U\times \cc V'',\psi_1,...,\psi_{N-1}, (id\times r'')^*(F); pr_{X'}\circ (id_U\times r'') )\in Fr_N(U,X')$$
We will sometimes write $(Z_1,U\times \cc V'',\psi, (id\times
r'')^*(F); pr_{X'}\circ (id_U\times r'') )$ to denote $a$.
\end{defs}

Under the notation from Proposition \ref{Framing_Of_X'} and
Proposition \ref{h'_theta} construct now a morphism $H_{\theta}\in
Fr_N(\bb A^1\times U',X')$.
Recall that under that notation we write
$\bb A^1\times U' \times X'$ for $\bb A^1\times U' \times_B X'$.
Let $Z'_{\theta}\subset \bb A^1\times U' \times X'$
be the closed subset from Proposition~\ref{h'_theta}.
Then one has closed inclusions
   $$Z'_{\theta}\subset \bb A^1\times U'\times X' \subset \bb A^1\times U'\times \bb A^N.$$
Let $in_{\theta}: Z'_{\theta}\subset \bb A^1\times U'\times X'$ be the closed inclusion.
Define an \'{e}tale neighborhood of
$Z'_{\theta}$ in $\bb A^1\times U' \times \bb A^N$ as follows:
\begin{equation}
\label{Ordinary_Etale_for_Z_theta_surj}
(\bb A^1\times U'\times \cc V'', \bb A^1\times U'\times \cc V''\xrightarrow{id\times id\times \rho''} \bb A^1\times U'\times \bb A^N,
(\id\times id\times s'')\circ in_{\theta}: Z'_{\theta} \hookrightarrow \bb A^1\times U'\times \cc V'').
\end{equation}

\begin{defs}
\label{H_for_etale_exc_surj} Under the notation from Proposition
\ref{Framing_Of_X'} and Proposition \ref{h'_theta} set $H_{\theta}$
to be equal to
$$(Z'_{\theta},\bb A^1\times U'\times \cc V'',pr^*((\pi\times id)^*(\psi)),
(id\times id\times r'')^*(h'_{\theta}); pr_{X'}\circ (id\times id\times r'') )\in
Fr_N(\bb A^1\times U',X'),$$
where $pr: \bb A^1\times U'\times \cc V'' \to U'\times \cc V''$ is the projection.
\end{defs}

\begin{lem}
\label{H_a_and_b_surj}
One has equalities $H_0=b$, $H_1=a \circ \pi$ in $Fr_N(U',X')$.
\end{lem}

\begin{proof}
The first equality is obvious. Let us prove the second one. By
Proposition \ref{h'_theta} one has $h'_1=(\pi \times id_{X'})^*(F)$. Thus
one has a chain of equalities in $Fr_N(U',X')$:
$$
a\circ \pi=
(Z'_1,U'\times \cc V'',(\pi \times id_{\cc V''})^*(\psi), (\pi \times id_{\cc V''})^*((id_U\times r'')^*(F));
pr_{X'}\circ (id_{U}\times r'')\circ (\pi \times id_{\cc V''}) )=
$$
$$
(Z'_1,U'\times \cc V'',(\pi \times id_{\cc V''})^*(\psi), (id_{U'}\times r'')^*((\pi \times id_{X'})^*(F));
pr_{X'}\circ (\pi \times id_{X'})\circ (id_{U'}\times r''))=
$$
$$(Z'_1,U'\times \cc V'',(\pi \times id_{\cc V''})^*(\psi), (id_{U'}\times
r'')^*(h'_1); pr_{X'}\circ (id_{U'}\times r''))= H_1.$$
\end{proof}

The following lemma follows from Lemma \ref{Criteria} and Remark
\ref{Running_in_2}.
\begin{lem}
\label{a_b_H_and_Pi_runs inside_surj} The morphisms $a|_{U-S}$,
$b|_{U'-S'}$, $H_{\theta}|_{\bb A^1\times(U'-S')}$ and
$\pi|_{U'-S'}$ run inside $X'-S'$, $X'-S'$, $X'-S'$ and $U-S$
respectively.
\end{lem}

By the preceding lemma and Definition~\ref{Runs_Inside} the morphisms
$a$, $b$, $H_{\theta}$ and $\pi$ define morphisms
$$\langle\langle a\rangle\rangle\in\bb ZF_N((U,U-S),(X',X'-S')),\langle\langle b\rangle\rangle \in\bb ZF_N((U',U'-S'),(X',X-S')),$$
$$\langle\langle H_{\theta}\rangle\rangle\in\bb ZF_N(\bb A^1\times (U',U'-S'),(X',X'-S')),\langle\langle\pi\rangle\rangle\in\bb ZF_N((U',U'-S'),(U,U-S)).$$
Lemma~\ref{H_a_and_b_surj} and Definition \ref{Runs_Inside} yield equalities
$$\langle\langle a \rangle\rangle \circ \langle\langle \pi \rangle\rangle = \langle\langle H_1 \rangle\rangle \ \
\text{and} \ \ \langle\langle H_0 \rangle\rangle = \langle\langle b
\rangle\rangle$$ in $\bb ZF_N((U',U'-S'),(X',X'-S'))$.

\begin{cor}\label{Pi_a_homotopic_b_surj} There is a relation $[[a]]\circ
[[\pi]]=[[b]]$ in $\overline {\bb ZF}_N((U',U'-S'),(X',X'-S'))$.
\end{cor}

\begin{proof}[Proof of Corollary \ref{Pi_a_homotopic_b_surj}]
In fact, by Corollary~\ref{Critaria_for_H}
one has a chain of equalities
   $$[[a]]\circ [[\pi]]= [[H_1]]=[[H_0]]=[[b]]$$
in $\overline {\bb ZF}_N((U',U'-S'),(X',X'-S'))$.
\end{proof}

\begin{proof}[Reducing Theorem \ref{Surj_Etale_exc} to Propositions~\ref{Framing_Of_X'} and~\ref{h'_theta}]
The support $Z_0$ of $b$ is the disjoint union $\Delta'(U') \sqcup
G'$. Thus, by Lemma \ref{Criteria_for_pairs} one has,
$$\langle\langle b \rangle\rangle=\langle\langle b_1 \rangle\rangle + \langle\langle b_2 \rangle\rangle$$
in $\bb ZF_N((U',U'-S'),(X',X'-S'))$, where
$$b_1=(\Delta'(U'),(U'\times \cc V'')_f,\psi_1,...,\psi_{N-1}, (id\times r'')^*(h'_0); pr_{X'}\circ (id\times r'') ),$$
$$b_2= (G',(U'\times \cc V'' - \Delta'(U'),\psi_1,...,\psi_{N-1}, (id\times r'')^*(h'_0); pr_{X'}\circ (id\times r'') )$$
and the function $f$ is defined just above Definition
\ref{b'_for_etale_exc}.
By Proposition \ref{h'_theta} one has $G'\subset U'\times (X'-S')$.
Thus $b_2=j\circ b_{G'}$ for the obvious morphism $b_{G'} \in
Fr_N(U',X'-S')$. Also,
$$\langle\langle b_2 \rangle\rangle=\langle\langle j \rangle\rangle\circ\langle\langle b_{G'}\rangle\rangle\in\bb ZF_N((U',U'-S'),(X',X'-S')),$$
where $j :(X'-S',X'-S') \hookrightarrow (X',X'-S')$ is a natural
inclusion. By the latter comments and Corollary
\ref{Pi_a_homotopic_b_surj} one gets,
\begin{equation}\label{almostly_eq_21}
[[a]]\circ [[\pi]]- [[j]]\circ [[b_{G'}]]= [[b_1]]
\end{equation}
in $\overline {\bb ZF}_N((U',U'-S'),(X',X'-S'))$.
Suppose now $\chr k=2$. Then by Theorem \ref{Spriamlenie_3} one has
$$2\cdot[[b_1]]=2\cdot ([[can']]\circ [[\sigma^N_{U'}]])$$
in $\overline {\bb ZF}_N((U',U'-S'),(X',X'-S'))$.
Hence the equality (\ref{almostly_eq_21})
yields the relation (\ref{Surjectivity_Part_Details_char_2}).
Whence Theorem \ref{Surj_Etale_exc} for $\chr k=2$.

Suppose now $\chr k\not=2$. To prove the equality
(\ref{Surjectivity_Part_Details}), and hence to prove Theorem
\ref{Surj_Etale_exc} in this case, it remains to check that
$[[b_1]]=[[can']]\circ [[\sigma^N_{U'}]]$.

Let $U''=(U')^h_{S'\cap U'}$ be the henzelization of $U'$ at $S'\cap
U'$ and let $\pi': U'' \to U'$ be the structure morphism. Recall
that $S'\cap U'$ is essentially $k$-smooth. Thus the pair $(U'',
S'\cap U')$ is a henselian pair with an essentially $k$-smooth
closed subscheme $S'\cap U'$. Recall that one has equality
(\ref{J_new_henselian}).
If $\chr k\not=2$ then by Theorem~\ref{Spriamlenie_2} one has
an equality $[[b_1]] \circ [[\pi']]=[[can']]\circ [[\pi']] \circ
[[\sigma^N_{U''}]]$ in $\overline {\bb
ZF}_N((U'',U''-S''),(X',X'-S'))$. Since $\pi' \circ
\sigma^N_{U''}=\sigma^N_{U'} \circ \pi'$ one has,
$$[[b_1]] \circ [[\pi']]=[[can']]\circ [[\sigma^N_{U'}]\circ [[\pi']] \in \overline {ZF}_N((U'',U''-S''),(X',X'-S')).$$
Applying Theorem \ref{Inj_Etale_exc} to the morphism $\pi': U'' \to
U'$, we see that for an integer $M\geq 0$ one has an equality
$$[[b_1]]\circ [[\sigma^M_{U'}]] =[[can']]\circ [[\sigma^{M+N}_{U'}]] \in \overline {ZF}_{M+N}((U',U'-S'),(X',X'-S')).$$
Thus,
$$[[a]]\circ [[\pi]]\circ [[\sigma^M_{U'}]] - [[j]]\circ [[b_{G'}]] \circ [[\sigma^M_{U'}]] = [[can']]\circ [[\sigma^{M+N}_{U'}]]
\in \overline {ZF}_{M+N}((U',U'-S'),(X',X'-S')).$$
Since $\pi \circ \sigma^M_{U'}=\sigma^M_{U}\circ \pi$, then we have
that
$$[[a]]\circ [[\sigma^M_{U}]]\circ [[\pi]] - [[j]]\circ [[b_{G'}]] \circ [[\sigma^M_{U'}]] = [[can']]\circ [[\sigma^{M+N}_{U'}]]
\in \overline {\bb ZF}_{M+N}((U',U'-S'),(X',X'-S')).$$ Set
$a_{new}=a \circ \sigma^M_{U}$, $b^{new}_{G'}=b_{G'}\circ
\sigma^M_{U'}$, $N(new)=M+N$. Having these the following
equality holds:
$$[[a_{new}]]\circ [[\pi]]- [[j]]\circ [[b^{new}_{G'}]]=[[can']]\circ [[\sigma^{N(new)}_{U'}]] \in \overline {ZF}_{M+N}((U',U'-S'),(X',X'-S')).$$
The latter equality is of the form
(\ref{Surjectivity_Part_Details}). Whence
Theorem~\ref{Surj_Etale_exc} in the characteristic not 2.
\end{proof}

\section{Three useful theorems}

We refer the reader to~\cite{Gab} or~\cite{FP} for the definition and basic properties of henzelization
of an affine scheme along a closed subscheme.

Let $X,X_1$ be $k$-smooth affine varieties, $Z\subset X$, $Z_1\subset X_1$ be closed subsets.
Let $f: X_1\to X$ be a $k$-morphism such that $Z_1\subset f^{-1}(Z)$.
For an \'{e}tale neighborhood $(W,\pi: W\to X,s:Z\to W)$ of $Z$ in $X$
set $W_1=X_1\times_X W$. Let $\pi_1: W_1\to X_1$ be the projection and let
$s_1=(i_1,f|_{Z_1)}: Z_1 \to W_1$, where $i_1: Z_1 \hookrightarrow X_1$ be the inclusion.
Then $(W_1,\pi_1,s_1)$ is an \'{e}tale neighborhood of $Z_1$ in $X_1$.
Denote by $f_W: W_1\to W$ the projection. Then one has a morphism
$\lim (f_W): \lim_{(W,\pi,s)} W_1 \to \lim_{(W,\pi,s)} W=X^h_Z$. Set,
   \begin{equation}\label{f_h}
    f^h=\lim (f_W)\circ can_f: (X_1)^h_{Z_1} \to X^h_Z,
   \end{equation}
where $can_f: (X_1)^h_{Z_1}\to \lim_{(W,\pi,s)} W_1$ is the canonical morphism.
Clearly, $\rho\circ f^h=f\circ \rho_1$, where  $\rho: X^h_Z\to X$ and
$\rho_1: (X_1)^h_{Z_1}\to X_1$ are the canonical morphisms.

The following properties of the morphism $f^h$ are straightforward:

\begin{enumerate}
\item For any affine $k$-smooth variety $X$ one has $\id^h_X=\id_{X^h_Z}$. If $p: X\to pt$
is the structure map, then for any closed $Z$ in $X$ the morphism
$p^h: X^h_Z \to (pt)^h_{pt}=pt$
is the structure morphism.

\item Given a $k$-morphism $f_1: X_2\to X_1$ of affine $k$-smooth varieties and
a closed subset $Z_2\subset X_2$ with $Z_2\subset f^{-1}_1(Z_1)$
one has
$(f\circ f_1)^h=f^h\circ f^h_1$.


\item If $i: Z\hookrightarrow X$ is the closed inclusion, $Z_1=Z$, then
$Z^h_Z=Z$ and $i^h: Z=Z^h_Z\to X^h_Z$ coincides with the canonical closed inclusion
$s: Z\to X^h_Z$.
\end{enumerate}

The last two properties imply the following property.
Let $W$ be an affine $k$-smooth variety.
Let $p: X \to W$ be an affine $W$-smooth scheme and $i: W\to X$ be a section of the morphism $p$.
Let $(X^h_{i(W)}, \rho: X^h_{t(W)}\to X, s: W\to X^h_{i(W)})$ be {\it the henselization of $X$ at $i(W)$,
}
where
$s: W \to X^h_{i(W)}$ be the canonical section of $\rho$.
Then one has equalities
   $$(i\circ p)^h=i^h \circ p^h=s\circ p^h: X^h_{i(W)} \to X^h_{i(W)}.$$

These observations imply the following

\begin{lem}\label{fhs}
Let $W$ be an affine $k$-smooth variety.
Let $p: X \to W$ be an affine $W$-smooth scheme and $i: W\to X$ be a section of the morphism $p$.
Let $f_s: \mathbb A^1\times X \to X$ be a morphism such that $f_1: X\to X$ is the identity,
$f_0: X\to X$ coincides with the morphism $X\xrightarrow{p} W\xrightarrow{i} X$
and
$(f_s)|_{\mathbb A^1\times W}=pr_W$. Then the morphism
$f^h_s: (\mathbb A^1\times X)^h_{\mathbb A^1\times i(W)} \to X^h_{i(W)}$
defined by~\eqref{f_h} has the following properties:
\begin{enumerate}
\item $(f^h_s)|_{(1\times X)^h_{(1\times i(W))}}: X^h_{i(W)}\to X^h_{i(W)}$ is the identity;
\item $(f^h_s)|_{(0\times X)^h_{(0\times i(W)}}: X^h_{i(W)}\to X^h_{i(W)}$
coincides with the morphism
$X^h_{i(W)}\xrightarrow{p^h} W \xrightarrow{s} X^h_{i(W)}$,
where
$s: W \hookrightarrow X^h_{i(W)}$ is defined just above the lemma.
\end{enumerate}
\end{lem}

\begin{proof}
The first assertion follows from the equalities
   $$\id_{X^h_{i(W)}}=\id^h_X=(f_1)^h=(f_s\circ i_1)^h=f^h_s \circ i^h_1=(f^h_s)|_{(1\times X)^h_{1\times i(W)}}.$$
As mentioned above, one has the equality $i^h=s: W\to X^h_{i(W)}$.
The equalities
   $$s\circ p^h=i^h \circ p^h=(i\circ p)^h=(f_0)^h=(f_s\circ i_0)^h=f^h_s \circ i^h_0=(f^h_s)|_{(0\times X)^h_{0\times i(W)}}$$
yield the second assertion.
\end{proof}

If we take $X=\mathbb A^m_W$, a $W$-morphism $i: W\to \mathbb A^m_W$ and the morphism
$f_s: \mathbb A^1\times \mathbb A^m_W \to \mathbb A^m_W$ mapping
$(s,y)$ to $s\cdot (id_{\mathbb A^m_W}-i\circ pr_W)+i\circ pr_W$,
where $pr_W: \mathbb A^m_W \to W$ is the projection.
Then
$f_s: \mathbb A^1\times X \to X$ satisfies
the hypotheses of Lemma~\ref{fhs}.
Thus Lemma~\ref{fhs} implies the following statement.

\begin{cor}\label{Retraction_0}
The morphism
$H_{s}:=f^h_s: (\mathbb A^1\times \mathbb A^m_W )^h_{\mathbb A^1\times i(W)}\to (\mathbb A^m_W )^h_{i(W)}$
has the following properties:
\begin{itemize}
\item[(a)]
$H_1:=(f^h_s)|_{(1\times \mathbb A^m_W )^h_{1\times i(W)}}: (\mathbb A^m_W )^h_{i(W)} \to (\mathbb A^m_W )^h_{i(W)}$
is the identity morphism;
\item[(b)] $H_0:=(f^h_s)|_{(0\times \mathbb A^m_W )^h_{0\times i(W)}}: (\mathbb A^m_W )^h_{i(W)} \to (\mathbb A^m_W )^h_{i(W)}$
coincides with the composite morphism
$(\mathbb A^m_W )^h_{i(W)}\xrightarrow{p^h} W \xrightarrow{s} (\mathbb A^m_W )^h_{i(W)}$, where
$p^h: (\mathbb A^m_W )^h_{i(W)} \to W$ is the structure morphism and
$s: W \hookrightarrow (\mathbb A^m_W )^h_{i(W)}$ is the canonical map defined just above Lemma
\ref{fhs}.
\end{itemize}
\end{cor}

\begin{thm}
\label{Spriamlenie_1} Let $W$ be an essentially $k$-smooth local
$k$-scheme and let $N\geq 1$ be an integer. Let $i: W \to W\times
\bb A^N$ be a section of the projection $pr_W: W\times \bb A^N \to
W$. Let
$$((W\times \bb A^N)^h_{i(W)},\rho: (W\times \bb A^N)^h_{i(W)}\to W\times \bb A^N, s: W\to (W\times \bb A^N)^h_{i(W)})$$
be the henselization of $W\times \bb A^N$ at $i(W)$ (particularly,
$i=\rho\circ s$). Let $X$ be a $k$-smooth scheme. Suppose
$$\alpha=(i(W), (W\times \bb A^N)^h_{i(W)},\phi_1,\dots,\phi_N;g) \in Fr_N(W,X),$$
is a $N$-framed correspondence such that the functions
$(\phi_1,\dots,\phi_N)$ generate the ideal $I=I_{s(W)}$ of those
functions in $k[(W\times \bb A^N)^h_{i(W)}]$, which vanish on the
closed subset $s(W)$. Let $A\in M_N(k[W])$ be a unique matrix
transforming the free basis $(\overline {t_1-
(t_1|_{i(W)})},\dots,\overline {t_N- (t_N|_{i(W)})})$ of the free
$k[W]$-module $I/I^2$ to the free basis $(\bar \phi_1,\dots,\bar
\phi_N)$ of the same $k[W]$-module. Suppose that the determinant
$J:=det(A)=1 \in k[W]$. Then,
\begin{equation}
\label{Spriamlenie_formula_1} [\alpha]=[g\circ s] \circ
[\sigma^N_W] \in \overline {\bb ZF}_N(W,X).
\end{equation}
If $W^{\circ} \subset W$ is Zariski open and $X^{\circ} \subset X$ is Zariski open
and $g(s(W^{\circ}))\subset X^{\circ}$, then
\begin{equation}
\label{Spriamlenie_formula_1_relative}
[[\alpha]]=[[g\circ s]] \circ [[\sigma^N_W]] \in \overline {\bb ZF}_N((W,W^{\circ}),(X,X^{\circ})).
\end{equation}
\end{thm}

\begin{thm}
\label{Spriamlenie_2}
Suppose $\chr k\neq 2$. Let $W$ be an essentially $k$-smooth local
$k$-scheme and $N\geq 1$ be an integer. Let $Z\subset W$ be a closed
subscheme (essentially $k$-smooth) such that the pair $(W,Z)$ is
henselian. Let $X$ be a $k$-smooth scheme. Let $i: W \to W\times \bb
A^N$, $\alpha \in Fr_N(W,X)$, $A\in M_N(k[W])$, $J:=det(A) \in
k[W]$, $s: W\to (W\times \bb A^N)^h_{i(W)}$ be the same as in Theorem \ref{Spriamlenie_1}. Suppose
that $J|_{Z}=1 \in k[Z]$. Then,
\begin{equation}
\label{Spriamlenie_formula_1_hensel}
[\alpha]=[g\circ s] \circ [\sigma^N_W] \in \overline {\bb ZF}_N(W,X).
\end{equation}
If $W^{\circ} \subset W$ is Zariski open and $X^{\circ} \subset X$ is Zariski open
and $g(s(W^{\circ}))\subset X^{\circ}$, then
\begin{equation}
\label{Spriamlenie_formula_1_hensel_relative} [[\alpha]]=[[g\circ
s]] \circ [[\sigma^N_W]] \in \overline {\bb
ZF}_N((W,W^{\circ}),(X,X^{\circ})).
\end{equation}
\end{thm}

\begin{thm}
\label{Spriamlenie_3}
Suppose that $\chr k=2$. Let $W$ be an essentially $k$-smooth local
$k$-scheme and $N\geq 1$ be an integer.
Let $X$ be a $k$-smooth scheme. Let $i: W \to W\times \bb
A^N$, $\alpha \in Fr_N(W,X)$, $A\in M_N(k[W])$,
$s: W\to (W\times \bb A^N)^h_{i(W)}$ be the same as in Theorem \ref{Spriamlenie_1}.
Suppose that $J:=det(A) \in k[W]^{\times}$. Then,
\begin{equation}
\label{Spriamlenie_formula_2}
2\cdot[\alpha]=2\cdot([g\circ s] \circ [\sigma^N_W]) \in \overline {\bb ZF}_N(W,X).
\end{equation}
If $W^{\circ} \subset W$ is Zariski open and $X^{\circ} \subset X$ is Zariski open
and $g(s(W^{\circ}))\subset X^{\circ}$, then
\begin{equation}
\label{Spriamlenie_formula_2_relative}
2\cdot[[\alpha]]=2\cdot([[g\circ s]] \circ [[\sigma^N_W]]) \in \overline {\bb ZF}_N((W,W^{\circ}),(X,X^{\circ})).
\end{equation}
\end{thm}

To prove these two theorems, we need some technical lemmas.

\begin{lem}
\label{Retraction} Let $W$ be a $k$-smooth affine scheme and let
$\cc W:=(W\times \bb A^N)^h_{W\times 0}$ be the henzelization of
$W\times \bb A^N$ at $W\times 0$. Let $\cc W_{\theta}:= (\bb A^1
\times W\times \bb A^N)^h_{\bb A^1\times W\times 0}$ be the
henzelization of $\bb A^1\times W\times \bb A^N$ at $\bb A^1\times
W\times 0$. Let $f_{\theta}: \bb A^1\times W\times \bb A^N \to W\times \bb A^N$
be a morphism given by $(\theta,w,y)\mapsto (w,\theta\cdot y))$.
Then the morphism
$H_{\theta}=f^h_{\theta}: \cc W_{\theta}\to
\cc W$
has the following properties:
\begin{itemize}
\item[$(a)$] $H_1: \cc W \to \cc W$ is the identity morphism;
\item[$(b)$] $H_0: \cc W \to \cc W$ coincides with the composite morphism
$\cc W \xrightarrow{\rho_0} W\times \bb A^N \xrightarrow{pr_W} W \xrightarrow{s_0} \cc W$,
where $(\cc W, \cc W \xrightarrow{\rho_0} W\times \bb A^N, W \xrightarrow{s_0} \cc W )$
is the henselization of $W\times \bb A^N$ at $W\times \{0\}$.
\end{itemize}

If $W^{\circ} \subset W$ is open, then set $\cc
W^{\circ}:=p^{-1}_{W\times \bb A^N}(W^{\circ} \times \bb A^N)$ and
$\cc W^{\circ}_{\theta}:=p^{-1}_{\bb A^1\times W\times \bb A^N}(\bb
A^1\times W^{\circ}\times \bb A^N)$. In that case $H_{\theta}(\cc
W^{\circ}_{\theta})\subset \cc W^{\circ}$.
\end{lem}

\begin{proof}
This lemma is a partial case of Corollary \ref{Retraction_0}.
\end{proof}

\begin{cor}[of Lemma \ref{Retraction}]
\label{Replacing_g}
Let
$h_{\theta}=(\bb A^1\times W\times 0,\cc W_{\theta},\psi;H_{\theta}) \in Fr_N(\bb A^1\times W,\cc W)$.
Then one has:
\begin{itemize}
\item[$(a)$] $h_1=(W\times 0,\cc W,\psi;id_{\cc W}) \in Fr_N(\bb A^1\times W,\cc W)$;
\item[$(b)$] $h_0=(W\times 0,\cc W,\psi;s_0\circ p_W)=s_0\circ (W\times 0,\cc W,\psi;p_W) \in Fr_N(\bb A^1\times W,\cc W)$,
where $p_W$ is the composite map
$\cc W \xrightarrow{\rho_0} W\times \bb A^N \xrightarrow{pr_W} W$
with $\rho_0$ and $s_0: W\hookrightarrow \cc W$ from the previous lemma.
\end{itemize}
Moreover, if $W^{\circ} \subset W$ is open, then $h_{\theta}|_{\bb
A^1 \times \cc W^{\circ}}$ runs inside $\cc W^{\circ}$.
\end{cor}

\begin{lem}
\label{Matrix} Let $(W\times 0,\cc W,\psi;p_W) \in Fr_N(W,W)$, where
$p_W: \cc W \to W$ is the morphism from Corollary \ref{Replacing_g}.
Let $A_{\theta}\in GL_N(k[W][\theta])$ be a matrix such that
$A_0=id$. Set $A:=A_1 \in GL_N(k[W])$. Take a row
$(\psi',\dots,\psi'_N):=(\psi_1,\dots,\psi_N)\cdot p^*_W(A)$ in
$k[\cc W]$ and take a $N$-framed correspondence
$$h_{\theta}:=(\bb A^1\times W\times 0,\bb A^1\times \cc W, \Psi_{\theta}, p_W\circ pr_{\cc W})\in Fr_N(\bb A^1\times W,W),$$
where
$\Psi_{\theta}=(pr^*_{\cc W}(\psi_1),\dots,pr^*_{\cc W}(\psi_N))\cdot p^*_W(A_{\theta})$
is a row in $k[\bb A^1\times \cc W]$. Then one has:
\begin{itemize}
\item[$(a)$] $h_0= (W\times 0, \cc W, \psi,p_W)$;
\item[$(b)$] $h_1= (W\times 0, \cc W, \psi',p_W)$.
\end{itemize}
Moreover, for any open $W^{\circ} \subset W$ the $N$-framed
correspondence $h_{\theta}|_{\bb A^1\times W^{\circ}}$ runs inside
$W^{\circ}$.
\end{lem}

\begin{lem}
\label{psi_and_sigma} Let $(W\times 0,\cc W,\psi;p_W) \in Fr_N(W,W)$
be as in Lemma \ref{Matrix}. Suppose the functions $\psi_1,\dots,
\psi_N$ generate the ideal $I\subset k[\cc W]$ consisting of all the
functions vanishing on the closed subset $W\times 0$. Furthermore,
suppose that for any $i=1,\dots, N$ one has that $\bar \psi_i=\bar
t_i$ in $I/I^2$. Set $\psi_{\theta,\ i}=(1-\theta)\psi_i+\theta t_i \in
k[\bb A^1\times \cc W]$. Set
$\psi_{\theta}:=(\psi_{\theta,1},\dots,\psi_{\theta,N})$. Set
$$h_{\theta}:=(\bb A^1\times W\times 0,\bb A^1\times \cc W,\psi_{\theta},p_W\circ pr_{\cc W})\in Fr_N(\bb A^1\times W,W).$$
Then one has:
\begin{itemize}
\item[$(a)$] $h_0=(W\times 0,\cc W,\psi;p_W)$;\\
\item[$(b)$] $h_1=(W\times 0,\cc W,t_1,\dots,t_N;p_W)=(W\times 0,W\times \bb A^N,t_1,\dots,t_N;pr_W)=\sigma^N_W$.
\end{itemize}
Moreover, for any open $W^{\circ} \subset W$, the $N$-framed
correspondence $h_{\theta}|_{\bb A^1\times W^{\circ}}$ runs inside
$W^{\circ}$.
\end{lem}

\begin{lem}
\label{Affine_Shift} Let $\alpha=(i(W), (W\times \bb
A^N)^h_{i(W)},\phi_1,\dots,\phi_N;g) \in Fr_N(W,X)$ be a $N$-framed
correspondence, where $i: W \to W\times \bb A^N$, $(W\times \bb
A^N)^h_{i(W)}$ and $s: W \to (W\times \bb A^N)^h_{i(W)}$ be as in
Theorem~\ref{Spriamlenie_1}. Let $T_i: W\times \bb A^N \to W\times
\bb A^N$ be a morphism taking a point $(w,v)$ to the point
$(w,v+i(w))$. Let $T^h_i: \cc W=(W\times \bb A^N)^h_{W\times 0} \to
(W\times \bb A^N)^h_{i(W)}$ be the morphism defined by~\eqref{f_h}. Then one has,
$$[\alpha]=[W\times 0,\cc W, \phi_1 \circ T^h_i,\dots,\phi_N \circ T^h_i;g\circ T^h_i] \in \overline {\bb ZF}_N(W,X).$$
Moreover, if $W^{\circ} \subset W$ is open and if $X^{\circ} \subset
X$ is any open such that $g(s^h(W^{\circ}))\subset X^{\circ}$, then
one has,
$$[[\alpha]]=[[W\times 0,\cc W, \phi_1 \circ T^h_i,\dots,\phi_N \circ T^h_i;g\circ T^h_i]] \in \overline {\bb ZF}_N((W,W^{\circ}),(X,W^{\circ})).$$
\end{lem}

\begin{proof}[Proof of Theorem \ref{Spriamlenie_1}]
Let $\alpha \in Fr_N(W,X)$ be the $N$-framed correspondence from
Theorem~\ref{Spriamlenie_1}. By Lemma \ref{Affine_Shift} one has an
equality in $\overline {\bb ZF}_N(W,X)$
$$[\alpha]=[W\times 0,\cc W, \psi_1,\dots,\psi_N;g\circ T^h_i]=[g\circ T^h_i] \circ [W\times 0,\cc W, \psi_1,\dots,\psi_N;id_{\cc W}],$$
where $\psi_i=\phi_1 \circ T^h_i$. By Corollary \ref{Replacing_g}
one has an equality in $\overline {\bb ZF}_N(W,\cc W)$:
$$[W\times 0,\cc W, \psi_1,\dots,\psi_N;id_{\cc W}]=[s_0]\circ [W\times 0,\cc W,\psi;p_W].$$
Thus one has
$$[\alpha]=[g\circ T^h_i\circ s_0]\circ [W\times 0,\cc W,\psi;p_W]=[g\circ s]\circ [W\times 0,\cc W,\psi;p_W] \in \overline {\bb ZF}_N(W,X).$$
Clearly, the functions $(\psi_1,\dots,\psi_N)$ generate the ideal
$I_0=I_{W\times 0}$ of those functions in $k[\cc W]$, which vanish
on the closed subset $W\times 0$. Let $A'\in M_N(k[W])$ be a unique
matrix, which transforms the free basis
$(\bar{t_1},\dots,\bar{t_N})$ of the free $k[W]$-module $I_0/I^2_0$
to the free basis $(\bar \psi_1,\dots,\bar \psi_N)$ of the same
$k[W]$-module. Clearly, $det(A')=det(A)$. Thus $det(A')=1 \in k[W]$.
The ring $k[W]$ is local. Thus $A'$ belongs to the group of
elementary $N\times N$ matrices over $k[W]$. Hence there is a matrix
$A_{\theta}\in M_N(k[W][\theta])$ such that $A_0=id$ and
$A_1=(A')^{-1} \in GL_N(k[W])$. By Lemma \ref{Matrix} one has an
equality
$$[W\times 0, \cc W, \psi,p_W]=[W\times 0, \cc W, \psi',p_W] \in \overline {\bb ZF}_N(W,W)$$
with the row $\psi'_1,\dots,\psi'_N$ as in Lemma \ref{Matrix}. By
construction, for any $i=1,\dots,N$ the function $\psi'_i$ has the
property: $\bar \psi'_i=\bar t_i$ in $I_0/I^2_0$. By Lemma
\ref{psi_and_sigma} one has an equality
$$[W\times 0, \cc W, \psi',p_W]=[\sigma^N_W] \in \overline {\bb ZF}_N(W,W).$$
Thus,
$$[\alpha]=[g\circ s]\circ [\sigma^N_W] \in \overline {\bb ZF}_N(W,X).$$

If $W^{\circ} \subset W$ is Zariski open and $X^{\circ} \subset X$
is Zariski open and $g(s^h(W^{\circ}))\subset X^{\circ}$, then the
same arguments prove the relation
\begin{equation}
\label{Spriamlenie_formula_1_relative}
[[\alpha]]=[[g\circ s]] \circ [[\sigma^N_W]] \in \overline {\bb ZF}_N((W,W^{\circ}),(X,X^{\circ})).
\end{equation}
Theorem \ref{Spriamlenie_1} is proved.
\end{proof}

Theorems  \ref{Spriamlenie_2} and \ref{Spriamlenie_3} are proved at the end of this section. Some preparations
are necessary for them. We begin with the following obvious

\begin{lem} \label{l:polynomialcycle}
Let $Y$ be a $k$-smooth variety which is not necessarily affine. Let $k[Y]$ be the ring of regular
functions on $Y$. Let $n>0$.
Let $a\in k[Y]^{\times}$. Let $p(t),q(t) \in k[Y][t]$ be two polynomials of degree $n$ with the
leading coefficient $a$. Let
$$(Z(p),Y\times \bb A^1,p(t),pr_Y) \ \text{and} \ (Z(q),Y\times \bb A^1,q(t),pr_Y)\in Fr_1(Y,Y)$$
be two framed correspondences.
Let $Z_s\subset \bb A^1\times Y\times \bb A^1$ be the vanishing locus of the polynomial
$p(t)+s(q(t)-p(t))\in k[Y][s,t]$ (here $s$ is the homotopy parameter). Let
   $$H_s:=(Z_s,\bb A^1\times Y\times \bb A^1,p(t)+s(q(t)-p(t)),pr_{Y}) \in Fr_1(\bb A^1\times Y,Y).$$
Then one have equalities in  $Fr_1(Y,Y)$:
   $$H_0=(Z(p),\bb A^1,p(t),pr_Y) \ \text{and} \ H_1=(Z(q),\bb A^1,q(t),pr_Y) \in Fr_1(Y,Y).$$
\end{lem}

Under the hypotheses of Lemma \ref{l:polynomialcycle} let $Y^0\subset Y$ be an open subset.
Then the framed correspondences
$$(Z(p),Y\times \bb A^1,p(t),pr_Y)|_{Y^0} \ \text{and} \ (Z(q),Y\times \bb A^1,q(t),pr_Y)|_{Y^0}\in Fr_1(Y^0,Y)$$
run inside $Y^0$ in the sense of Definition \ref{Runs_Inside}.
Thus following notation from that definition, they define morphisms
$\langle\langle Z(p),Y\times \bb A^1,p(t),pr_Y \rangle\rangle, \ \ \langle\langle Z(q),Y\times \bb A^1,q(t),pr_Y\rangle\rangle \in \bb ZF_1((Y,Y^0),(Y,Y^0))$.
The framed correspondence $(H_s)|_{\bb A^1\times Y^0}$ runs inside $Y^0$. Hence it defines a morphism
$$\langle\langle H_s \rangle\rangle \in \bb ZF_1(\bb A^1\times (Y,Y^0),(Y,Y^0)).$$
Clearly,
$\langle\langle H_0 \rangle\rangle=\langle\langle Z(p),Y\times \bb A^1,p(t),pr_Y \rangle\rangle$,
$\langle\langle H_1 \rangle\rangle=\langle\langle Z(q),Y\times \bb A^1,q(t),pr_Y \rangle\rangle$
in $\bb ZF_1((Y,Y^0),(Y,Y^0))$.

We have thus proved the following

\begin{lem}\label{l:polynomialcycle_pairs}
Under the notation and the hypotheses of Lemma \ref{l:polynomialcycle} and the notation from Definition
\ref{Runs_Inside} the following equality holds:
$$[[Z(p),Y\times \bb A^1,p(t),pr_Y]]=[[Z(q),Y\times \bb A^1,q(t),pr_Y]]\in \overline {\bb ZF}_1((Y,Y^0),(Y,Y^0)).$$
\end{lem}

Let $Y$ be again a $k$-smooth variety not necessary affine. Let $k[Y]$ be the ring of regular
functions on $Y$. Let $n>0$ and $p(t)=t^nR(t)$, where $R(t)=r_0+r_1t+...+r_Nt^N \in k[Y][t]$
with $r_i\in k[Y]$ be a polinomial such that its leading coefficient is a unit in $k[Y]$ and $r_0\in k[Y]^{\times}$.
Let $U=(Y\times \bb A^1)_{R(t)}\subset Y\times \bb A^1$ be the principal open subset
corresponding to $R(t)$. One has $R(t)=r_0+tR_1(t)$. Consider a polynomial
   $$h(s,t)=sR(t)t^{n}+(1-s)r_0t^{n}\in k[s,t].$$
Then $h(s,t)=t^n\cdot(r_0+t\cdot R_1(t)\cdot s)$.
If $S$ is the vanishing locus of
$ r_0+t\cdot R_1(t)\cdot s$, then $S\cap \bb A^1\times Y\times 0=\emptyset$.
Hence for the zero locus $Z(h)$ of $h$ one has $Z(h)=(\bb A^1\times Y\times 0)\sqcup S$.
Set,
\[
H^R_s:=(\bb A^1 \times Y \times \{0\},(\bb A^1\times U)\setminus S,
sR(t)t^{n}+(1-s)r_0t^{n},pr_Y\circ pr_U)\in Fr_1(\bb A^1\times Y,Y).
\]

The following lemma is inspired by \cite[Lemma 4.13]{AGP}.

\begin{lem}\label{analog_4_10}
Let $Y$ be a $k$-smooth variety which is not necessarily affine. Let $k[Y]$ be the ring of regular
functions on $Y$. Let $a\in k[Y]$, $n>0$ and $q(t)=(t-a)^nQ(t)$, where $Q(t)\in k[Y][t]$
be a polinomial such that its leading coefficient is a unit in $k[Y]$ and $Q(a)\in k[Y]^{\times}$.
Let $U=(Y\times \bb A^1)_{Q(t)}\subset Y\times \bb A^1$ be the principal open subset
corresponding to $Q(t)$. Then there is a framed correspondence
$H^Q_s\in Fr_1(\bb A^1\times Y,Y)$ such that in  $Fr_1(Y,Y)$ one have equalities
$$H^Q_1=(Z(t-a),U,q(t);pr_Y) \ \text{and} \ \ H^Q_0=(Z(t-a),Y\times \bb A^1,Q(a)(t-a)^n;pr_Y).$$
\end{lem}

\begin{proof}
We may assume that $a=0$.
Then take $H^R_s$ just as above with $R(t)=Q(t+a)$. Clearly,
$H^R_1=(Y\times \{0\},U,R(t)t^{n},pr_Y)$ and
$H^R_0=(Y\times \{0\},U,R(0)t^{n},pr_Y)$
in $Fr_1(Y,Y)$, whence the lemma.
\end{proof}

Under the hypotheses of Lemma \ref{analog_4_10} let $Y^0\subset Y$ be an open subset.
Then the framed correspondences
$(Z(t-a),U,q(t);pr_Y)|_{Y^0} \in Fr_1(Y^0,Y)$
and
$(Z(t-a),Y\times \bb A^1,Q(a)(t-a)^n;pr_Y)|_{Y^0} \in Fr_1(Y^0,Y)$
run inside $Y^0$ in the sense of Definition \ref{Runs_Inside}.
Thus following notation from that definition, they define morphisms
$$\langle\langle Z(t-a),U,q(t);pr_Y\rangle\rangle, \ \ \langle\langle Z(t-a),Y\times \bb A^1,Q(a)(t-a)^n;pr_Y\rangle\rangle \in \bb ZF_1((Y,Y^0),(Y,Y^0)).$$

\begin{lem}\label{analog_4_10_pairs}
Under the notation and hypotheses of Lemma \ref{analog_4_10} and the notation from Definition
\ref{Runs_Inside} one has
$$[[Z(t-a),U,q(t);pr_Y]]= [[Z(t-a),Y\times \bb A^1,Q(a)(t-a)^n;pr_Y]] \in \overline {\bb ZF}_1((Y,Y^0),(Y,Y^0)).$$
\end{lem}

\begin{proof}
We may assume that $a=0$. Let $R(t)=Q(t+a)\in k[Y][t]$ be as in the proof of
Lem\-ma~\ref{analog_4_10}. Clearly, the morphism
$H^R_s|_{\bb A^1\times Y^0}: \bb A^1\times Y^0 \to Y$
runs inside $Y^0$ in the sense of Definition \ref{Runs_Inside}.
Hence following that definition, it defines a morphism
$\langle\langle H^R_s \rangle\rangle \in \bb ZF_1(\bb A^1 \times (Y,Y^0),(Y,Y^0))$.
One has equalities in $\bb ZF_1((Y,Y^0),(Y,Y^0))$:
$$\langle\langle H^R_0 \rangle\rangle=\langle\langle Y\times 0,Y\times \bb A^1,R(0)t^n;pr_Y\rangle\rangle \ \text{and} \
\langle\langle H^R_1 \rangle\rangle=\langle\langle Y\times 0,U,R(t)t^n;pr_Y\rangle\rangle,$$
whence the lemma.
\end{proof}

Corollaries \ref{a_t_2_and_t2}, \ref{t_3}, Proposition \ref{t_and_a_2_t}
and their proofs are inspired by \cite[Lemma 7.3]{Nesh}.

\begin{cor}\label{a_t_2_and_t2}
Suppose $\chr k\neq 2$.
Let $\lambda \in k[Y]^{\times}$. Then under the notation from Definition~\ref{Runs_Inside}
$$[[(Y\times 0, Y\times \bb A^1,\lambda\cdot t^2;pr_Y)]]=[[(Y\times 0, Y\times \bb A^1,t^2;pr_Y)]]$$
in $\overline {\bb ZF}_1((Y,Y^0),(Y,Y^0))$.
\end{cor}

\begin{proof}
For brevity we drop $pr_W$ from the notation.
By Lemma \ref{l:polynomialcycle_pairs} one has an equality
$[[(Y\times 0, Y\times \bb A^1,t^2)]]=[[(Y\times 0, Y\times \bb A^1,(t-1)(t+1))]]$
in $\overline {\bb ZF}_1((Y,Y^0),(Y,Y^0))$.
By the additivity relation from Lemma \ref{Criteria_for_pairs} and Lemma \ref{analog_4_10_pairs} one has
$$[[(Y\times 0, Y\times \bb A^1,(t-1)(t+1))]]=[[(Y\times 0, Y\times \bb A^1,2t)]]+[[(Y\times 0, Y\times \bb A^1,-2t)]]$$
in $\overline {\bb ZF}_1((Y,Y^0),(Y,Y^0))$.
Similarly
$[[(Y\times 0, Y\times \bb A^1,\lambda \cdot t^2)]]=[[(Y\times 0, Y\times \bb A^1,\lambda \cdot (t-\lambda^{-1})\cdot (t+\lambda^{-1}))]]$
and
$$[[(Y\times 0, Y\times \bb A^1,\lambda \cdot (t-\lambda^{-1})\cdot (t+\lambda^{-1}))]]=[[(Y\times 0, Y\times \bb A^1,2t)]+[(Y\times 0, Y\times \bb A^1,-2t)]]$$
in $\overline {\bb ZF}_1((Y,Y^0),(Y,Y^0))$, whence the corollary.
\end{proof}

\begin{cor}\label{t_3}
Let $\lambda \in k[Y]^{\times}$. Then under the notation from Definition
\ref{Runs_Inside}
one has an equality in $\overline {\bb ZF}_1((Y,Y^0),(Y,Y^0))$
$$[[(Y\times 0, Y\times \bb A^1,\lambda^2 \cdot t;pr_Y)]]+[[(Y\times 0, Y\times \bb A^1,\lambda \cdot t^2;pr_Y)]]=[[(Y\times 0, Y\times \bb A^1,t^3;pr_Y)]].$$
\end{cor}

\begin{proof}
For brevity we drop $pr_W$ from the notation.
By Lemma \ref{l:polynomialcycle_pairs} one has an equality
$[[(Y\times 0, Y\times \bb A^1,t^3)]]=[[(Y\times 0, Y\times \bb A^1,t^3+\lambda\cdot t^2)]]$
in $\overline {\bb ZF}_1((Y,Y^0),(Y,Y^0))$.
By the additivity relation from Lemma \ref{Criteria_for_pairs} one has an equality in
$\overline {\bb ZF}_1((Y,Y^0),(Y,Y^0))$:
$$[[(Y\times 0, Y\times \bb A^1,t^3+\lambda\cdot t^2)]]=[[(Y\times 0, (Y\times \bb A^1)_{t+\lambda},(t+\lambda)\cdot t^2)]]+
[[\{t+\lambda=0\}, (Y\times \bb A^1)_{t^2},t^2(t+\lambda))]].$$
By Lemma \ref{analog_4_10_pairs} one has equalities in
$\overline {\bb ZF}_1((Y,Y^0),(Y,Y^0))$:
$$[[(Y\times 0, (Y\times \bb A^1)_{t+\lambda},(t+\lambda)\cdot t^2)]]=[[(Y\times 0, (Y\times \bb A^1),\lambda\cdot t^2)]]$$
and
$$[[\{t+\lambda=0\}, (Y\times \bb A^1)_{t^2},t^2(t+\lambda))]]=[[Y\times 0, (Y\times \bb A^1)_{t^2},\lambda^2 \cdot t)]],$$
whence the corollary.
\end{proof}

\begin{prop}\label{t_and_a_2_t}
Suppose $\chr k\neq 2$.
Let $\lambda \in k[Y]^{\times}$. Then under the notation from Definition~\ref{Runs_Inside}
one has an equality in $\overline {\bb ZF}_1((Y,Y^0),(Y,Y^0))$
$$[[(Y\times 0, Y\times \bb A^1,\lambda^2 \cdot t;pr_Y)]]=[[(Y\times 0, Y\times \bb A^1,t;pr_Y)]].$$
\end{prop}

\begin{proof}
For brevity we drop $pr_W$ from the notation.
By Corollary \ref{t_3} one has
$$[[(Y\times 0, Y\times \bb A^1,\lambda^2 \cdot t)]]+[[(Y\times 0, Y\times \bb A^1,\lambda \cdot t^2)]]=
[[(Y\times 0, Y\times \bb A^1,t)]]+[[(Y\times 0, Y\times \bb A^1,t^2)]]$$
in $\overline {\bb ZF}_1((Y,Y^0),(Y,Y^0))$.
Corollary \ref{a_t_2_and_t2} now completes the proof of the proposition.
\end{proof}

\begin{proof}[Proof of Theorem \ref{Spriamlenie_2}]
Repeating literally the proof of Theorem \ref{Spriamlenie_1}, one
gets an equality
\begin{equation}
\label{Spriamlenie_formula_2_hensel_relative} [[\alpha]]=[[g\circ
s]] \circ [[\sigma^{N-1}_W]]\circ [[J\cdot t]] \in \overline {\bb
ZF}_N((W,W^{\circ}),(X,X^{\circ})),
\end{equation}
where $[[J\cdot t]]$ is the class in
$\overline {\bb ZF}_N((W,W^{\circ}),(W,W^{\circ})$
of the element
$\langle\langle J\cdot t \rangle\rangle$ corresponding to the $1$-framed
correspondence
$$(W\times 0, W\times \bb A^1,J\cdot t; pr_W) \in Fr_1(W,W)$$
by means of Definition \ref{Runs_Inside}.
Since the pair $(W,Z)$ is henselian and $J|_Z=1$, and the characteristic of the ground field $k$ is not $2$, then $J=\lambda^2$
for a unit $\lambda\in k[W]^{\times}$.
By Corollary \ref{t_and_a_2_t} one has an equality
$[[J\cdot t]]=[[t]] \in \overline {\bb
ZF}_N((W,W^{\circ}),(W,W^{\circ}))$ and our theorem follows.
\end{proof}

\begin{prop}\label{mu_t_and_a_t_char_2}
Suppose $\chr k=2$.
Let $\mu \in k[Y]^{\times}$. Then under the notation from Definition~\ref{Runs_Inside}
one has an equality in $\overline {\bb ZF}_1((Y,Y^0),(Y,Y^0))$
$$2\cdot[[(Y\times 0, Y\times \bb A^1,\mu \cdot t;pr_Y)]]=2\cdot[[(Y\times 0, Y\times \bb A^1,t;pr_Y)]].$$
\end{prop}

\begin{proof}
For brevity we drop $pr_W$ from the notation.
By Lemma \ref{l:polynomialcycle_pairs} one has an equality
$[[(Y\times 0, Y\times \bb A^1,t^2)]]=[[(Y\times 0, Y\times \bb A^1,t(t+\mu))]]$
in $\overline {\bb ZF}_1((Y,Y^0),(Y,Y^0))$.
By the additivity relation from Lemma \ref{Criteria_for_pairs} and Lemma \ref{analog_4_10_pairs} one has
$$[[(Y\times 0, Y\times \bb A^1,t(t+\mu))]]=[[(Y\times 0, Y\times \bb A^1,\mu t)]]+[[(Y\times 0, Y\times \bb A^1,\mu t)]]=2\cdot[[(Y\times 0, Y\times \bb A^1,\mu t)]]$$
in $\overline {\bb ZF}_1((Y,Y^0),(Y,Y^0))$.
Thus,
$$2\cdot[[(Y\times 0, Y\times \bb A^1,\mu t)]]=[[(Y\times 0, Y\times \bb A^1,t^2)]]=2\cdot[[(Y\times 0, Y\times \bb A^1,t)]],$$
whence the proposition.
\end{proof}

\begin{proof}[Proof of Theorem \ref{Spriamlenie_3}]
Repeating literally the proof of Theorem \ref{Spriamlenie_1}, one
gets an equality
\begin{equation}
\label{Spriamlenie_formula_2_hensel_relative} [[\alpha]]=[[g\circ
s]] \circ [[\sigma^{N-1}_W]]\circ [[J\cdot t]] \in \overline {\bb
ZF}_N((W,W^{\circ}),(X,X^{\circ})),
\end{equation}
where $[[J\cdot t]]$ is the class in
$\overline {\bb ZF}_N((W,W^{\circ}),(W,W^{\circ})$
of the element
$\langle\langle J\cdot t \rangle\rangle$ corresponding to the $1$-framed
correspondence
$(W\times 0, W\times \bb A^1,J\cdot t; pr_W) \in Fr_1(W,W)$
by means of Definition \ref{Runs_Inside}.
By Proposition \ref{mu_t_and_a_t_char_2} one has an equality
$2\cdot[[J\cdot t]]=2\cdot[[t]] \in \overline {\bb
ZF}_N((W,W^{\circ}),(W,W^{\circ}))$, and our theorem follows.
\end{proof}

\section{Construction of $h'_{\theta}$, $F$ and $h_{\theta}$ from Propositions \ref{h'_theta} and \ref{h_theta} }

In the first part of this section we construct the functions
$h'_{\theta}$, $F$ from Proposition \ref{h'_theta} and prove this proposition.
In the second part of this section we construct the function
$h_{\theta}$ from Proposition~\ref{h_theta} and prove this proposition.

Let $X$ and $X'$ be as in Remark \ref{Elementary_fibr} and let
$q: X \to B$ be the almost elementary fibration from
Remark \ref{Elementary_fibr}.
Since $q: X \to B$ is an almost elementary fibration
there is a commutative diagram of the form (see Definition \ref{DefnElemFib})
\begin{equation}
\label{SquareDiagram_2}
    \xymatrix{
     X\ar[drr]_{q}\ar[rr]^{j}&&
\overline X\ar[d]_{\overline q}&&X_{\infty}\ar[ll]_{i}\ar[lld]^{q_{\infty}} &\\
     && B  &\\    }
\end{equation}
with morphisms $j$, $\overline q$, $i$, $q_{\infty}$ subjecting the conditions
(i)--(iv) from Definition \ref{DefnElemFib}.

The composite morphism
$X' \xrightarrow{\Pi} X \xrightarrow{j} \bar X$
is quasi-finite. Let
$\bar X'$ be the normalization of $\bar X$ in
$Spec(k(X'))$.
Let $\bar \Pi: \bar X' \to \bar X$
be the canonical morphism (it is finite and surjective).
Then
$(\bar \Pi)^{-1}(X)$
coincides with the normalization of $X$ in
$Spec(k(X'))$.
Let $f':=f|_{(\bar \Pi)^{-1}(X)}$,
where
$f$ is from Definition
\ref{X'_new}.
Let $Y'=\{f'=0\}$ be the closed subscheme of $(\bar \Pi)^{-1}(X)$.
The morphism
$(q\circ (\bar \Pi|_{\bar \Pi^{-1}(X)}))|_{Y'}: Y' \to B$ is finite,
since $q|_Y: Y \to B$ is finite and
$\bar \Pi$
is finite.
Thus $Y'$ is closed in $\bar X'$.
Since $Y'$ is in $(\bar \Pi)^{-1}(X)$
it has the empty intersection with $X_{\infty}$.
Hence
$$X'= \bar X' - ((\bar \Pi)^{-1}(X_{\infty})\sqcup Y').$$
Both $(\bar \Pi)^{-1}(X_{\infty})$ and $Y'$ are Cartier divisors in
$\bar X'$. The Cartier divisor $(\bar \Pi)^{-1}(X_{\infty})$ is
ample. Thus the Cartier divisor $D':= (\bar
\Pi)^{-1}(X_{\infty})\sqcup Y'$ is ample as well and $(q\circ \bar
\Pi)|_{D'}: D' \to B$ is finite.

Set $\bar \Gamma=\bar X' \xrightarrow{(\bar \Pi,id)} \bar X\times_B
\bar X'$ (the graph of the $B$-morphism $\bar \Pi$). The
projection $\bar X\times_B \bar X' \to \bar X'$ is a smooth
morphism, since $\bar q$ is smooth. The morphism $(\bar \Pi,id)$ is
a section of the projection. Hence $\bar \Gamma$ is a Cartier
divisor in $\bar X\times_B \bar X'$.

Set $\Gamma=pr^{-1}_{\bar X}(U)\subset U\times_B \bar X'$. Then
$\Gamma\subset U\times_B \bar X'$ is a Cartier divisor. The scheme
$U'$ is contained in $\Gamma$ as an open subscheme via the inclusion
$(\pi,can')$, where $can': U' \to X'$ is the canonical morphism. The
composite morphism $pr_U \circ (\pi,can'): U'\to U$ coincides with
$\pi:U'\to U$. Thus $pr_{\bar X}|_{\Gamma}: \Gamma \to U$ is
\'{e}tale at the points of $U'$.

\begin{lem}
\label{Gamma'_Delta'_G'} Set $\Gamma'=U'\times_U \Gamma \subset
U'\times_U U \times \bar X'=U'\times_B \bar X'$. Then
$\Gamma'\subset U'\times_B \bar X'$ is a Cartier divisor. Moreover,
$$\Gamma'=\Delta(U')\sqcup G'$$
and $G'\cap (U' \times_B S')= \emptyset$,
where $S' \subset X'$ is the closed subscheme from Section~\ref{Pre_for_Inj_Etale_Exc}.
\end{lem}

\begin{proof}
Consider the diagonal morphism
$\Delta: U'\to U'\times_B \bar X'$. It lands in $U'\times_U \Gamma=\Gamma'$
and it is a section of the projection $U'\times_U \Gamma \to U'$.
The morphism $\Gamma \to U$ is \'{e}tale at all points of $U'$.
Hence the morphism $\Gamma'=U'\times_U \Gamma \to U'$ is \'{e}tale at all points of
$U'\times_U U'$. Particularly, it is \'{e}tale along the diagonal
$\Delta(U')$.
Hence the morphism $\Delta: U'\to U'\times_U \Gamma$ is \'{e}tale.
Thus $\Delta(U')$ is an open subset of $U'\times_U \Gamma=\Gamma'$.
Since $\Delta(U')$ is a closed subset of $\Gamma'$, hence
$\Gamma'=\Delta(U')\sqcup G'$.

Prove now that $G'\cap (U' \times_B S')= \emptyset$.
For that consider the closed subscheme $S$ of the scheme $X$
from Section~\ref{Pre_for_Inj_Etale_Exc} and recall that
$S$ and $S'$ are reduced schemes and the morphism
$\Pi|_{S'}: S' \to S$ is a scheme isomorphism.
There is a chain of inclusions of subsets
$$(\pi\times id)(G'\cap U'\times_B S')\subset (\pi\times id)(G')\cap U\times_B S'\subset \Gamma\cap (U\times_B S') \subset \Gamma_{(\pi|_{S'\cap U'})},$$
where $\Gamma_{(\pi|_{S'\cap U'})}$ is the graph of $\pi|_{S'\cap U'}: S'\cap U' \to U$.
One has an equality
   $$((\pi\times id)|_{U'\times_B S'})^{-1}(\Gamma_{(\pi|_{S'\cap U'})})=\Delta(S'\cap U').$$
Thus $G'\cap (U'\times_B S') \subset \Delta(S'\cap U')$. Finally,
$$G'\cap (U'\times_B S') \subset G'\cap \Delta(S'\cap U')\subset G' \cap \Delta(U')=\emptyset.$$
\end{proof}

\begin{rem}
\label{Gamma_UtimesS'_DeltaZ} It is easy to check that $\Gamma \cap
U\times_B S' =\delta(S'\cap U')$, where $\delta(s')=(\pi(s'),s')$.
\end{rem}

\begin{defs}
\label{UtimesD_U'timesD} Set $\cc D'=U\times_B D'$ and $\cc
D''=U'\times_U \cc D'=U'\times_B D'$. They are Cartier divisors on
$U\times_B \bar X'$ and $U'\times_B \bar X'$ respectively. Both $\cc
D'$ and $\cc D''$ are finite over $B$.
\end{defs}

Let $s_0\in \Gamma(U\times_B \bar X', \cc L(\cc D'))$ be the
canonical section of the invertible sheaf $\cc L(\cc D')$ (its
vanishing locus is $\cc D'$). Let $s_{\Gamma} \in \Gamma(U\times_B
\bar X', \cc L(\Gamma))$ be the canonical section of the invertible
sheaf $\cc L(\Gamma)$ (its vanishing locus is $\Gamma$). Let
$s_{\Delta(U')}\in \Gamma(U'\times_B \bar X', \cc L(\Delta(U'))$
be the canonical section of the invertible sheaf $\cc
L(\Delta(U'))$ (its vanishing locus is $\Delta(U')$). Let $s_{G'}
\in \Gamma(U'\times_B \bar X', \cc L(G'))$ be the canonical section
of the invertible sheaf $\cc L(G')$ (its vanishing locus is $G'$).

\begin{notn}
Set $I'=\cc L(-D')$, $I''=\cc L(-D'')$. They are the ideal sheaves defining the Cartier
divisors $D'$ and $D''$ respectively. Denote by $J'$ the ideal sheaf defining the closed
subscheme $U\times_B \bar X'$. Denote by $J''$ the ideal sheaf defining the closed
subscheme $U'\times_B \bar X'$.
\end{notn}
By Serre's vanishing theorem there is an integer $n\gg 0$ such that the following cohomology
groups vanish:
$H^1(U\times_B \bar X', J'\otimes \cc L(-\Gamma) \otimes \cc L(n\cc D'))$,
$H^1(U\times_B \bar X', I'\otimes \cc L(n\cc D'))$,
$H^1(U'\times_B \bar X', I''\otimes \cc L(-\Delta(U'))\otimes \cc L(n\cc D''))$,
$H^1(U'\times_B \bar X', J''\otimes \cc L(-\Delta(U'))\otimes \cc L(n\cc D''))$.

The fact that these cohomology groups vanish guaranties the existence of sections
$s_1$ and $t'_0$ from Constructions~\ref{s_1} and~\ref{t_0} below.

\begin{constr}\label{s_1}
Find a section $s_1 \in \Gamma(U\times_B \bar X', \cc L(n\cc D'))$ such that:

$(1)$ $s_1|_{U\times_B S'}=r_1 \otimes (s_{\Gamma}|_{U\times_B S'})$, where
$r_1 \in \Gamma(U\times_B S', \cc L(n\cc D' - \Gamma))|_{U\times_B S'})$
has no zeros;

$(2)$ $s_1|_{\cc D'}$ has no zeros.
\end{constr}

\begin{constr}\label{t_1}
Set $t_1:=(\pi\times id)^*(s_1)\in \Gamma(U'\times_B \bar X', \cc L(n\cc D''))$.
The properties of $t_1$ are as follows:

$(1')$ $t_1|_{U'\times_B S'}=r'_1 \otimes (s_{\Delta(U')}|_{U'\times_B S'})\otimes (s_{G'}|_{U'\times_B S'})\cdot \lambda$,
where $\lambda \in k[U']^{\times}$ and
$r'_1$ equals $((\pi\times id)|_{U'\times_B S'})^*(r_1) \in \Gamma(U'\times_B S', \cc L(n\cc D''- \Gamma')|_{U'\times_B S'})$;

$(2')$ $t_1|_{\cc D''}$ has no zeros.
\end{constr}
The second property of $t_1$ is obvious. To prove the first one recall that
$\Gamma'=\Delta(U')\sqcup G'$ by Lemma \ref{Gamma'_Delta'_G'}. Hence there is
$\lambda \in k[U']^{\times}$ such that
$(\pi\times id)^*(s_{\Gamma})=\lambda\cdot (s_{\Delta(U')}\otimes s_{G'})$.

\begin{constr}
\label{t_0} Construct a section $t_0\in \Gamma(U'\times_B \bar X',
\cc L(n\cc D''))$ of the form $t_0=t'_0\otimes s_{\Delta(U')}$,
where $t'_0\in \Gamma(U'\times_B \bar X', \cc L(n\cc D''-
\Delta(U')))$ satisfies the following conditions:

$(1'')$ $t'_0|_{\cc D''}=(t_1|_{\cc D''})\otimes (s_{\Delta(U')}|_{\cc D''})^{-1}$;

$(2'')$ $t'_0|_{U'\times_B S'}=r'_1 \otimes (s_{G'}|_{U'\times_B S'})\cdot \lambda$,
where $r'_1$ and $\lambda$ are from Construction \ref{t_1}.
\end{constr}

\begin{lem}
\label{t_0andt_1}
The following properties are true:

$(1''')$ $t_0|_{\cc D''}=t_1|_{\cc D''}$ and both sections have no zeros on $\cc D''$;

$(2''')$ $t_0|_{U'\times_B S'}=t_1|_{U'\times_B S'}$ and both
sections have no zeros on $(U'-S')\times_B S'$.
\end{lem}

Indeed, the first equality is obvious. The second one follows from the chain of equalities
$$t_0|_{U'\times_B S'}=(t'_0|_{U'\times_B S'})\otimes (s_{\Delta(U')}|_{U'\times_B S'})=
r'_1\otimes (s_{G'}|_{U'\times_B S'})\otimes (s_{\Delta(U')}|_{U'\times_B S'})\cdot \lambda=t_1|_{U'\times_B S'}.$$

\begin{defs}
\label{h'_theta_and_F} Let $s'_0:=(\pi\times id)^*(s_0) \in
\Gamma(U'\times_B \bar X', \cc L(\cc D''))$. Set,
$$h'_{\theta}=\frac{((1-\theta)t_0+ \theta t_1)|_{\bb A^1\times U'\times_B X'}}{(s'_0)^{\otimes n}|_{\bb A^1\times U'\times_B X'}}\in k[\bb A^1\times U'\times_B X']
\ \ \text{and} \ F=\frac{s_1|_{U\times_B X'}}{(s_0)^{\otimes n}|_{U\times_B X'}}\in k[U\times_B X'].$$
\end{defs}

\begin{proof}[Proof of Proposition \ref{h'_theta}]
Let $prj: \bb A^1 \times U'\times_B \bar X' \to U'\times_B \bar X'$
be the projection. Consider two sections
$(1-\theta)t_0+ \theta t_1$ and $(s'_0)^{\otimes n}$
of the line bundle $prj^*(\cc L(n\cdot \cc D''))$
on $\bb A^1 \times U'\times_B \bar X'$.
By Lemma~\ref{t_0andt_1} these two sections have no common zeros. Thus one has morphism
$$[pr,(1-\theta)t_0+ \theta t_1: (s'_0)^{\otimes n}]: \bb A^1\times U'\times_B \bar X' \to \bb A^1\times U'\times \mathbb P^1,$$
where $pr: \bb A^1\times U'\times_B \bar X'\to \bb A^1\times U'$
is the projection. This morphism is quasi-finite and projective. Hence it is
finite and surjective. It follows that any of its base changes is finite and surjective.
Particularly, the morphism $(pr,h'_{\theta}): \bb A^1\times U'\times X' \to \bb
A^1\times U'\times \bb A^1$ is finite and surjective,
because the closed subset
$\{(s'_0)^{\otimes n}=0\}$
in $\bb A^1\times U'\times_B \bar X'$
coincides with the one $\bb A^1\times \cc D''$.
This proves the assertion (a) of Proposition \ref{h'_theta}.
The assertion (e) of Proposition~\ref{h'_theta} is proved in the same fashion.
Lemma \ref{Gamma'_Delta'_G'} yields the assertion (b).
Lemma \ref{t_0andt_1}(2''') yields the assertion (d).
The assertion (c) follows from the construction of $F$ and $h'_1$.
The property (1) of the section $s_1$ yields the assertion (f), whence the proposition.
\end{proof}

In the rest of the section under the hypotheses of Proposition
\ref{h_theta} we will construct a function $h_{\theta} \in k[\bb
A^1\times U\times_B X]$ and prove Proposition \ref{h_theta}.

Let $X$ and $X'$ be as in Remark \ref{Elementary_fibr} and let $q: X
\to B$ be the almost elementary fibration from Remark
\ref{Elementary_fibr}.
Let $\bar X$, $j: X\to \bar X$ and $X_{\infty}$ and $i: X_{\infty} \to \bar X$ be
as in the diagram (\ref{SquareDiagram_2}). So, they satisfy the conditions
(i)--(iv) from Definition \ref{DefnElemFib}.

The composite morphism $X' \xrightarrow{\Pi}
X \xrightarrow{j} \bar X$ is quasi-finite. Let $\bar X'$ be the
normalization of $\bar X$ in $Spec(k(X'))$. Let $\bar \Pi: \bar X'
\to \bar X$ be the canonical morphism (it is finite and surjective).
Let $X_{\infty}\subset \bar X$ be the Cartier divisor from diagram
(\ref{SquareDiagram_2}). Set $X'_{\infty}:=(\bar
\Pi)^{-1}(X_{\infty})$ (scheme-theoretically). Then $X'_{\infty}$ is
a Cartier divisor on $\bar X'$. Set
$$E:=U\times_B X_{\infty} \ \ \text{and} \ \ E':=U\times_B X'_{\infty}.$$
These are Cartier divisors on $U\times_B \bar X$ and $U\times_B \bar
X'$ respectively and $(id\times \bar \Pi)^*(E)=E'$.

Choose an integer $n\gg 0$. Find a section $r_1\in \Gamma(U\times_B
S, \cc L(nE - \Delta(U)|_{U\times_B S})$ which has no zeros. Let
$s_{\Delta(U)}\in \Gamma(U\times_B \bar X, \cc L(\Delta(U))$ be the
canonical section of the invertible sheaf $\cc L(\Delta(U))$ (its
vanishing locus is $\Delta(U)$). Similar to Construction \ref{s_1}
we are able to do the following

\begin{constr}
\label{s'_1}
Find a section $s'_1\in \Gamma(U\times_B \bar X', \cc L(nE'))$
with the following properties.\\
$(1)$ The Cartier divisor $Z'_1:=\{s'_1=0\}$ has the following properties:\\
$(1a)$ $Z'_1\subset U\times_B X'$;\\
$(1a')$ the Cartier divisor $Z'_1$ is finite and \'{e}tale over $U$;\\
$(1b)$ the morphism $i=(id\times \Pi)|_{Z'_1}: Z'_1 \hookrightarrow
U\times_B X$ is a closed embedding. Denote by
$Z_1$ the closed subscheme $i(Z'_1)$ of the scheme $U\times_B X$.\\
$(2)$ $s'_1|_{U\times_B S'}=(id\times \bar \Pi)^*(s_{\Delta(U)})|_{U\times_B S'}\otimes ((\id\times \Pi)|_{U\times_B S'})^*(r_1)$.
\end{constr}


\begin{lem}
\label{1a_1a'_1b_yield_1c} Properties $(1a)$, $(1a')$ and $(1b)$
yield the following property: \\
$(1c)$ one has a scheme equality $(id\times \bar \Pi)^{-1}(Z_1)=Z'_1\sqcup Z'_2$.
\end{lem}

\begin{proof}
The morphism $id\times \Pi: U\times_B X' \to U\times_B X$ is \'{e}tale. It follows that the morphism
$(id \times \Pi)|_{(id \times \Pi)^{-1}(Z_1)}: (id\times \Pi)^{-1}(Z_1) \to Z_1$ is  \'{e}tale.
Since $i: Z'_1\to Z_1$
is an isomorphism, hence the \'{e}tale morphism $(id\times \Pi)|_{(id\times \Pi)^{-1}(Z_1)}$ has a section,
whose image is $Z'_1$. Thus $(id\times \Pi)^{-1}(Z_1)=Z'_1\sqcup Z'_2$.
The property $(1a')$ shows now that $(id\times \bar \Pi)^{-1}(Z_1)=Z'_1\cup \bar Z'_2$,
where $\bar Z'_2$ is the closure of $Z'_2$ in $U\times \bar X'$.
One has $Z'_1\cap \bar Z'_2\subset (U\times_B X')\cap \bar Z'_2=Z'_2$.
Hence $Z'_1\cap \bar Z'_2\subset Z'_1\cap Z'_2=\emptyset$.
Thus
$(id\times \bar \Pi)^{-1}(Z_1)=Z'_1\sqcup \bar Z'_2$.
\end{proof}

\begin{lem}\label{Z'_1_and_UtimesS'}
One has $Z'_2\cap U\times_B S'=\emptyset$.
\end{lem}

\begin{proof}
One has a chain of inclusions
$$(id\times \bar \Pi)((U\times_B S')\cap Z'_2)\subset (id\times \bar \Pi)((U\times_B S')\cap (id\times \bar \Pi)^{-1}(Z_1))=
(U\times_B S)\cap Z_1.$$
This inclusion and the fact that the morphism
$(id\times \Pi)|_{U\times_B S'}: U\times_B S' \to U\times_B S$ is an isomorphism
yield the following inclusions
$(U\times_B S')\cap Z'_2\subset (U\times_B S')\cap Z'_1\subset Z'_1$.
Since $U\times_B S' \cap Z'_2\subset Z'_2$, hence
$U\times_B S' \cap Z'_2\subset Z'_1\cap Z'_2=\emptyset$
by Lemma \ref{1a_1a'_1b_yield_1c}.
\end{proof}

Note that the Cartier divisor $Z_1$ in $U\times_B \bar X$ is
equivalent to the Cartier divisor $dnE$, where $d=[k(X'): k(X)]$.
Let $s_1\in \Gamma(U\times_B \bar X, \cc L(Z_1))$ be the canonical
section (its vanishing locus is $Z_1$). By property $(1c)$ from
Construction \ref{s'_1} one has an equality
\begin{equation}
\label{s_1_s'_1_s'_2}
(id\times \bar \Pi)^*(s_1)=(s'_1 \otimes s'_2)\cdot \mu,
\end{equation}
where $\mu \in k[U]^{\times}$ and
$s'_2\in \Gamma(U\times_B \bar X', \cc L(Z'_2))$
is the canonical section of the line bundle
$\cc L(Z'_2)$.

\begin{defs}\label{t_1_and_s_1}
Set $t_1=s_1 \in \Gamma(U\times_B \bar X, \cc L(Z_1))=\Gamma(U\times_B \bar X, \cc L(dnE))$.
\end{defs}

Similar to Construction \ref{t_0} we are able to do the following

\begin{constr}\label{t_0_and_s_Delta_and_t'0}
Construct a section $t_0\in
\Gamma(U\times_B \bar X, \cc L(dnE))$ of the form
$t_0=s_{\Delta(U)}\otimes t'_0$, where $t'_0\in \Gamma(U\times_B
\bar X, \cc L(dnE - \Delta(U)))$ and $s_{\Delta(U)} \in
\Gamma(U\times_B \bar X, \cc L(\Delta(U)))$ is the canonical section
(its vanishing locus is $\Delta(U)$ and
$t'_0$ has the following properties:\\
$(1')$ $t'_0|_{E}=(t_1|_E) \otimes (s_{\Delta(U)}|_E)^{-1}$;\\
$(2')$ $((id\times \bar \Pi)|_{U\times_B S'})^*(t'_0|_{U\times_B
S})= ((\id\times \Pi)|_{U\times_B S'})^*(r_1)\otimes
(s'_2|_{U\times_B S'})\cdot (\mu|_{U\times_B S'})$, where
$r_1$ is from Construction \ref{s'_1} and $s'_2$, $\mu \in
k[U]^{\times}$ are defined just above
(since $U\times_B S'\cong
U\times_B S$, then condition $(2')$ on $t'_0$ is a condition on
$t'_0|_{U\times_B S}$).
\end{constr}

\begin{lem}\label{t_0andt_1_for_UtimesbarX}
The following statements are true: \\
$(1'')$ $t_0|_{E}=t_1|_{E}$ and both sections have no zeros on $E$;\\
$(2'')$ $t_0|_{U\times_B S}=t_1|_{U\times_B S}$ and both sections have no zeros on $(U-S)\times_B S$;\\
$(3''')$ $t'_0|_{U\times_B S}$ has no zeros.
\end{lem}

\begin{proof}
Indeed, the first equality is obvious. To prove the second one, it suffices to prove the equality
$$((id\times \bar \Pi)|_{U\times_B S'})^*(t_0|_{U\times_B S})=((id\times \bar \Pi)|_{U\times_B S'})^*(t_1|_{U\times_B S}).$$
This equality is a consequence of the following chain of equalities:
$$
((id\times \bar \Pi)|_{U\times_B S'})^*(t_0|_{U\times_B S})=
(id\times \bar \Pi)^*(s_{\Delta(U)})|_{U\times_B S'}\otimes ((\id\times \Pi)|_{U\times_B S'})^*(r_1)\otimes (s'_2|_{U\times_B S'})\cdot (\mu|_{U\times_B S'})=
$$
$$
=s'_1|_{U\times_B S'}\otimes (s'_2|_{U\times_B S'})\cdot (\mu|_{U\times_B S'})=((id\times \bar \Pi)|_{U\times_B S'})^*(t_1|_{U\times_B S}).
$$
The first equality holds by property $(2')$ from Construction
\ref{t_0_and_s_Delta_and_t'0}, the second equality holds by property
$(2)$ from Construction \ref{s'_1}.
The  assertion (3'') follows from Lemma~\ref{Z'_1_and_UtimesS'} and Construction~\ref{t_0_and_s_Delta_and_t'0}(2').
\end{proof}

\begin{defs}
\label{h_theta_defined} Set,
$$h_{\theta}=\frac{((1-\theta)t_0+ \theta t_1)|_{\bb A^1\times U\times X}}{(s^{\otimes dn}_E)|_{\bb A^1\times U\times X}} \in k[\bb A^1\times U\times_B X],$$
where $s_E \in \Gamma(U\times_B \bar X, \cc L(E))$ is the canonical section.
\end{defs}

\begin{proof}[Proof of Proposition \ref{h_theta}]
Let $prj: \bb A^1 \times U\times_B \bar X \to U\times_B \bar X$
be the projection. Consider two sections $(1-\theta)t_0+ \theta t_1$
and $s_E^{\otimes dn}$ of the line bundle
$prj^*(\cc L(dnE))$
on $\bb A^1 \times U\times_B \bar X$.
By Lemma~\ref{t_0andt_1_for_UtimesbarX}
these two sections have no common zeros. Thus one has a morphism
$$[pr,(1-\theta)t_0+ \theta t_1: s_E^{\otimes dn}]: \bb A^1\times U\times_B \bar X \to \bb A^1\times U\times \mathbb P^1,$$
where $pr: \bb A^1\times U'\times_B \bar X'\to \bb A^1\times U'$
is the projection. This morphism is quasi-finite and projective. Hence it is
finite and surjective. It follows that any of its base changes is finite and surjective.
Particularly, the morphism $(pr,h_{\theta}): \bb A^1\times U\times X \to \bb
A^1\times U\times \bb A^1$ is finite and surjective,
because the closed subset $\{s_E^{\otimes dn}=0\}$ in
$\bb A^1\times U\times_B \bar X$ coincides with the one
$\bb A^1\times E$. This proves the assertion (a) of Proposition \ref{h_theta}.

Lemma \ref{t_0andt_1_for_UtimesbarX} yields the assertion (d) of Proposition~\ref{h_theta}.
The property (1b) from Construction~\ref{s'_1} and Lemma \ref{1a_1a'_1b_yield_1c}
yield the assertion (c) of Proposition~\ref{h_theta}. Lemma \ref{Z'_1_and_UtimesS'}
and the property (2') from Construction \ref{t_0_and_s_Delta_and_t'0}
yield the assertion (b) of Proposition~\ref{h_theta}. Proposition~\ref{h_theta} follows.
\end{proof}

\section{Nisnevich cohomology with coeffitients in a $\bb ZF_*$-sheaf}

\begin{lem}\label{Abelian}
The category of Nisnevich sheaves of Abelian groups with $\bb
ZF_*$-transfers is a Gro\-then\-dieck category.
\end{lem}

\begin{proof}
The category of presheaves of Abelian groups with $\bb
ZF_*$-transfers is plainly a Grothendieck category. Note that it can
be identified with the category of framed raditive presheaves of
Abelian groups. Since for every radditive framed presheaf of Abelian
groups $F$ the associated sheaf in the Nisnevich topology has a
unique structure of a framed presheaf such that the map $F\to
F_{\nis}$ is a map of framed presheaves by~\cite[4.5]{Voe2}, our
lemma is now proved similar to~\cite[6.4]{GP}.
\end{proof}

The main purpose of this section is to prove the following

\begin{prop}
\label{Ext_and_Cohom_1}
For any Nisnevich sheaf $\cc F$ with $\bb ZF_*$-transfers, any integer $n$ and any $k$-smooth variety $X$,
there is a natural isomorphism
$$H^n_{Nis}(X,\cc F)=\Ext^n(\bb ZF_*(X), \cc F),$$
where the $\Ext$-groups are taken in the Grothendieck category of
Nisnevich sheaves with $\bb ZF_*$-transfers.
\end{prop}

Recall that for a morphism $f:Y\to X$ we denote by $\check C(f)$ or
$\check C(Y)$ the Cech simplicial object defined by $f$.

\begin{lem}[\cite{Voe2}, Thm. 4.4]\label{cechstrelka}
Let $f:Y\to X$ be an etale (respectively Nisnevich) covering of a
scheme $X$. Then for any $n$ the map of simplicial presheaves
   $$Fr_n(-,\check C(Y))\to Fr_n(-,X)$$
is a local equivalence in the etale (respectively Nisnevich)
topology.
\end{lem}

\begin{defs}\label{F_m_U_Y/(Y-S)}
Define $F_m(U,Y)\subset Fr_m(U,Y)$ as a subset consisting of
$(Z,W,\phi;g)\in Fr_m(U,Y)$ such that $Z$ is connected.

Clearly, the set $F_m(U,Y)-\emptyset_m$ is a free basis of the abelian group
$\bb ZF_m(U,Y)$. However, the assignment $U\mapsto F_m(U,Y)$
is not a presheaf even on the category $Sm/k$.
\end{defs}

Applying the proof of \cite[Thm. 4.4]{Voe2} one can conclude that
the following lemma holds:

\begin{lem}\label{cechstrelka_F_n}
Let $f:Y\to X$ be an etale (respectively Nisnevich) covering of a scheme $X$.
Then for any local essentially $k$-smooth henzelian scheme $U$ and
for any integer $n\geq 0$ the map of pointed simplicial sets
   $$F_n(U,\check C(Y))\to F_n(U,X)$$
is a weak equivalence.
\end{lem}

\begin{cor}\label{cechstrelka_ZF_n}
Let $f:Y\to X$ be an etale (respectively Nisnevich) covering of a scheme $X$. Then for any $n$ the maps of simplicial presheaves
   $$\bb ZF_n(-,\check C(Y))\to \bb ZF_n(-,X), \quad \bb ZF_*(-,\check C(Y))\to \bb ZF_*(-,X)$$
are local equivalences in the Nisnevich topology.
\end{cor}

\begin{cor}\label{ZF_injective_are_acyclic}
Let $I$ be an injective Nisnevich sheaf with $\bb ZF_*$-transfers.
Then for any $k$-smooth variety $X$ one has $H^i_{Nis}(X,I)=0$ for
all $i>0$.
\end{cor}

\begin{proof}
Using the preceding corollary, our proof is similar to that
of~\cite[1.7]{SV1}.
\end{proof}

\begin{proof}[Proof of Proposition \ref{Ext_and_Cohom_1}]
Corollary~\ref{ZF_injective_are_acyclic} implies the proposition.
\end{proof}

Proposition~\ref{Ext_and_Cohom_1} implies the following useful

\begin{cor}
\label{H_and_transfers}
For any Nisnevich sheaf $\cc F$ with $\bb ZF_*$-transfers and any integer $n$,
the presheaf $X\mapsto H^n_{Nis}(X,\cc F)$
has a canonical structure of a $\bb ZF_*$-presheaf.
\end{cor}

In fact, this holds for the presheaf $X\mapsto \Ext^n(\bb ZF_*(X),
\cc F)$.

\section{Homotopy invariance of cohomology presheaves}

In this section we prove Theorems \ref{H1_hom_inv} and \ref{Hn_hom_inv}. They
complete the proof of Theorem~\ref{thm1}, which is the main result of the paper.
Each statement in this section except Lemma~\ref{H1_and_H0}
is split in two parts when the characteristic of the base field does not equal 2 and
does equal 2. We will prove the case when the characteristic is not 2, because
the case when $\chr k=2$ is proved similarly and is left to the reader.
Recall that the base field $k$ is supposed to be infinite and perfect in this section.

\begin{defs}\label{F_minus_1}
Let $\cc G$ be a homotopy invariant presheaf of abelian groups with $\bb ZF_*$-transfers.
Then the presheaf $X\mapsto \cc G_{-1}(X):=\cc G(X\times (\bb A^1-0))/\cc G(X)$ is also
a homotopy invariant presheaf of abelian groups with $\bb ZF_*$-transfers.
If the presheaf $\cc G$ is a Nisnevich sheaf, then the presheaf $\cc F_{-1}$
is also a Nisnevich sheaf.
If the presheaf $\cc G$ is quasi-stable, then so is the presheaf  $\cc G_{-1}$.

If the presheaf $\cc G$ is a Nisnevich sheaf on $Sm/k$ and $Y$ is a $k$-smooth variety, then denote by
$\cc G|_Y$ the restriction of $\cc G$ to the small Nisnevich site of $Y$.
\end{defs}

\begin{lem}
\label{H1_and_H0}
For any $\bb A^1$-invariant quasi-stable
$\bb ZF_*$-sheaf of abelian groups $\cc F$,
any $k$-smooth variety $Y$ and any $k$-smooth divisor $D$ in $Y$
the canonical morphism
$$H^1_D(Y,\cc F)\to H^0_{Nis}(Y,{\cc H}^1_D(Y,\cc F))$$
is an isomorphism.
\end{lem}

\begin{proof}
The local-global spectral sequence yields an exact sequence of the form
$$H^1_{Nis}(Y,{\cc H}^0_D(Y,\cc F)) \to H^1_D(Y,\cc F) \to H^0_{Nis}(Y,{\cc H}^1_D(Y,\cc F)) \to H^2_{Nis}(Y,{\cc H}^0_D(Y,\cc F)).$$
By Theorem~\ref{Few_On_P_w_Tr}(3') the sheaf ${\cc H}^0_D(Y,\cc F))$ vanishes.
\end{proof}

\begin{lem}\label{F_1(D)}
Let $X$ be an essentially $k$-smooth local henzelian scheme and let $D\subset X$ be a $k$-smooth divisor.
Let $i: D \hookrightarrow X$ be the closed embedding. Then
for any $\bb A^1$-invariant quasi-stable
$\bb ZF_*$-sheaf of abelian groups $\cc F$ the Nisnevich sheaves
$i_*(\cc F_{-1}|_D)$ and ${\cc H}^1_D(X,\cc F))$  on the small
Nisnevich site of $X$ are isomorphic if $\chr k\ne 2$.
If $\chr k=2$ then the same statement holds under
the additional assumption that the $\bb ZF_*$-sheaf $\cc F$ is a sheaf of
$\bb Z[1/2]$-modules.
\end{lem}

\begin{proof}
The group $H^1_D(X,\cc F))$ is isomorphic to
$\cc F(X-D)/Im(\cc F(X))=\cc F(X-D)/\cc F(X)$.
The latter equality makes sense by Theorem
\ref{Few_On_P_w_Tr}(item (3')).
Since $X$ is essentially $k$-smooth and henselian,
and $D$ is essentially $k$-smooth, then there is
a morphism $r: X\to D$ such that the composite map
$D\xrightarrow{i} X\xrightarrow{r} D$
is the identity.
Let $x\in X$ be the closed point. Clearly, $x\in D$.
Set $V:=Spec(\cc O_{D\times \bb A^1,(x,0)})$.
Let $f\in k[X]$ be a function
defining the smooth divisor $D$.
Then the morphism
$$(r,f): X\to D\times \bb A^1$$
takes values in $V$.
We keep the same notation for the corresponding morphism $(r,f): X \to V$.
Note that $(r,f)^{-1}(D\times 0)=D$. Thus the morphism $(r,f)$
induces a homomorphism
$$[[(r,f)]]^*: \cc F(V-D\times 0)/\cc F(V) \to \cc F(X-D)/\cc F(X).$$

{\it We claim that it is an isomorphism}.
To prove this claim note that the morphism $(r,f)$
induces a scheme isomorphism
$X^h_D \to V^h_{D\times 0}$,
where $X^h_D$ is the henselization of $X$ at $D$
and $V^h_{D\times 0}$
is the henselization of $V$ at $D\times 0$.
Now Theorem
\ref{Few_On_P_w_Tr}(item (5))
{\it implies the claim}.

By Corollary~\ref{Exc_On_Rel_Aff_Line_3}
the pull-back map
   $$\cc F_{-1}(D)=\cc F(D\times (\bb A^1-0))/\cc F(D\times \bb A^1) \to \cc F(V-D\times 0)/\cc F(V)$$
is an isomorphism, too.
Thus there is a natural isomorphism
$$
\cc F_{-1}(D)=\cc F(D\times (\bb A^1-0))/\cc F(D\times \bb A^1) \xrightarrow{[[can]]^*} \cc F(V-D)/\cc F(V) \xrightarrow{[[(r,f)]]^*} \cc F(X-D)/\cc F(X)=
$$
$$\cc F(X-D)/Im(\cc F(X))
\xrightarrow{\partial} H^1_D(X,\cc F))
$$
leading to an isomorphism
of Nisnevich sheaves $i_*(\cc F_{-1}|_D)\cong {\cc H}^1_D(X,\cc F)$ on the small Nisnevich site of $X$.
\end{proof}

\begin{lem}\label{F_1(D_times_A1)}
Let $X$ be an essentially $k$-smooth local henzelian scheme and let $D\subset X$ be a $k$-smooth divisor.
Let $I: D\times \bb A^1 \hookrightarrow X\times \bb A^1$ be the closed embedding. Then
for any $\bb A^1$-invariant quasi-stable
$\bb ZF_*$-sheaf of abelian groups $\cc F$
the two Nisnevich sheaves $I_*(\cc F_{-1}|_{D\times \bb A^1})$
and ${\cc H}^1_{D\times \bb A^1}(X\times \bb A^1,\cc F))$
on small Nisnevich site of $X\times \bb A^1$ are isomorphic
if $\chr k\ne 2$. If $\chr k=2$ then the same statement holds under
the additional assumption that the $\bb ZF_*$-sheaf $\cc F$ is a sheaf of
$\bb Z[1/2]$-modules.
\end{lem}

\begin{proof}
The proof is similar to that of Lemma~\ref{F_1(D)}.
\end{proof}

\begin{cor}
\label{H1_D_and_H1_D_times_A1}
Suppose $\chr k\ne 2$. Then the pull-back map
$p^*_X: H^1_D(X,\cc F)\to H^1_{D\times \bb A^1}(X\times \bb A^1,\cc F)$
is an isomorphism, where $p_X: X\times \bb A^1 \to X$ is the projection.
If $\chr k=2$, then the same statement holds under
the additional assumption that the $\bb ZF_*$-sheaf $\cc F$ is a sheaf of
$\bb Z[1/2]$-modules.
\end{cor}

\begin{proof}
One can check that the following diagram commutes
$$\xymatrix{\cc F_{-1}(D\times \bb A^1) \ar^(.4){\beta}[rr]&& H^0_{Nis}(X\times \bb A^1,{\cc H}^1_{D\times \bb A^1}(X\times \bb A^1,\cc F)) &&
H^1_{D\times \bb A^1}(X\times \bb A^1,\cc F)\ar_(.4){\psi}[ll]\\
               \cc F_{-1}(D) \ar^{\alpha}[rr]\ar^{p^*_D}[u]&& H^0_{Nis}(X,{\cc H}^1_{D}(X,\cc F))  \ar^{p^*_X}[u]&& H^1_{D}(X,\cc F) \ar_{\phi}[ll]\ar^{p^*_X}[u]}$$
where the maps $\alpha$ and $\beta$ are isomorphisms of
Lemmas \ref{F_1(D)} and \ref{F_1(D_times_A1)} respectively,
the maps $\phi$ and $\psi$ are isomorphisms of Lemma
\ref{H1_and_H0}, the vertical maps are the pull-back maps.
The sheaf $\cc F_{-1}$ is homotopy invariant, because so is the sheaf $\cc F$.
It follows that the map $p^*_D$ is an isomorphism,
whence the corollary.
\end{proof}

\begin{cor}\label{H0_and_H1_with_support}
Suppose $\chr k\ne 2$. Under the the hypotheses of Lemma~\ref{F_1(D_times_A1)}
the boundary map
$$\partial: \cc F((X-D)\times \bb A^1) \to H^1_{D\times \bb A^1}(X\times \bb A^1,\cc F)$$
is surjective. If $\chr k=2$, then the same statement is true if we assume
that the $\bb ZF_*$-sheaf $\cc F$ is a sheaf of $\bb Z[1/2]$-modules.
\end{cor}

\begin{proof}
By Corollary~\ref{H1_D_and_H1_D_times_A1} the pull-back map
$p^*_X: H^1_D(X,\cc F)\to H^1_{D\times \bb A^1}(X\times \bb A^1,\cc F)$
is an isomorphism. The boundary map
$\partial: \cc F(X-D) \to H^1_D(X,\cc F)$ is surjective, because
$X$ is local Henselian, whence the corollary.
\end{proof}

\begin{prop}
\label{Injectivity_on_H1}
Suppose $\chr k\ne 2$. Under the the hypotheses of Lemma~\ref{F_1(D_times_A1)} the map
$$H^1_{Nis}(X\times \bb A^1,\cc F)) \to H^1_{Nis}((X-D)\times \bb A^1,\cc F))$$
is injective. If $\chr k=2$, then the same statement is true if we assume
that the $\bb ZF_*$-sheaf $\cc F$ is a sheaf of $\bb Z[1/2]$-modules.
\end{prop}

\begin{proof}
This follows from Corollary~\ref{H0_and_H1_with_support}.
\end{proof}

\begin{prop}
\label{H1_of_Aff_line}
Suppose $\chr k\ne 2$. Let $\cc F$ be an $\bb A^1$-invariant quasi-stable
$\bb ZF_*$-sheaf of abelian groups. Then
   $$H^1_{Nis}(\bb A^1_K, \cc F)=0.$$
If $\chr k=2$, then the same statement is true if we assume
that the $\bb ZF_*$-sheaf $\cc F$ is a sheaf of $\bb Z[1/2]$-modules.
\end{prop}

\begin{proof}
Let $a\in H^1_{Nis}(\bb A^1_K, \cc F)$. We want to prove that $a=0$.
The Nisnevich topology is trivial at the generic point of the affine line $\bb A^1_K$.
Therefore there is a Zariski open subset $U$ in $\bb A^1_K$ such that the restriction
of $a$ to $U$ vanishes. Let $Z$ be the complement of $U$ in $\bb A^1_K$ regarded
as a closed subscheme with the reduced structure (it consists of finitely many closed points).
Let $V:=\sqcup_{z\in Z} (\bb A^1)^h_z$, where each summand is the henselization of the affine line
at $z\in Z$. Then the cartesian square
$$\xymatrix{V-Z\ar[r]\ar[d]&V\ar^{\Pi}[d]\\
               U\ar[r]&\bb A^1_K}$$
gives rise to a long exact sequence
$$\cc F(U) \oplus \cc F(V) \to \cc F(V-Z) \xrightarrow{\partial} H^1_{Nis}(\bb A^1_K,\cc F) \to H^1_{Nis}(U,\cc F) \oplus H^1_{Nis}(V,\cc F).$$
The left arrow is surjective by Theorem~\ref{Few_On_P_w_Tr}(items (2), (5)).
The group $H^1_{Nis}(V,\cc F)$ vanishes by the choice of $V$.
Thus the map $H^1_{Nis}(\bb A^1_K,\cc F) \to H^1_{Nis}(U,\cc F)$
is injective, and hence $a=0$.
\end{proof}

\begin{prop}\label{H1_vanishes_1}
Suppose the base field $k$ is infinite and perfect with $\chr k\ne 2$.
Let $\cc F$ be an $\bb A^1$-invariant quasi-stable $\bb ZF_*$-sheaf of Abelian groups.
Let $Y$ be a $k$-smooth variety and let
$a\in H^1_{Nis}(Y\times \bb A^1, \cc F)$
be an element such that its restriction to $X\times \{0\}$ vanishes. Then $a=0$.
If $\chr k=2$, then the same statement is true if we assume
that the $\bb ZF_*$-sheaf $\cc F$ is a sheaf of $\bb Z[1/2]$-modules.
\end{prop}

\begin{proof}
The exact sequence
$$0 \to H^1_{Nis}(Y,p_*(\cc F)) \xrightarrow{\alpha} H^1_{Nis}(Y\times \bb A^1) \xrightarrow{\beta}  H^0_{Nis}(Y, R^1p_*(\cc F))$$
and the fact that the sheaf $\cc F$ is  homotopy invariant show that for
$$A:=Ker[i^*_0: H^1_{Nis}(Y\times \bb A^1, \cc F) \to H^1_{Nis}(Y, \cc F)]$$
the map $\beta|_A: A \to H^0_{Nis}(Y, R^1p_*(\cc F))$
is injective. The stalk of the sheaf $R^1p_*(\cc F)$ at a point $y \in Y$
is $H^1_{Nis}(Y^h_y \times \bb A^1, \cc F)$,
where $Y^h_y=Spec(\cc O^h_{Y,y})$ is the henzelization of the local scheme $Spec(\cc O_{Y,y})$.
By Proposition
\ref{H1_of_Aff_line}
there is a closed subset $Z$ in $Y$ such that $\beta(a)|_{Y-Z}=0$.
Since the field $k$ is perfect, there is a proper closed subset $Z_1\subset Z$
such that $Z-Z_1$ is $k$-smooth. Then
$Z-Z_1$ is a $k$-smooth closed subvariety in $Y-Z_1$.

We claim that $a_1:=a|_{(Y-Z_1)\times \bb A^1}=0$. In fact, $a_1|_{(Y-Z_1)\times 0}=0$.
Thus it suffices to check that all stalks of the element $\beta(a_1)$ vanish.
Let $y \in Y-Z_1$ be a point. If $y \in Y-Z$ then $\beta(a_1)_y=0$, because $\beta(a_1)|_{Y-Z}=0$.
If $y\in Z-Z_1$ then shrinking $Y-Z_1$ around $y$ we may assume that
there is a $k$-smooth divisor $D$ in $Y-Z_1$ containing $Z-Z_1$. In this case
$a_1|_{Y-D}=0$. Now Proposition
\ref{Injectivity_on_H1}
shows that $\beta(a_1)_y=0$.
We have proved that $a_1=0$.

Now there is a proper closed subset $Z_2\subset Z_1$
such that $Z_1-Z_2$ is $k$-smooth. Then
$Z_1-Z_2$ is a $k$-smooth closed subvariety in $Y-Z_2$.
Arguing just as above, we conclude that $a_2:=a|_{(Y-Z_2)\times \bb A^1}=0$.
Continuing this process finitely many times, we conclude that $a=0$.
\end{proof}

\begin{thm}\label{H1_hom_inv}
Suppose the base field $k$ is infinite and perfect with $\chr k\ne 2$.
If $\cc F$ is an $\bb A^1$-invariant quasi-stable $\bb ZF_*$-sheaf of abelian groups, then
the $\bb ZF_*$-presheaf of abelian groups $X\mapsto H^1_{Nis}(X,\cc F)$
is $\bb A^1$-invariant and quasi-stable. If $\chr k=2$, then the same statement is true if we assume
that the $\bb ZF_*$-sheaf $\cc F$ is a sheaf of $\bb Z[1/2]$-modules.
\end{thm}

\begin{proof}
By Corollary
\ref{H_and_transfers}
the presheaf
$X\mapsto H^1_{Nis}(X,\cc F)$
has a canonical structure of a
$\bb ZF_*$-presheaf.
Let $X$ be a $k$-smooth variety.
Let $\sigma_X \in Fr_1(X,X)$ be the distinguished morphism of level one.
The assignement
$X\mapsto (\sigma^*_X: \cc F(X) \to \cc F(X))$
is an endomorphism of the Nisnevich sheaf
$\cc F|_{Sm/k}$. Thus for each $n$ it induces an endomorphism
of the cohomology presheaf
$\sigma^*: H^n(-,\cc F) \to H^n(-,\cc F)$.
Since $\sigma^*$ acts on $\cc F$ as an isomorphism, it acts
as an isomorphism on the presheaf $H^n(-,\cc F)$. We see that the $\bb ZF_*$-presheaf
$H^n(-,\cc F)$ is quasi-stable.

To show that the presheaf $X\mapsto H^1_{Nis}(X,\cc F)$
is $\bb A^1$-invariant, note that
the pull-back map
$i^*_0: H^1_{Nis}(X\times \bb A^1,\cc F) \to H^1_{Nis}(X,\cc F)$
is surjective. It is also injective by Proposition
\ref{H1_vanishes_1}. Our theorem now follows.
\end{proof}

\begin{thm}\label{Hn_hom_inv}
Suppose the base field $k$ is infinite and perfect with $\chr k\ne 2$.
Let $\cc F$ be an $\bb A^1$-invariant quasi-stable
$\bb ZF_*$-sheaf of abelian groups.
Then for any integer $n \geq 2$, the presheaf $X\mapsto H^n_{Nis}(X,\cc F)$
is an $\bb A^1$-invariant and quasi-stable
$\bb ZF_*$-presheaf of abelian groups.
If $\chr k=2$, then the same statement is true if we assume
that the $\bb ZF_*$-sheaf $\cc F$ is a sheaf of $\bb Z[1/2]$-modules.
\end{thm}

\begin{proof}
We can apply the same arguments as in the proof of Theorem
\ref{H1_hom_inv} to show that
the presheaf $X\mapsto H^n_{Nis}(X,\cc F)$ is a $\bb ZF_*$-presheaf of abelian groups,
which is, moreover, quasi-stable.

It remains to check that the presheaf is homotopy invariant. We may assume till the end of
the proof that each presheaf $X\mapsto H^j_{Nis}(X,\cc F)$ with $j<n$ is homotopy invariant.

In order to complete the proof of the theorem, we shall need the following lemma.


\begin{lem}
\label{Vanishing3}
Let $X$ be an essentially $k$-smooth local henzelian scheme and let $D\subset X$ be a $k$-smooth divisor.
Let $I: D\times \bb A^1 \hookrightarrow X\times \bb A^1$ be the closed embedding. Suppose $\chr k\ne 2$. Then
for any $\bb A^1$-invariant quasi-stable
$\bb ZF_*$-sheaf of abelian groups $\cc F$ and any $n\geq 2$ one has
   $$H^n_{D\times \bb A^1}(X\times \bb A^1,\cc F)=0.$$
If $\chr k=2$, then the same statement is true if we assume that the
$\bb ZF_*$-sheaf $\cc F$ is a sheaf of $\bb Z[1/2]$-modules.
\end{lem}

\begin{proof}
Applying the local-global spectral sequence, it is sufficient to check that the groups
$H^i_{Nis}(X\times \bb A^1,\cc H^j_{D\times \bb A^1}(X\times \bb A^1,\cc F))$ vanish
for $i+j=n$ (here $i,j \geq 0$). Let $j=1$. By Lemma~\ref{F_1(D_times_A1)} one has
$\cc H^1_{D\times \bb A^1}(X\times \bb A^1,\cc F))=I_*(\cc F_{-1})$.
Thus,
$$H^{n-1}_{Nis}(X\times \bb A^1, \cc H^1_{D\times \bb A^1}(X\times \bb A^1,\cc F)))=H^{n-1}_{Nis}(X\times \bb A^1,I_*(\cc F_{-1}))=$$
$$=H^{n-1}_{Nis}(D\times \bb A^1,\cc F_{-1})=H^{n-1}_{Nis}(D,\cc F_{-1})=0.$$
The latter equality holds since $D$ is local henzelian.
If $j=0$, then the sheaf
$\cc H^j_{D\times \bb A^1}(X\times \bb A^1,\cc F)$
vanishes.
Obviously, the stalk of this sheaf vanishes at every point $z\in (X-D)\times \bb A^1$.
If $z\in D\times \bb A^1$ then
the stalk vanishes by Theorem
\ref{Few_On_P_w_Tr} (item (3')). Let $2\leq j\leq n$.
Then
the sheaf
$\cc H^j_{D\times \bb A^1}(X\times \bb A^1,\cc F))$
vanishes. Obviously, the stalk of this sheaf vanishes at every point $z\in (X-D)\times \bb A^1$.
If $z\in D\times \bb A^1$ then
the stalk of the sheaf at the point $z$ is equal to the group
$H^j_{\cc D}(\cc X,\cc F)$,
where $\cc X$ is the henselzation of $X\times \bb A^1$ at $z$
and $\cc D$ is the henselzation of $D\times \bb A^1$ at $z$.
Since $H^j_{Nis}(\cc X,\cc F)=0$ one has equalities
$$H^j_{\cc D}(\cc X,\cc F)=H^{j-1}_{Nis}(\cc X-\cc D,\cc F)/Im[H^{j-1}_{Nis}(\cc X,\cc F)]=H^{j-1}_{Nis}(\cc X-\cc D,\cc F)/H^{j-1}_{Nis}(\cc X,\cc F).$$
The latter equality makes sense by Theorem
\ref{Few_On_P_w_Tr}(item (3')).

Now, applying Theorem
\ref{Few_On_P_w_Tr}(item (5))
and
Corollary \ref{Exc_On_Rel_Aff_Line_3}
to the presheaf
$H^{j-1}_{Nis}(- ,\cc F)$
and
arguing as in the proof of Lemma \ref{F_1(D)}
we get an isomorphism
$$H^{j-1}_{Nis}(\cc X-\cc D,\cc F)/H^{j-1}_{Nis}(\cc X,\cc F)\cong H^{j-1}_{Nis}(\cc D \times \bb G_m,\cc F)/H^{j-1}_{Nis}(\cc D \times \bb A^1,\cc F).$$
It suffices to check that
$H^{j-1}_{Nis}(\cc D \times \bb G_m,\cc F)=0$.
The presheaf $H^{j-1}_{Nis}(-  \times \bb G_m,\cc F)$
is a homotopy invariant  $\bb ZF_*$-presheaf, which is also quasi-stable.
By Theorem
\ref{Few_On_P_w_Tr}(item (3'))
the map
$$H^{j-1}_{Nis}(\cc D \times \bb G_m,\cc F)\to H^{j-1}_{Nis}(Spec (k(D)) \times \bb G_m,\cc F)=H^{j-1}_{Nis}(\bb G_{m,k(D)},\cc F)$$
is injective. By Theorem \ref{InjOnAffLine} the map
$H^{j-1}_{Nis}(\bb G_{m,k(D)},\cc F)\to
H^{j-1}_{Nis}(Spec(k(D)(t)),\cc F)$ is injective. The latter group
vanishes, because the Nisnevich topology on $Spec(K)$ is trivial,
where $K$ is a finitely generated field over $k$. Thus
$H^{j-1}_{Nis}(\cc D \times \bb G_m,\cc F)=0$ and $H^j_{\cc D}(\cc
X,\cc F)=0$, too.
\end{proof}

{\it Let us return to the proof of Theorem~\ref{Hn_hom_inv}}.
Under the hypotheses of Lemma
\ref{F_1(D_times_A1)}, the preceding lemma implies
the map
\begin{equation}
\label{Injectivity_Hn}
H^n_{Nis}(X\times \bb A^1,\cc F) \to H^n_{Nis}((X-D)\times \bb A^1,\cc F)
\end{equation}
is injective.

Next, we claim that for a $k$-smooth variety $Y$ and the projection
$p: Y \times \bb A^1 \to Y$
the Nisnevich sheaves $R^jp_*(\cc F)$
vanish for $j=1,...,n-1$.
In fact, such a sheaf is associated
with the presheaf
$U\mapsto H^j_{Nis}(U\times \bb A^1, \cc F)$.
The presheaf
$H^j_{Nis}(U, \cc F)$
is homotopy invariant.
Thus
$H^j_{Nis}(U\times \bb A^1, \cc F)=H^j_{Nis}(U, \cc F)$.
Since $j\geq 1$
the associated Nisnevich sheaf vanishes.

Since the Nisnevich sheaves $R^jp_*(\cc F)$
vanish for $j=1,...,n-1$, one has an exact sequence
$$0 \to H^n_{Nis}(Y,p_*(\cc F)) \xrightarrow{\alpha} H^n_{Nis}(Y\times \bb A^1) \xrightarrow{\beta}  H^0_{Nis}(Y, R^np_*(\cc F)).$$
Since the sheaf $\cc F$ is homotopy invariant, then for
$$A:=Ker[i^*_0: H^n_{Nis}(Y\times \bb A^1, \cc F) \to H^1_{Nis}(Y, \cc F)]$$
the map $\beta|_A: A \to H^0_{Nis}(Y, R^np_*(\cc F))$
is injective. Arguing as in the proof of Proposition
\ref{H1_vanishes_1}
and using the fact that map
(\ref{Injectivity_Hn}) is injective,
we get the following

\begin{lem}\label{Hn_vanishes_1}
Suppose the base field $k$ is infinite and perfect with $\chr k\ne 2$.
Let $\cc F$ be an $\bb A^1$-invariant quasi-stable $\bb ZF_*$-sheaf of Abelian groups.
Let $Y$ be a $k$-smooth variety and let
$a\in H^n_{Nis}(Y\times \bb A^1, \cc F)$
be an element such that its restriction to $Y\times \{0\}$ vanishes.
Then $a=0$. If $\chr k=2$, then the same statement is true if we assume
that the $\bb ZF_*$-sheaf $\cc F$ is a sheaf of $\bb Z[1/2]$-modules.
\end{lem}

The pull-back map
$i^*_0: H^n_{Nis}(Y\times \bb A^1, \cc F)\to H^n_{Nis}(Y, \cc F)$
is surjective by functoriality. By Lemma
\ref{Hn_vanishes_1}
it is also injective. This completes the proof of Theorem~\ref{Hn_hom_inv}.
\end{proof}

\end{document}